%% file: quantum.tex
\documentclass[margincite, leqno]{sl2art}

\usepackage{mathtools}

\usepackage{todonotes}

\usepackage{mathrsfs}

\usepackage{tikz,tikz-cd}
\usetikzlibrary{arrows.meta}
\tikzset{>={Stealth}}

\makeatletter
\def\qed@warning{}
\makeatother

\numberwithin{equation}{section}

\usepackage[export]{adjustbox}
\usepackage{graphicx,subcaption}
\graphicspath{{fig/}}

\usepackage{enumitem}

\newlist{thmenum}{enumerate}{1}
\setlist[thmenum]{label=(\alph*)}

\input{macros-quantum.tex}

\NewDocumentCommand{\brokenref}{m}{??}

\declaretheorem[style=bigtheorem,title=Theorem,refname={Theorem, Theorems}]{bigtheorem}
\declaretheorem[style=bigtheorem,title=Conjecture,refname={Conjecture, Conjectures}]{conjecture}
\declaretheorem[style=theorem]{proposition}
\declaretheorem[style=theorem,sibling=proposition]{theorem}
\declaretheorem[style=theorem,sibling=proposition]{lemma}

\declaretheorem[style=theorem,sibling=proposition]{corollary}
\declaretheorem[style=definition,sibling=proposition]{definition}
\declaretheorem[style=definition,sibling=proposition]{remark}
\declaretheorem[style=definition,sibling=proposition]{example}

\usepackage[noabbrev,capitalize]{cleveref}
\crefname{equation}{equation}{equations}
\Crefname{equation}{Equation}{Equations}
\crefname{bigtheorem}{Theorem}{Theorems}
\Crefname{bigtheorem}{Theorem}{Theorems}
\crefname{conjecture}{Conjecture}{Conjectures}
\Crefname{conjecture}{Conjecture}{Conjectures}

\title{A quantization of the \(\operatorname{SL}_2(\mathbb{C})\) Chern-Simons invariant of tangle exteriors}
\author{Calvin McPhail-Snyder}
\address{Department of Mathematics, Duke University}
\email{calvin@sl2.site}

\begin{document}

\begin{abstract}
  We define a sequence of invariants \(\mathcal{Z}_{\nr}^{\psi}\) of tangles with flat \(\mathfrak{sl}_{2}\) connections (i.e.\ hyperbolic structures) on their complements.
  These can be interpreted as a geometric twist of the Kashaev invariant or as a quantization of the \(\operatorname{SL}_{2}(\mathbb{C})\) Chern-Simons invariant.
  To support the second interpretation we give a new description \(\mathcal{I}^{\psi}\) of the Chern-Simons invariant of a tangle exterior.
  \(\mathcal{Z}_{\nr}^{\psi}\) directly recovers \(\mathcal{I}^{\psi}\) when \(\nr = 1\).
  We build \(\mathcal{Z}_{\nr}^{\psi}\) using modules over unrestricted quantum \(\mathfrak{sl}_{2}\) at a root of unity and the holonomy \(R\)-matrices previously constructed by the author and Reshetikhin (\href{https://arxiv.org/abs/2509.02354}{arXiv:2509.02354}).
  Unlike most previous constructions of geometric quantum invariants \(\mathcal{Z}_{\nr}^{\psi}\) is defined without any phase ambiguity.
  It is natural to conjecture that \(\mathcal{Z}_{\nr}^{\psi}\) is related to the quantization of Chern-Simons theory with complex, noncompact gauge group \(\operatorname{SL}_{2}(\mathbb{C})\) and we discuss how to interpret our results in this context.
\end{abstract}


\maketitle

\tableofcontents
\section{Introduction}

Let \(L\) be a link in \(S^3\).
In this paper we study geometric invariants of \(L\) that depend on a choice of representation  \(\rho : \pi_1(S^3 \setminus L) \to \slg\).
Typically we are interested in \(\rho\) up to conjugation, i.e.\ up to \defemph{gauge transformation}.
Because \(\slg\) is the double cover of the isometry group \(\pslg\) of hyperbolic \(3\)-space we can think of \(\rho\) as a generalized hyperbolic structure and gauge-equivalent \(\rho\) represent the same structure.

Our goal is to define a new family \(\qfunc{}\) of geometric link invariants that we interpret as a quantization of the \(\slg\) Chern-Simons invariant of the exterior of \(L\).
To explain what this means we must first say exactly what we mean by ``Chern-Simons invariant".
For closed manifolds it is a complex number, while for manifolds with boundary it is a section of a line bundle.

\subsection{The classical \texorpdfstring{\(\slg\)}{SL₂(ℂ)} Chern-Simons invariant}
\label{sec:intro Chern-Simons}

First consider a \emph{closed} \(3\)-manifold \(M\) and a representation \(\rho : \pi_1(M) \to \slg\).
Choose a flat \(\mathfrak{sl}_2(\mathbb{C})\) connection \(\theta\) with holonomy \(\rho\).
Denote the integral of the Chern-Simons form of \(\theta\) by
\[
  S(M, \theta)
  \defeq
  -\frac{\ii}{4} \int_{M} \tr [\theta, d\theta] + \frac{2}{3} \tr [\theta, \theta \wedge \theta] 
\]
It is independent of the gauge class of \(\theta\) modulo \(4\ii\pi^2\) and thus defines an invariant \(S(M, \rho) \in \mathbb{C}/4\ii\pi^2 \mathbb{Z}\).
When \(\rho\) is sufficiently nontrivial (say, it acts irreducibly on \(\mathbb{C}^2\)) it induces a metric \(g\) of constant negative curvature on \(M\).
It can be shown \cite[Section 3]{Yoshida1985} that 
\[
  S(M, \rho)
  = 
  \operatorname{Vol}(M, g) + i \operatorname{CS}(M, g)
\]
has real part the volume of \(g\) and imaginary part the Chern-Simons invariant of the \(\operatorname{SO}(3)\) frame-field of \(g\).
As such, it is natural to think of \(S\) as a complexification of the hyperbolic volume.
For this reason it is frequently called the \defemph{complex volume}.
\(S(M,\rho)\) can be computed in terms of an ideal triangulation of \(M\) by a sum of dilogarithms and this definition makes sense for any \(\rho\).

For a knot%
\note{
  This discussion works just as well for links or for general \(3\)-manifolds with torus boundary but the notation becomes more cumbersome.
}
\(K\) in \(S^3\) the \defemph{knot exterior} \(\extr{K} = S^3 \setminus \nu(K)\) is the complement of an open regular neighborhood of \(K\).
It is a \(3\)-manifold with one torus boundary component.
The Chern-Simons integral \(S(\extr{K}, \theta)\) depends not just on the holonomy \(\rho\) of the flat connection \(\theta\) but also on the behavior of \(\theta\) on the boundary \cite{Kirk1993}.
To describe this dependence choose a \defemph{decoration} of \(\rho\), which is an eigenspace for the image of the peripheral subgroup under \(\rho\).
This choice determines eigenvalues \(m, \ell\) of a meridian and longitude of \(K\).
A \defemph{log-decoration} (in the language of \cite{McPhailSnyderAlgebra}) is a choice \(\mathfrak{s}\) of logarithms \(\mu, \lambda\) of \(m, \ell\).
Results of \textcite{Kirk1993} show that fixing a log-decoration gives a well-defined invariant
\[
  S(\extr{K}, \rho, \mathfrak{s}) \in \mathbb{C}/\tu \mathbb{Z}
\]
and that \(S\) changes in a simple way when \(\mathfrak{s}\) does.
In other words, for a manifold with boundary like a knot exterior
\[
  \mathcal{I}\leftfun(K; \rho, \mathfrak{s}\rightfun)
  \defeq
  \exp\leftfun(
    S(\extr{K}, \rho, \mathfrak{s})
  \rightfun)
  \in \mathbb{C}^{\times}
\]
is not simply a complex number but instead should be viewed as a section of a line bundle over the representation variety of \(\extr{K}\).

In this context it is natural to consider a modification of \(S\) by an elementary function
\(
  \psi(\rho, \mathfrak{s}) \defeq \tu \mu \lambda.
\)
We can think of \(\psi\) as a Dehn filling correction term.
If \(\rho\) admits a \((p,q)\) Dehn filling with \(p \mu + q \lambda = \tu\) then up to normalization \(\psi(\rho, \mathfrak{s})\) is the complex length (length plus \(\ii\) times torsion) of the geodesic core of the added solid torus, so the value of \(S + \psi\) is the complex Chern-Simons invariant of the filled manifold.
More generally \(S + \psi\) is a holomorphic function on Dehn surgery space (in fact, on the entire character variety) and is up to normalization the function denoted \(f\) in \cite[Theorem 2]{zbMATH03947227} and \(F\) in \cite[Theorem 2]{Yoshida1985}.
Furthermore \(S + \psi\) depends only on the log-meridian \(\mu\) and not the log-longitude.
For this (and other) reasons it is the natural way for us to normalize the Chern-Simons invariant of a knot exterior, so we work with the \(\psi\)-shifted Chern-Simons invariant
\[
  \csfunc{K, \rho, \mu} \defeq
  \exp\leftfun(
    S(\extr{K}, \rho, \mathfrak{s})
    +
    \psi(\rho, \mathfrak{s})
  \rightfun)
  \in \mathbb{C}^{\times}
  .
\]

Here is a more detailed description of the line bundle mentioned above.
Let \(\sigma : \pi_{1}(T^2) \to \slg\) be a representation of the torus fundamental group.
Fixing a meridian-longitude pair \((\mer, \lon)\) a decoration is a common eigenspace of \(\sigma(\mer)\) and \(\sigma(\lon)\), which defines preferred eigenvalues \(m, \ell\).
The set of quasi-periodic functions
\begin{equation}
  \label{eq:CS quasi periodicity knots}
  f : \set{\mu \in \mathbb{C} \given e^{\tu \mu} = m } \to \mathbb{C} \text{ such that } f(\mu + 1) = \ell^{2} f(\mu)
\end{equation}
is a one-dimensional complex vector space \(\mathcal{E}(\sigma)\) and together these give a line bundle \(\mathcal{E} \to \repvar{T^{2}}\) over the decorated representation variety of the torus.
We can use the framing and orientation of \(K\) to fix an identification \(\varphi\) between \(\partial \extr{K}\) and \(T^{2}\) and we define \(\csfunc{\partial \extr{K}}\) to be \(\varphi^{*} \mathcal{E}\).
It is straightforward to show that the map \(\mu \mapsto \csfunc{K, \rho, \mu}\) satisfies the quasi-periodicity condition \eqref{eq:CS quasi periodicity knots}, which means \(\rho \mapsto \csfunc{K, \rho}\) is a section of the pullback bundle \(r^{*} \csfunc{\partial \extr{K}}\):
\begin{equation}
  \label{eq:CS inv section diagram}
  \begin{tikzcd}
  & \csfunc{\partial \extr{K}} \arrow[d] \\
    \repvar{\extr{K}} \arrow[r, "r"] \arrow[ur, "\csfunc{K}"] & \repvar{\partial \extr{K}}
  \end{tikzcd}
\end{equation}

This behaviour is expected from the point of view of topological quantum field theory.
A theory \(\mathcal{F}\) defined on a cobordism category of manifolds (without a choice of representation) assigns
\begin{itemize}
  \item 
    complex numbers \(\mathcal{F}(M)\) to closed \(3\)-manifolds,
  \item
    vector spaces \(\mathcal{F}(\Sigma)\) to surfaces, and
  \item
    vectors to manifolds with boundary.
    In particular for a knot exterior \(\mathcal{F}(\extr{K}) \in \mathcal{F}(\partial \extr{K}) \iso \mathcal{F}(T^2)\).
\end{itemize}
A geometric theory like \(\csfunc{}\) additionally depends on a choice of (decorated) representation \(\rho\), so it assigns
\begin{itemize}
  \item 
    functions \(\csfunc{M} : \repvar{M} \to \mathbb{C}\) on the decorated representation variety to closed \(3\)-manifolds,
  \item
    line bundles \(\csfunc{\Sigma} \to \repvar{\Sigma}\) over the decorated representation variety to surfaces, and
  \item
    sections of line bundles to manifolds with boundary.
    These are pullbacks of the boundary line bundles along restriction maps as in \eqref{eq:CS inv section diagram}.
\end{itemize}
In this case \(\csfunc{}\) is an invertible, classical field theory, so its values on surfaces are  line bundles.
One could more generally consider quantum invariants taking values in higher-rank vector bundles, and \(\qfunc{}\) is such an invariant.

\subsection{The quantum invariant}

Our main result is the definition of a new family \(\qfunc{}\) of tangle invariants that we interpret as a quantization of \(\csfunc{}\).
The quantization parameter is an integer \(\nr \ge 2\) which corresponds to \(\hbar = \tu /\nr\) and \(q = e^{\tu /\nr}\).
Here we describe properties of the associated knot invariants \(\qinv{}\) to avoid technicalities involving vanishing quantum dimensions. 

\(\qinv{}\) assigns an oriented framed knot \(K\) and a decorated representation \(\rho : \pi_1(\extr{K}) \to \slg\) a function
\[
  \mu \mapsto \qinv{K ; \rho, \mu}
\]
on the set of log-meridians of \(\rho\).
This function satisfies a similar quasi-periodicity condition
\begin{equation}
  \qinv{K ; \rho, \mu + \nr} 
  =
  \ell^{-2}
  \qinv{K ; \rho, \mu} 
\end{equation}
but now the period is \(\nr\) so \(\qinv{K; \rho}\) is determined by its values on \(\nr\) log-meridians \(\mu, \mu + 1, \dots, \mu + \nr - 1\).
In other words, there is an analogous rank \(\nr\) \emph{vector} bundle \(\qfunc{T^2}\) over  \(\repvar{T^2}\) and \(\qinv{K}\) defines a section of the pullback of this bundle to \(\repvar{K} = \repvar{\extr{K}}\):
\[
  \begin{tikzcd}
  & \qfunc{\partial \extr{K}} \arrow[d] \\
    \repvar{\extr{K}} \arrow[r, "r"] \arrow[ur, "\qinv{K}"] & \repvar{\partial \extr{K}}
  \end{tikzcd}
\]
The vector bundle \(\qfunc{T^2}\) is closely related to the line bundle \(\csfunc{T^2}\) and this behaviour makes it natural to interpret \(\qinv{}\) as a quantization of \(\csfunc{}\).

There are other, more concrete reasons for this interpretation.
In order to compute \(\csfunc{K; \rho, \mu}\) in practice one cuts \(\extr{K}\) into pieces (say, via an ideal triangulation) and computes the Chern-Simons invariant of each piece.
The invariants of the pieces depend on additional choices of local data that have to be made coherently, for example a choice of \defemph{flattening} \cite{Neumann2004} of the triangulation.
These flattenings can be viewed in terms of explicit boundary conditions for flat connections on the tetrahedra \cite{Marche2012}.
In \cref{sec:The Chern-Simons invariant of a tangle} we describe a more convenient way to do this for octahedral decompositions.
\(\qinv{K, \rho, \mu}\) is defined using pieces depending on local boundary conditions directly analogous to those for \(\csfunc{}\).
Where \(\csfunc{}\) uses classical dilogarithms \(\qinv{}\) uses quantum dilogarithms.

These are algebraic, formal reasons to view \(\qfunc{}\) as a quantization of \(\csfunc{}\).
A more physical claim is that \(\qfunc{}\) is a mathematical realization of quantum Chern-Simons theory with complex gauge group \(\slg\).
We do not show this here, but we view our results as a precise conjecture about the properties of such a quantization.
In this direction we note that the state space of the torus under geometric quantization of complex Chern-Simons theory \cite{Andersen2016} involves spaces of quasi-periodic functions.
This observation inspired our description of \(\csfunc{}\) and \(\qfunc{}\) in terms of spaces of such functions.
It would be interesting to study this geometric quantization in more detail for knot complements and compare the results to \(\qfunc{}\).

In this paper we define \(\qfunc{}\) only for link and tangle exteriors, not general \(3\)-manifolds.
Using the surgery construction \cite{Reshetikhin1991} it should be possible to extend \(\qfunc{}\) to a \(3\)-manifold invariant and more generally a TQFT for manifolds equipped with \(\slg\) representations.
This requires dealing with several technical issues, many of which have already been resolved by \citeauthor{Blanchet2016} \cite{Costantino2012,Blanchet2016}.
We expect to recover their TQFT from \(\qfunc{}\) in the limit where the representation \(\rho\) is diagonal, as is already known to occur for link invariants (\cref{sec:Relation to the Kashaev--Akutsu-Deguchi-Ohtsuki invariant}).

\subsection{Volume conjectures}
One motivation for constructing geometric quantum invariants like \(\qfunc{}\) is the Volume Conjecture of \textcite{Kashaev1997}, which proposes a relationship between algebraic invariants like the Jones polynomial and geometric invariants like \(\csfunc{}\).
By constructing an invariant like \(\qfunc{}\) with features of both one could hope to explain or prove this conjecture.

\begin{conjecture}
  \label{conj:our volume}
  Let \(K\) be a hyperbolic knot with complete hyperbolic structure \(\rho_{\operatorname{hyp}} : \pi_{1}(\extr{K}) \to \pslg\) lifted to \(\slg\) with meridian eigenvalue \(1\).
  Let \(\rho\) be \emph{any} boundary-parabolic representation of \(\pi_{1}(\extr{K})\).
  Then as \(\nr \to \infty\) the value of \(\qinv{}\) at the log-meridian \((\nr - 1)/2\) grows exponentially
  \begin{equation}
  \label{eq:volume conj ours}
    \frac{1}{\nr}
    \log \qinv{K, \rho , (\nr - 1)/2 }
    \to
    \frac{
      \log \csfunc{K, \rho_{\operatorname{hyp}}, 0 }
      }{
      2 \pi 
    }
  \end{equation}
  with rate controlled by the complex Chern-Simons invariant (in particular, the hyperbolic volume) of the complete structure.
\end{conjecture}

The surprising part of \cref{conj:our volume} is that the growth rate is determined by the volume of \(\rho_{\operatorname{hyp}}\), not of \(\rho\).
By \cref{thm:KADO} the value \(\qinv{K, \pm 1 , (\nr - 1)/2 }\) at the trivial representation \(\pm 1 : \pi_{1}(\extr{K}) \to \pslg\) is Kashaev's invariant, so the complexification \cite{Murakami2002} of Kashaev's Volume Conjecture is the claim that
\begin{equation}
  \label{eq:volume conj Kashaev}
  \frac{1}{\nr}
  \log \qinv{K, \pm 1 , (\nr - 1)/2 }
  \to
  \frac{
    \log \csfunc{K, \rho_{\operatorname{hyp}}, 0 }
    }{
    2 \pi 
  }
\end{equation}
We can think of \cref{conj:our volume} as a weaker statement since it involves a nontrivial representation \(\rho\).
One might be able to prove the weaker conjecture \eqref{eq:volume conj ours} as a step towards the original Volume Conjecture \eqref{eq:volume conj Kashaev}.

We give details and supporting numerical evidence for the figure-eight knot in \cref{sec:Example: the figure-eight knot,rem:vol conj evidence}.
We also have theoretical reasons to expect exponential growth: \(\qinv{}\) can be expressed as a sum of \defemph{state integrals}, which are contour integrals over a space parametrizing hyperbolic structures.
Applying the saddle point method to these integrals leads to exponential growth.
In \cite{McPhailSnyderState} we give a precise definition of the state integrals and discuss barriers to applying the saddle point approximation.

\subsection{Overview of the construction}

We define \(\qfunc{}\) using a geometric variant of the Reshetikhin-Turaev construction \cite{Reshetikhin1990} due to \citeauthor{Kashaev2005}.
The representation theory of \(\qgrplong\) is key to this construction but we focus here on the connections to Chern-Simons theory.
To read this paper it suffices to understand some basic facts about categories of modules over Hopf algebras and the statements of some theorems in \cite{McPhailSnyderAlgebra}.
These are given in \cref{sec:algebras}.

At a root of unity \(\omega\) the quantum group \(\qgrplong\) has a family of modules parametrized by an algebraic group \(\slg^{*}\) closely related to \(\slg\) \cite{de1991representations}.
\citeauthor{Kashaev2004} \cite{Kashaev2004} showed how to define braidings for these modules and interpreted the geometric parametrization in terms of flat \(\sla\) connections on the tangle complement \cite{Kashaev2005}.
This gives invariants of tangles with \(\slg\)-representations.
\textcite{Blanchet2018} used the Kashaev-Reshetikhin construction \cite{Kashaev2005} to define link invariants \(\operatorname{F}_{\nr}\) similar to \(\qinv{}\), but there are a number of unresolved technical issues in their work.
Most importantly:
\begin{enumerate}
  \item there are no explicit formulas for the braidings, which makes it difficult to compute \(\operatorname{F}_{\nr}\) in practice, and
  \item the braidings, hence \(\operatorname{F}_{\nr}\) are only well-defined up to a \(\nr^{2}\)th root of unity.
\end{enumerate}
The second is particularly significant because the surgery TQFT construction \cite{Reshetikhin1991} requires a sum over link invariants and this sum is not well-defined if its terms have indeterminate phases.

Recent work of the author and Reshetikhin \cite{McPhailSnyderAlgebra} solves both of these problems.
The key idea is to use a nonstandard presentation of \(\qgrplong\) inspired by cluster algebras.
This gives a parametrization of \(\qgrplong\)-modules using \defemph{octahedral coordinates} closely related to hyperbolic geometry \cite{McphailSnyder2022hyperbolicstructureslinkcomplements}.
Our presentation gives recurrences for the matrix coefficients of the braiding that one can solve explicitly.
Using the octahedral coordinates it is natural to express the matrix coefficients in terms of quantum dilogarithms.
Via an analogy with the classical Chern-Simons invariant we prove that this braiding satisfies the braid relation exactly, with no phase ambiguity.

The earlier paper \cite{McPhailSnyderAlgebra} focuses on the representation theory underlying this construction and the relevant geometric braiding.
This paper focuses on the still-nontrivial problem of turning this braiding into a tangle invariant.
Our goal is to explain how to understand the algebraic methods of \cite{McPhailSnyderAlgebra} in terms of Chern-Simons theory and hyperbolic geometry.
In the process we use an improved procedure \cite{McphailSnyder2024octahedralcoordinateswirtingerpresentation} for determining the octahedral coordinates, which makes it much easier to handle a number of technical issues.%
\note{
  The main problem is that the octahedral coordinates are only defined away from certain singular configurations which need to be avoided.
  It is much easier to work with Wirtinger presentations that are always well-defined and determine the octahedral parameters from these.
  The approach used here avoids the generic biquandle factorizations of \textcite{Blanchet2018}.
}
Another difficulty in \cite{McPhailSnyderAlgebra} is keeping track of log-colorings and log-longitudes of tangle diagrams.
Here we use a different description of the \(\qgrplong\)-modules assigned to points that allows us to either eliminate the log-parameters or incorporate them into a choice of basis.
To do this we use the natural action of the Weyl algebra on spaces of quasi-periodic functions, a version of the Schr\"{o}dinger representation.

Our methods are closely related to Turaev's homotopy quantum field theories \cite{Turaev2010}, which produce invariants of tangles and manifolds with representations into a group \(G\).
However, our formalism is more general:
\begin{enumerate}
  \item We work with tangle diagrams with colors on both the segments and regions, which means we use \(2\)-categories and not monoidal categories.
  \item Our functors are only defined away from certain singular configurations, so we need to show that these can be avoided.
    As part of this we consider categories of tangle \emph{diagrams}, not tangles.
  \item A HQFT with target \(K(G,1)\) produces invariants of \(G\)-colored tangles valued in scalars, which is to say a section of a product bundle over the \(G\)-representation variety of the tangle.
    Our invariants are instead sections of a (potentially nontrivial) vector bundle.
\end{enumerate}

We call the extra data labeling the regions of our diagrams a \defemph{shadow coloring}.
Shadows (the region colors) are elements of an abelian group \(\mathbb{C}^{2}\) with an action of the group \(\slg\) labeling the segments.
This data is a special case of a \defemph{crossed module} \cite[Section 2.2]{zbMATH07707527}, which is a combinatorial description of a pointed connected homotopy \(2\)-type \(\delooping \eta\), just as a group \(G\) is a description of the homotopy type of its classifying space \(\delooping G\).
Categories graded by a crossed module \(\eta\) were used by \textcite{zbMATH07707527} to define quantum invariants of manifolds equipped with maps to \(\delooping \eta\).
Our invariants should fit into this framework: we can view our results in \cref{sec:Modules} as describing a grading of (part of) the category of \(\qgrp\)-modules by a crossed module.
The fact that \(\pi_{2}(\extr{K})\) vanishes for non-split knots \(K\) ought to be related to the independence of \(\qinv{K, \rho, \mu}\) from the choice of shadow coloring.

The octahedral coordinates are a version of the \defemph{Ptolemy coordinates} \cite{Zickert2016,Garoufalidis2015} associated to the octahedral decomposition of a tangle diagram.
As such they are closely related to cluster coordinates on a space \(\mathscr{P}\) parametrizing suitably decorated \(\slg\) representations of a punctured disc.
In particular shadow colorings appear to be closely related to the \defemph{pinnings} used by \citeauthor{arXiv:1904.10491} \cite{arXiv:1904.10491} to describe the gluing of \(\mathscr{P}\) and related moduli spaces.
In addition, the natural appearance of Schr\"{o}dinger representations in the construction of \(\qfunc{}\) suggests that it is related to the quantization of \(\mathscr{P}\).
This may provide a more systematic method to define \(\qfunc{}\) and prove its invariance.

\subsection{Results}

We sketch the statements of our main results.
Let \(\nr \ge 2\) be an integer.
In \cref{def:tcat} we define a \(2\)-category \(\tcat\) encoding tangles with decorated \(\slg\) representations of their complements, plus some additional gauge data.
More precisely, \(\tcat\) has
\begin{description}
  \item[objects] 
    (regions of diagrams labeled with) elements of \(\shadset = \mathbb{C}^{2} \setminus \set{0}\).
    These are additional gauge data that we use to keep track of local basepoints, as explained in \cref{sec:shadow colorings}.
  \item[\(1\)-morphisms]
    lists of oriented, \defemph{colored} points.
    A coloring is a matrix \(g \in \slg\) (the value of \(\rho\) on the meridian) and vectors \(u, gu \in \shadset\) labelling the regions adjacent to the point; we think of the point as a morphism \(u \to gu\).
    We also specify an eigenspace \([v]\) of \(g\) (a \defemph{decoration}) and a logarithm \(\mu\) of an eigenvalue of \(g\) determined by \([v]\).
  \item[\(2\)-morphisms]
    oriented framed tangle diagrams up to planar isotopy (but not isotopy) colored with similar data to the points, again subject to admissibility conditions.
    We call these \defemph{colored tangle diagrams}.
\end{description}
We compose \(1\)-morphisms by disjoint union, subject to compatibility between the region labels.
The composition of \(2\)-morphisms is the usual composition of colored tangles.
We need to work with tangle \emph{diagrams} (not tangles) because our admissibility conditions depend on a choice of diagram.
We say (\cref{def:equivalent tangle diagrams}) that two diagrams of \(\tcat\) are \defemph{colored isotopic} if the underlying tangles are isotopic and their geometric data agree.

Because \(\qgrplong\) is a Hopf algebra its category of modules is monoidal.
We consider a subcategory \(\qcat\) of \defemph{weight modules} on which the center of \(\qgrplong\) acts semisimply.
As for any monoidal category we can view it as a \(2\)-category with one object.
The \defemph{delooping} of \(\qcat\) is the \(2\)-category \(\delooping\qcat\) with
\begin{description}
  \item[objects] 
    a single object \(\bullet\)
  \item[\(1\)-morphisms]
    finite-dimensional \(\qgrplong\) weight modules, with composition of \(1\)-morphisms the tensor product of modules.
  \item[\(2\)-morphisms]
    \(\qgrplong\)-module intertwiners, composed as linear maps.
\end{description}

In \cref{sec:construction of invariant} we define a \(2\)-functor
\[
  \qfunc{} : \tcat \to \delooping\qcat
\]
and show that \(\qfunc{D}\) depends only on the colored isotopy class of the tangle diagram \(D\), so it gives an invariant of the underlying oriented framed tangle \(T\), the induced log-decorated representation of the complement of \(T\), and some extra gauge data determined by the shadow coloring.
\Cref{thm:qfunc defined and properties} states this precisely and gives the dependence on the framing and log-meridians.

Along the way we discuss how the \(\nr = 1\) case of this construction recovers the Chern-Simons invariant \(\csfunc{}\) of link exteriors.
In the process we show (\cref{sec:The Chern-Simons invariant of a tangle}) how to define \(\csfunc{}\) for tangles, which to my knowledge has not been previously described in detail.
This construction helps explain our interpretation of \(\qfunc{}\) as a quantization of \(\csfunc{}\), but it is also part of our proof that \(\qfunc{}\) satisfies the Reidemeister \(3\) move:
we know for representation-theoretic reasons it satisfies it up to a scalar, but checking that this scalar is \(1\) requires the same fact for \(\csfunc{}\).

Because the modules underlying \(\qfunc{}\) have vanishing quantum dimension the value of \(\qfunc{}\) on any link diagram is zero.
To extract nontrivial link invariants we use a standard trick: if \(L\) is a link with distinguished component \(K\), we choose  a \(1\)-\(1\) tangle diagram \(D\) whose closure is \(L\) (a \defemph{cut presentation}) with a coloring inducing \(\rho\) and \(\mu\).
Since \(\qfunc{D}\) is an endomorphism of an \defemph{absolutely simple} module (\ref{def:absolutely simple object}) it is a scalar times the identity, and we define \(\qinv{L, K; \rho ,\mu}\) to be this scalar.

We show (\cref{thm:qinv defined and properties}) that  that \(\qinv{L, K ; \rho, \mu}\) is independent of the choice of cut presentation and is gauge invariant.
It does depend on the choice of component \(K\), but we should be able to remove this dependence using the modified dimensions of \textcite{Geer2009}; we do not do this here to avoid more complicated theorem statements.
We explain how to reduce the computation of \(\qfunc{}\) to a tensor contraction in \cref{sec:state-sums and examples} and give some examples, including exact values for Hopf links.

In a few special cases \(\qinv{}\) recovers known link invariants.
In the limit where \(\rho\) is reducible (in particular, diagonal) we recover the invariants of \textcite{Kashaev1995} and \textcite{Akutsu1992}.
When \(\nr = 2\) we can use the main result of \cite{McPhailSnyder2020} to show (\cref{thm:torsion relation}) that \(\qfunc[2]{L, K; \rho, \mu}\) is a refinement of the \(\rho\)-twisted Reidemeister torsion.
Specifically, if \(\overline{L}\) is the mirror image of \(L\) then for a certain choice of log-decorated representation \(\overline{\rho}, \overline{\mu}\) we have
\[
  \qfunc[2]{L, K; \rho, \mu}
  \qfunc[2]{\overline{L}, \overline{K}; \overline{\rho}, \overline{\mu}}
  =
  \tau(\comp{L}, \rho)
\]
where \(\tau(\comp{L}, \rho)\) is the twisted Reidemeister torsion.

\subsection{Related constructions}
There are other invariants in quantum topology that are closely related or equivalent to ours.
We discuss a few of them below.

\subsubsection{Quantum hyperbolic invariants}
Let \(M\) be a \(3\)-manifold, \(L\) a link in \(M\), and \(\rho : \pi_{1}(M \setminus L) \to \slg\) a representation.
For each odd integer \(N\) \citeauthor{Baseilhac2004} \cite{Baseilhac2004} define a quantum hyperbolic invariant \(H_N(M, L, \rho) \in \mathbb{C}\) defined up to some phase indeterminacies.
Their construction recovers the Kashaev invariant \cite{Baseilhac2011kashaev} and can be extended to a TQFT \cite{Baseilhac2007}.

The \(H_N\) are constructed by finding an ideal triangulation of \((M, L)\) expressing \(\rho\) in terms of shape parameters for each tetrahedron.
One must also make some additional choices of integer data related to a flattening; the invariant does not depend on these up to the aforementioned phase indeterminacy.
Via the quantum dilogarithm each decorated tetrahedron corresponds to a tensor.
Gluing the tetrahedra to obtain \((M, L, \rho)\) gives a tensor contraction whose value is \(H_N(M, L, \rho)\).
This does not depend on the choice of triangulation because the quantum dilogarithms satisfy a \defemph{five-term} (\defemph{pentagon}) relation and are therefore invariant under the \(3\)-\(2\) move on triangulations.

The Baseilhac-Benedetti construction is closely related to ours: both are extensions of Kashaev's original work \cite{Kashaev1995} on quantum dilogarithms.
In both cases one takes a sum involving one quantum dilogarithm for each tetrahedron in a triangulation.
Because we are defining Reshetikhin-Turaev link invariants we always work with tangle diagrams and their octahedral decompositions.
They instead use Turaev-Viro state sums, which are defined for more general triangulations.
We also use a different version of the quantum dilogarithm (\cref{sec:quantum dilogarithms}).
We expect that their invariants are equivalent%
\note{%
  In general when starting from the same algebraic data the TV invariants are the norm-square of the RT invariants.
  Baseihac-Benedetti use the \(6j\) symbols of the upper Borel of \(\qgrp\), so their invariant should be the same as ours, which uses modules for the full algebra \(\qgrp\).
}
to \(\qfunc{}\) up to a change in normalization involving a \(\nr\)th root of \(\csfunc{}\).

There are a few advantages of our approach over that of \citeauthor{Baseilhac2004}.
The most obvious is that \(\qfunc{}\) is defined unambiguously while the quantum hyperbolic invariants \(H_N\) are defined only up to a phase indeterminacy.
In addition, because we use ideal octahedra and not tetrahedra our invariants are better at handling geometrically degenerate representations.
This leads to a much more direct proof that \(\qfunc{}\) recovers the Kashaev invariant when the representation is trivial.
One disadvantage is that \(\qfunc{}\) is not (yet) defined for general \(3\)-manifolds, only for link exteriors; we plan to address this in future work.

\subsubsection{Skein algebras of surfaces}
\citeauthor{Bonahon2021} \cite{Bonahon2021, Bonahon2022} used skein algebras to define another type of geometric quantum invariant.
Suppose \(M = M_{\phi}\) is a \(3\)-manifold obtained as the mapping torus of a diffeomorphism \(\phi : \Sigma \to \Sigma\) of a surface \(\Sigma\).
We can create a link \(L\) in \(M\) by adding punctures to \(\Sigma\).
At a root of unity a representation of the Kauffman bracket skein algebra of the punctured surface \(\Sigma\) determines a representation \(\pi_{1}(\Sigma) \to \slg\) \cite{Bonahon2019}.
If \(\phi\) fixes the character of this representation it induces a representation \(\rho : \pi_{1}(M_{\phi}) \to \slg\).
Again there is some additional data at the punctures closely related to decorations and log-meridians.
The trace of the action of \(\phi\) on the skein algebra gives an invariant of \((M, \rho)\).

As before the technical details of the construction are quite different than for \(\qfunc{}\) but we expect that the two invariants are equivalent up to a change in normalization.
Both invariants are derived from quantizations of the moduli space of \(\slg\) representations:
\(\qfunc{}\) appears to be more closely related to quantizations of cluster algebras, while the BWY invariants are defined using skein algebras.
Since both quantizations ought to be equivalent our invariants should be as well.
More concretely, the Kauffman bracket skein algebra is related to the \emph{cobraiding} on the quantum matrix group \(\operatorname{SL}_q(2)\), while our construction is related to the braiding on the quantum enveloping algebra \(\qgrplong[q]\).
These are dual to each other so we expect an equivalence.

\subsubsection{Kashaev-Reshetikhin invariants}
Finally, we mention some related work of Chen, the author, Morrison, and Snyder \cite{Chen2021}.
Extending the original work of \textcite{Kashaev2005} we show how to define quantum knot invariants \(\operatorname{KR}_\ell\) depending on a representation \(\rho^{(\ell)}\) of the \defemph{extended knot group} of \(K\) into \(\slg\); the data extending \(\rho\) to \(\rho^{(\ell)}\) is closely related to a log-decoration of \(\rho\).
A major difference is that a point with \(\weyl\)-character \(\chi\) (\cref{sec:algebras}) is assigned the adjoint module \(\qgrp/\ker \chi\), not the simple module \(\qfunc{\chi}\) defined in \cref{sec:Modules}.
This allows a direct proof that the braid group representation is normalized correctly and avoids the use of quantum dilogarithms.
The drawback is that the geometric meaning of \(\operatorname{KR}_{\ell}\) and its relation to Kashaev's invariants are less clear.

Another result in \cite{Chen2021} is that \(\operatorname{KR}_\ell\) can be viewed as a rational function on an extended version \(\mathfrak{X}_{K}^{(\ell)}\) of the character variety \(\mathfrak{X}_K\) of the knot.
Further, for hyperbolic knots one can express \(\operatorname{KR}_\ell\) as a rational function on the geometric component of a similarly extended \(A\)-polynomial curve.
Similar results hold for \(\qfunc{}\), although the presence of dilogarithms results in holomorphic functions instead of rational functions.

\subsection*{Acknowledgements}
I would like to thank Francis Bonahon, Dan Freed, Nathan Geer, David Green, Effie Kalfagianni, Andy Neitzke, and Bertrand Patureau-Mirand for helpful discussions.

\subsection*{Plan of the paper}
\begin{description}
  \item[\cref{sec:tangles}]
    We set conventions on tangles, representations, decorations, and coordinates on them related to octahedral decompositions.
    In particular, we define the category \(\tcat\) of colored tangles on which our functors \(\csfunc{}, \qfunc{}\) are defined.
    We discuss Reidemeister moves for colored tangle diagrams and their relationship to colored isotopy.
  \item[\cref{sec:invariants from models}]
    We define a \defemph{model} of \(\tcat\), which is the data needed to define a functorial invariant of \(\tcat\) in a pivotal category.
    We prove that models give invariants and discuss some simple examples related to quandle cocycles.
  \item[\cref{sec:The Chern-Simons invariant of a tangle}]
    We explain how to define the classical Chern-Simons invariant of a tangle as a model of \(\tcat\) and show it recovers the usual invariant of link exteriors and \(3\)-manifolds.
  \item[\cref{sec:construction of invariant}]
    We define the quantum invariant \(\qfunc{}\) as a model of \(\tcat\) and discuss its properties.
    Some proofs are delayed to \cref{sec:The pinched limit,sec:proofs}.
  \item[\cref{sec:state-sums and examples}]
    We explain how to compute \(\qfunc{}\) in practice.
    We compute it numerically for boundary-parabolic representations of the figure-eight knot and find exact values for Hopf links.
  \item[\cref{sec:torsion}]
    We apply results of \cite{McPhailSnyder2020} to show that that \(\qfunc[2]{}\) determines the \(\slg\)-twisted Reidemeister torsion.
  \item[\cref{sec:The pinched limit}]
    We show that the braiding of \(\qfunc{}\) makes sense even at geometrically degenerate (\defemph{pinched}) crossings.
    We use these results to recover the Kashaev and ADO invariants.
  \item[\cref{sec:proofs}]
    We prove that \(\qfunc{}\) is a model of \(\tcat\), completing the proof of our main results.
  \item[\cref{sec:quantum dilogarithms}]
    We set conventions on the quantum dilogarithm \(\pf{}\) and classical dilogarithm \( \ldil{} \) and collect some properties used elsewhere in the paper.
\end{description}

\subsection*{Conventions}
Throughout the paper \(\nr\) is an integer \(\ge 2\) and \(\ii\) (not \(i\)) is the imaginary unit.
We write \(\omega^{x} = \exp(\tu x /\nr)\), so \(\omega = e^{\tu /\nr}\) is a primitive \(\nr\)th root of unity.
All tangles are oriented and framed and tangle diagrams have the blackboard framing.

\section{Tangles, octahedral decompositions, and coordinates for representations}
\label{sec:tangles}

Here we establish conventions on tangle diagrams and how to represent decorated representations via colorings of tangle diagram.
We explain how to understand the geometric parameters of the associated octahedral decompositions in terms of these colorings, which leads to an admissibility condition.
We define colored analogues of the usual Reidemeister moves that preserve equivalence of diagrams.
Unlike in the usual case equivalence of \emph{admissible} diagrams (i.e.\ those for which \(\qfunc{}\) is defined) is not generated by Reidemeister moves: it possible to have two equivalent diagrams for which every path of moves includes an inadmissible diagram.%
\note{
  I do not know of a example of this phenomenon, but I also do not know how to prove it cannot occur.
}
Instead equivalence is generated by Reidemeister moves and gauge transformations.
The latter can be realized by dragging a strand over a diagram, so the solution to our problem is to allow \defemph{stabilization} moves that add an additional unlinked component to the tangle diagram, as well as their inverses.
In \cref{thm:diagrams linked by stabilization and R} we use stabilization moves to reduce the proof of invariance of \(\qfunc{}\) (more generally, any invariant coming from a model of \(\tcat\) in a pivotal category) to checking admissible Reidemeister moves.

\subsection{Conventions on tangles and diagrams}

\begin{marginfigure}
  \begin{tikzpicture}[line width=1, baseline=30, scale=2, xscale = 1.5] 
    \coordinate (1) at (0,1);
    \coordinate (2) at (0,0);
    \coordinate (3) at (1,0);
    \coordinate (4) at (1,1);
    \draw[->] (2) node[left] {\(2\)} \br (4) node[right] {\(4\)};
    \draw[->] (1) node[left] {\(1\)} \br (3) node[right] {\(3\)};
    \node at (0.5, 1) {\(\lN\)};
    \node at (0, 0.5) {\(\lW\)};
    \node at (0.5, 0) {\(\lS\)};
    \node at (1, 0.5) {\(\lE\)};
  \end{tikzpicture}
  \caption{
    Standard region and segment labels near a crossing (of either sign). In this tangle region \(\lW\) is below segment \(1\) and region \(\lN\).
  }
  \label{fig:crossing-regions}
\end{marginfigure}

\begin{marginfigure}
  \begin{tikzpicture}[line width = 1, scale = 2]
    \draw[->] (0,0) to (1, 0) node[right] {\(i\)};
    \node[above] at (0.5,0) {\(j\)};
    \node[below] at (0.5,0) {\(j'\)};
  \end{tikzpicture}
  \caption{
    Here region \(j'\) is below region \(j\) across segment \(i\).
  }
  \label{fig:region-orientation-rule}
\end{marginfigure}

A \defemph{tangle} is a (PL) embedding of arcs and circles into \([0,1] \times \mathbb{R}^{2}\) so that the endpoints of the arcs lie on \(\set{0} \times \mathbb{R}^{2} \) and \(\set{1} \times \mathbb{R}^{2} \), up to boundary-fixing isotopy.
Our convention is that tangles and their diagrams read left-to-right and that disjoint union of tangles is top-to-bottom.
Unless otherwise stated all tangles are oriented and framed and tangle diagrams are given the blackboard framing.
We write \(\pi(T)\) for the fundamental group of the complement of \(T\) as a submanifold of \([0,1] \times \mathbb{R}^{2}\).
Sometimes we refer to a tangle or link \defemph{exterior} \(\extr{T}\) which is the complement of a regular neighborhood of \(T\); in particular a link exterior \(\extr{L}\) has one torus boundary component for each component of  \(L\).

\begin{definition}
  \label{def:tangle parts}
  Suppose \(T\) has \(n\) incoming and \(m\) outgoing boundary points and let \(D\) be a diagram of \(T\).
  We can think of \(D\) as a decorated \(4\)-valent graph \(G\) embedded in \([0,1] \times \mathbb{R}\) with \(n\) edges intersecting \(\set{0} \times \mathbb{R}\) and \(m\) edges intersecting \(\set{1} \times \mathbb{R}\).
  We assign names to various parts of \(D\):
  \begin{itemize}
    \item
     The \defemph{segments} of \(D\) are the edges%
      \note{%
        Usually these are called the ``edges'' of the diagram, but we do not want to confuse them with edges of ideal polyhedra in the octahedral decomposition so we follow the terminology of \textcite{Kim2018}.
      }
      of \(G\).
    \item
    An \defemph{arc} of \(D\) is a set of adjacent over-segments.
    \item
      A \defemph{component} of \(D\) is a set of adjacent segments; these are in bijection with connected components of \(T\).
    \item
    A \defemph{region} of \(D\) is a connected component of the complement of \(G\), equivalently a vertex of the dual graph of \(G\).
    When regions are adjacent across a segment as in \cref{fig:region-orientation-rule} we say that region \(j'\) is below segment \(i\) across region \(j\).
    For example in \cref{fig:crossing-regions} region \(\lW\) is below region \(\lN\) across segment \(1\).
  \end{itemize}
\end{definition}

\begin{figure}
  \centering
  \begin{tikzpicture}[line width=1, baseline=30, scale=1, xscale = 1.5] 
    \coordinate (1) at (0,1);
    \coordinate (2) at (0,0);
    \coordinate (3) at (1,0);
    \coordinate (4) at (1,1);
    \draw[->] (2) node[left] {\(2\)} \br (4) node[right] {\(4\)};
    \draw[white, line width=10] (1) \br (3);
    \draw[->] (1) node[left] {\(1\)} \br (3) node[right] {\(3\)};
  \end{tikzpicture}
  \quad \quad
  \begin{tikzpicture}[line width=1, baseline=30, scale=1, xscale = 1.5] 
    \coordinate (1) at (0,1);
    \coordinate (2) at (0,0);
    \coordinate (3) at (1,0);
    \coordinate (4) at (1,1);
    \draw[->] (1) node[left] {\(1\)} \br (3) node[right] {\(3\)};
    \draw[white, line width=10] (2) \br (4);
    \draw[->] (2) node[left] {\(2\)} \br (4) node[right] {\(4\)};
  \end{tikzpicture}
  \caption{Positive (left) and negative (right) crossings with our standard segment labeling.}
  \label{fig:crossing-types}
\end{figure}

Let \(D\) be a diagram of a tangle \(T\).
The choice of \(D\) gives a presentation for \(\pi(T)\) with one generator \(x_i\) for each arc and one relation
\begin{gather}
  \label{eq:wirtinger-positive}
  x_{2'} = x_{1}^{-1} x_{2} x_{1} \text{ (positive)}
  \\
  \label{eq:wirtinger-negative}
  x_{1'} = x_{2} x_{1} x_{2}^{-1} \text{ (negative)}
\end{gather}
at each crossing, where the indices are as in \cref{fig:crossing-types}.
(At a positive crossing \(1\) and \(1'\)  are the same arc.)
We call this the \defemph{Wirtinger presentation} \(\pi(D)\) of \(\pi(T)\).
If \(i\) is a segment of \(D\) we write \(x_{i}\) for the associated Wirtinger generator.

\subsection{Decorated representations}

\begin{definition}
  \label{def:decorated representation}
  A \defemph{decorated representation} \(\rho\) of a tangle \(T\) is a representation \(\rho : \pi(T) \to \slg\) with a choice of invariant line for each meridian of \(T\).
  More formally, each component \(i\) of the oriented tangle has a conjugacy class of meridians \([x_i] \subset \pi(T)\).
  For a representative \(x_i \in \pi(T)\) we choose a line \(L_i \subset \mathbb{C}^2\) (thought of as a set of \emph{row} vectors) with
  \[
    L_i \rho(x_i) = L_i.
  \]
  We frequently write \(L_i = [v_i] = \operatorname{Span}(v_i)\) to indicate \(v_i\) is an eigenvector spanning \(L_i\) and call the pair \((g_i, [v_{i}])\) a \defemph{decorated matrix}.
\end{definition}

This definition does not depend on the choice of representative meridian: if \(x_i' = y^{-1} x_i y\) is any other representative of the conjugacy class we assign it the line \(L_i \rho(y)\), since
\[
  L_i \rho(y) \rho(y^{-1} x_i y) = L_i \rho(y).
\]
This choice is called a \defemph{decoration} of component \(i\), and a decoration of \(\rho\) is a decoration of each of the components of \(T\).
A decoration of \(\rho\) induces an equivalent decoration of any conjugate \(g^{-1} \rho g\) in a similar way.

Generically a representation \(\rho\) of knot exterior has two decorations, as a diagonalizable matrix has two eigenspaces.
When \(\tr \rho(x) = \pm 2\) but \(\rho \ne \pm 1\) is nontrivial (i.e. when \(\rho\) is boundary-parabolic) there is only one decoration, and when \(\rho(x) = \pm 1\) is trivial there are infinitely many.
As such \(\repvar{K}\) is a branched \(2\)-sheeted covering space of the usual  \(\slg\) representation variety of \(K\), ramified over parabolic and trivial representations.
Similar statements apply to links and tangles.

\begin{definition}
  \label{def:decorated matrix}
  A \defemph{decorated matrix} is a pair \((g, [v])\) of a matrix \(g \in \slg\) and a left eigenspace \([v] = \operatorname{Span}\{v\}\) of \(g\).
  We call the \(m \in \mathbb{C}^{\times}\) with%
  \note{%
    The minus sign here is unfortunately correct: it matches the conventions of \cite{McPhailSnyderAlgebra}, which were chosen before the connection with decorations was fully understood.
  }
  \begin{equation}
    \label{eq:distinguished eigenvalue}
     v \cdot g = m^{-1} v
  \end{equation}
  the \defemph{distinguished eigenvalue} of \((g, [v])\).
  We frequently write \(g\) for \((g, [v])\).
  Decorated matrices are not a group, but we can generalize conjugation to them by
  \begin{equation}
    \label{eq:quandle decorated matrices}
    (g_{1}, [v_{1}])
    \qn 
    (g_{2}, [v_{2}])
    \defeq
    (g_{2}^{-1} g_{1} g_{2}, [v_{1} g_{2}] )
  \end{equation}
  The binary operation \(\qn\) defines a \defemph{quandle} \cite{zbMATH00165698} because it is invertible and satisfies 
  \begin{gather}
    \label{eq:quandle distributivity}
    (x \qn y) \qn z = (x \qn z) \qn (y \qn z)
    \\
    \label{eq:quandle idempotency}
    x \qn x = x
  \end{gather}
\end{definition}

\begin{definition}
  Let \(D\) be a diagram of a tangle \(T\).
  A \defemph{decorated \(\slg\)-coloring} of \(D\) assigns each arc of \(D\) a decorated matrix subject to the rules
  \begin{equation}
    \label{eq:quandle coloring rules}
    \begin{tikzpicture}[line width=1, baseline=30, scale=1] 
      \draw[->] (0,0) node[left] {\(g_{2}\)} \br (1.5,1) ;
      \draw (1.5,1) node[right] { \(  g_{2} \qn g_{1} \) };
      \draw[white, line width=10] (0,1) node[left] {} \br (1.5,0) node[right] {};
      \draw[->] (0,1) node[left] {\(g_{1}\)} \br (1.5,0) ;
      \draw (1.5,0) node[right] {\(g_{1}\)};
    \end{tikzpicture}
    \quad \quad
    \begin{tikzpicture}[line width=1, baseline=30, scale=1] 
      \draw[->] (0,1) node[left] {\(g_{1} \)} \br (1.5,0) ;
      \draw (1.5,1) node[right] { \(  g_{2} \) };
      \draw[white, line width=10] (0,0) node[left] {} \br (1.5,1) node[right] {};
      \draw[->] (0,0) node[left] {\(g_{2}\)} \br (1.5,1) ;
      \draw (1.5,0) node[right] {\(g_{1} \qn^{-1} g_{2}\)};
    \end{tikzpicture}
  \end{equation}
  We write \(\repvar{D}\) for the set of decorated \(\slg\)-colorings of \(D\) and \(\repvar{T} = \repvar{D}\) for some possibly unspecified diagram of \(T\).
  This is well defined by the following lemma:
\end{definition}

\begin{lemma}
  For any two diagrams \(D, D'\) of the same tangle there is a bijection \(\repvar{D} \to \repvar{D'}\).
\end{lemma}

\begin{proof}
  Because \(\qn\) is invertible colorings are preserved by \(\reid{2}\) moves, \eqref{eq:quandle distributivity} corresponds to \(\reid{3}\) moves, and \eqref{eq:quandle idempotency} corresponds to \(\reid{1}\) moves.
  Some relevant diagrams occur in the proof of \cref{thm:cocycles give models}.
\end{proof}

\begin{definition}
  \label{def:decorated matrix log}
  A \defemph{log-decorated matrix} is a decorated matrix \(g = (g, [v])\) along with a choice of logarithm \(\mu\) of the distinguished eigenvalue of \(g\).
  Our convention is that  \(e^{\tu \mu}= m\).
  The operation \(\qn\) extends to log-decorated matrices in the obvious way (by preserving \(\mu\)).
  A \defemph{log-decorated \(\slg\)-coloring}%
  \note{
    In \cite{McPhailSnyderAlgebra} a log-decoration also includes a choice of logarithm \(\lambda\) of the distinguished longitudes defined in \cref{sec:Longitude eigenvalues}.
    Our invariants are normalized to remove the dependence on \(\lambda\), so we drop it from the definition.
  }
  of a tangle diagram  \(D\) is a coloring of \(D\) by log-decorated matrices satisfying \eqref{eq:quandle coloring rules}.
\end{definition}

If \(D\) has \(c\) components then the space \(\repvarlog{D}\) of log-decorated \(\slg\)-colorings of \(D\) is a \(\mathbb{Z}^{c}\)-covering of \(\repvar{D}\).
\begin{remark}
  \label{rem:cohomology class}
  A decorated representation \(\rho\) whose images lies in the diagonal subgroup of \(\slg\) is completely determined by its meridian eigenvalues, so we can identify it with the cohomology class 
  \[
    \sigma \in \cohomol{1}{\extr{L}; \mathbb{C}^{\times}}, \sigma([\mer_i]) = m_{i}
  \]
  where \([\mer_i]\) is the homology class of the meridian of component \(i\).
  These are the cohomology classes used by \cite{Costantino2012,Blanchet2016}.
  From this perspective a log-decoration of \(\rho\) is the same thing as a lift of \(\sigma\) to \(\cohomol{1}{\extr{L}; \mathbb{C}}\).
\end{remark}

\subsection{Shadow-colored diagrams and octahedral decompositions}
\label{sec:shadow colorings}

\begin{marginfigure}
  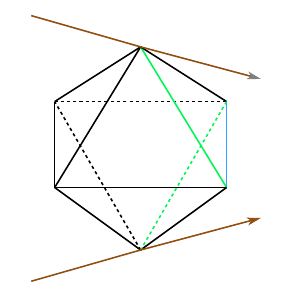
  \caption{A side view of the ideal octahedron at a positive crossing. The dashed edges indicate identifications of \(P_{\pm}\) with \(P_{\pm}'\).}
  \label{fig:octahedron-side}
\end{marginfigure}

One way to study hyperbolic structures on tangles is to use ideal triangulation of the tangle complement: this is a triangulation whose \(0\)-skeleton (which we think of as being removed) lies on the tangle.
Because our method for constructing invariants depends on a choice of diagram \(D\) it is convenient to use the \defemph{octahedral decomposition} associated to \(D\).%
\note{%
  Strictly speaking it is a semi-ideal triangulation because there are two material points; we can also think of it as a triangulation of the tangle exterior minus two points.
  \textcite{Kim2018} explain in detail why the usual geometric tools like developing maps still work in this case.
}
It assigns a (twisted) ideal octahedron to each crossing, as shown in \cref{fig:octahedron-side}.
We divide this octahedron into four tetrahedra, one for each corner of the crossing.

To put a hyperbolic structure on an ideal tetrahedron we can place its vertices on the boundary at infinity of hyperbolic space, i.e. \(\mathbb{C}P^1\).
Such geometric ideal tetrahedra are determined up to congruence by the cross-ratio of their vertices, usually called the \defemph{shape parameter}.%
\note{
  For details see \cite{McphailSnyder2022hyperbolicstructureslinkcomplements}.
  Note that we are using a different convention on negatively oriented tetrahedra.
}
Strictly speaking a shape parameter \(z^{0}\) is assigned to a pair of opposite edges, and the other edge pairs are assigned
\[
  z^{1} = \frac{1}{1-z^{0}} \text{ and } z^{2} = 1 - \frac{1}{z^{0}} = \frac{1}{1-z^{1}}.
\]

\textcite{Kim2018} give a detailed discussion of the geometry and combinatorics of the octahedral decomposition.
In particular, they explain how to solve Thurston's gluing equations in terms of parameters assigned to the segments and regions of the tangle diagram, which we call \defemph{octahedral coordinates}.

To discuss invariance of \(\qfunc{}\) under Reidemeister moves we need to consider the algebraic properties of the octahedral coordinates.
They correspond to a generalization of a quandle (itself a generalization of the conjugation structure of a group) called a \defemph{generically defined biquandle} by \cite{Blanchet2018}; similar structures were used in \cite{Kashaev2005} and studied systematically by \citeauthor{zbMATH06642281} \cite{zbMATH06642281}.
Using such colorings introduces some technical complications, so we prefer a different method.
In \cite{McphailSnyder2024octahedralcoordinateswirtingerpresentation} we show how to determine octahedral colorings using \(\slg\) colored tangle diagrams with an additional choice of gauge data called a shadow coloring.

\begin{definition}
  \label{def:shadow coloring}
  Write \(\shadset = \mathbb{C}^{2} \setminus \set{0}\).
  A \defemph{shadow coloring} of a decorated \(\slg\) colored diagram \(D\) is an assignment \(j \mapsto u_j\) of an element \(u_j \in \shadset\) to each region \(j\) subject to the rule in \cref{fig:shadow-rule}.
  Here we think of the shadow colors \(u_j\) as \emph{column} vectors.
\end{definition}

\begin{marginfigure}
  \begin{tikzpicture}[line width = 1, scale = 2]
    \draw[->] (0,0) to (1, 0) node[right] {\((g, [v])\)};
    \node[above] at (0.5,0.1) {\(u\)};
    \node[below] at (0.5,-0.1) {\(gu\)};
  \end{tikzpicture}
  \caption{Rule for shadow colorings}
  \label{fig:shadow-rule}
\end{marginfigure}

\begin{lemma}[\protect{\cite[Lemma 2.9]{McphailSnyder2024octahedralcoordinateswirtingerpresentation}}]
  \label{thm:shadow colorings one region}
  Every decorated \(\slg\)-colored diagram has a shadow coloring.
  For fixed \((D,\rho)\) the space of shadow colorings is parametrized by the color of a single region.
\end{lemma}

\begin{proof}[Proof idea]
  Once you pick the shadow of a single region the rule in \cref{fig:shadow-rule} determines the shadows of every other region.
\end{proof}

Using the lemma we identify a shadow coloring of a tangle diagram \((D, \rho)\) with the shadow of its topmost region and thus write \(\repvar{D} \times \shadset\) for the space of shadow-colored, decorated representations of \(\pi(D)\).

\begin{remark}
  \label{rem:shadow colorings}
  Any group \(G\) has a \defemph{conjugation quandle} \(\textsf{C}(G)\) defined by \(g \qn h \defeq h^{-1}g h\).
  If \(\textsf{Q}\) is a quandle, \(\varphi\) is a quandle homomorphism \(\varphi : \textsf{Q} \to \textsf{C}(G)\), and \(\textsf{X}\) is a set with left \(G\)-action then one can consider colorings of tangle diagrams by \((\mathsf{Q},\mathsf{X})\) as above: across a strand with color \(x\) the region color \(u\) becomes \(\phi(x) u\).
  Shadow colorings closely related to ours were used by \citeauthor{Inoue2014} to compute complex volumes of boundary-parabolic representations of links \cite{Inoue2014}.
  In \cref{sec:The Chern-Simons invariant of a tangle} we extend their work to general representations and tangles.
\end{remark}

\begin{definition}
  \label{def:chi parameters}
  Let 
  \(
  \evec{1} =
  \left[
  \begin{smallmatrix}
  1 \\ 0    
  \end{smallmatrix}
  \right]
  ,
  \evec{2}
  \left[
  \begin{smallmatrix}
  0 \\ 1    
  \end{smallmatrix}
  \right]
  \)
  be the standard basis of \(\mathbb{C}^{2}\) viewed as column vectors.
  For a decorated shadow coloring \((\rho, u)\) of a diagram \(D\) we define \defemph{region parameters} for each region \(j\) by
  \begin{equation}
    \label{eq:region parameter}
    a_j = \evec{1}^{\operatorname{T}} u = \det(u, \evec{2})
  \end{equation}
  \defemph{segment parameters} for each segment \(i\) below region \(i^{\up}\) by
  \begin{equation}
    \label{eq:segment parameter}
    b_i =
    -
    \frac{
      v_i \evec{2}
      }{
      v_i u_{i^{\up}}
    }
  \end{equation}
  and define the \defemph{meridian parameter} \(m_k\) of component \(k\) to be the distinguished eigenvalue of \cref{eq:distinguished eigenvalue}:
  \begin{equation}
    \label{eq:meridian parameter}
    v_i g_i = m_{k}^{-1} v_i
  \end{equation}
  where \(i\) is any segment belonging to component \(k\).
  A shadow coloring is \defemph{admissible} if the region and segment parameters are all elements of \(\mathbb{C} \setminus \set{0}\) (i.e.\@ are not \(0\) or \(\infty\)) and we write
  \[
    \admvar{D} \subset \repvar{D} \times \shadset
  \]
  for the space of admissible shadow colorings.
  Because admissibility is independent of the log-meridians it extends in the obvious way to shadow colorings of log-decorated tangles.
\end{definition}

\begin{marginfigure}
  \begin{tikzpicture}[line width = 1, scale = 2]
    \draw[->] (0,0) node[left] {\((g, [v])\)} to (1, 0);
    \node[above right] at (0,0.1) {\(u\)};
    \node[below right] at (0,-0.1) {\(gu\)};
    \node at (0.5, 0.5) { \( \det(u, \evec{2}) \) };
    \node at (0.5, -0.5) { \( \det(gu, \evec{2}) \) };
    \node[right] at (1, 0) { \( \dfrac{ v \evec{2} }{ v u} \) };
  \end{tikzpicture}
  \caption{Segment and region parameters}
  \label{fig:chi parameters example}
\end{marginfigure}

It is possible for some \(\rho \in \repvar{D}\) to not correspond to any admissible shadow coloring; for example, this will be the case if some distinguished eigenvector \(v_i\) lies in \(\operatorname{Span}\{ (1,0)\}\).
Despite this the conjugation orbit \(\set{ g^{-1} \rho g \given g \in \slg}\) of \(\rho\) will always contain a decorated representation with an admissible shadow coloring.
We discuss this further in \cref{sec:Equivalence of tangles and diagrams,def:gauge transformations}.
Admissibility is a relatively weak condition: it still allows geometrically degenerate tetrahedra, but in a controlled way.

\begin{definition}
  \label{def:pinched crossing}
  Let \(\rho \in \repvar{D}\) be a decorated representation of a tangle diagram \(D\).
  A crossing of the colored diagram \((D, \rho)\) is \defemph{pinched} if \([v_1] = [v_2]\), that is if the incoming segments have the same distinguished eigenspace.
\end{definition}

Let \(X\) be an admissibly shadow-colored crossing of sign \(\epsilon\).
Define
\begin{align}
  \label{eq:shape-N}
  z_{\lN}^0
  &\defeq
  \left( \frac{b_{4}}{b_{1}} \right)^{\epsilon}
  &
  z_{\lN}^1
  &\defeq
  \frac{K_{X}}{a_{\lN}}
  \\
  \label{eq:shape-W}
  z_{\lW}^0
  &\defeq
  \left( \frac{b_{2}}{b_{1} m_{1}} \right)^{\epsilon}
  &
  z_{\lW}^1
  &\defeq
  \frac{K_{X} m_{1}^{\epsilon}}{a_{\lW}} 
  \\
  \label{eq:shape-S}
  z_{\lS}^0
  &\defeq
  \left( \frac{b_{2} m_{2} }{b_{3} m_{1}} \right)^{\epsilon}
  &
  z_{\lS}^1
  &\defeq
  \frac{K_{X} m_{1}^{\epsilon} }{a_{\lS} m_{2}^{\epsilon} }
  \\
  \label{eq:shape-E}
  z_{\lE}^0
  &\defeq
  \left( \frac{b_{4} m_{2} }{b_{3} } \right)^{\epsilon}
  &
  z_{\lE}^1
  &\defeq
  \frac{K_{X}}{a_{\lE} m_{2}^{\epsilon} }
\end{align}
where \(\epsilon \in \set{1, -1}\) is the sign of the crossing and
\(
  K_{X}=
  a_{\lN}/(1 - \left(b_{4}/b_{1}\right)^{\epsilon}).
\)
Using \eqref{eq:segment parameter} it is easy to see that \(X\) is pinched if in some (equivalently, all) shadow coloring of \((D, \rho)\) some (equivalently, all) of the shape parameters \(z_{i}^{0}\) are equal to \(1\).
In particular if \(X\) is not pinched \(K_{X}\) and the \(z_{j}^{1}\) are well-defined.

\begin{theorem}[\cite{McphailSnyder2022hyperbolicstructureslinkcomplements,McphailSnyder2024octahedralcoordinateswirtingerpresentation}]
  \label{thm:shape parameters}
  Let \((\rho, u)\) be an admissible shadow coloring of \(D\).
  Suppose that no crossing of \(D\) is pinched.
  Then the shape parameters of (\ref{eq:shape-N}--\ref{eq:shape-E}) give a solution of Thurston's gluing equations for the octahedral decomposition associated to \(D\).
  Furthermore the holonomy representation induced by this solution agrees with \(\rho\).
\end{theorem}

\begin{proof}
  In \cite{McphailSnyder2022hyperbolicstructureslinkcomplements} we define an \defemph{octahedral coloring} of a link (or tangle) diagram to be an assignment of parameters \(b_{i}, a_{j}, m_{k}\) satisfying certain equations and show that any octahedral coloring gives a solution of the gluing equations (at least whenever the coloring is not pinched so the equations make sense).
  We then give a formula for the formula for the holonomy \(\rho\) of the hyperbolic structure determined by the gluing.
  In \cite{McphailSnyder2024octahedralcoordinateswirtingerpresentation} we work in the opposite direction: given \(\rho\) and a choice of admissible shadow coloring we prove that the formulas of \cref{def:chi parameters} determine an octahedral coloring \cite[Theorem 2.11]{McphailSnyder2024octahedralcoordinateswirtingerpresentation} and that the induced holonomy representation recovers \(\rho\) \cite[Theorem 1]{McphailSnyder2024octahedralcoordinateswirtingerpresentation}.
\end{proof}

\begin{corollary}
  \label{thm:segment relations}
  At any (non-pinched) crossing of sign \(\epsilon\) the region and segment parameters are related by 
  \begin{align*}
    \frac{
      a_{\lW}
      }{
      a_{\lN}
    }
    m_{1}^{-\epsilon}
    &=
    \frac{
      1 - z_{\lW}^{0}
      }{
      1 - z_{\lN}^{0}
    }
    &
    \frac{
      a_{\lS}
      }{
      a_{\lE}
    }
    m_{1}^{-\epsilon}
    &=
    \frac{
      1 - z_{\lS}^{0}
      }{
      1 - z_{\lE}^{0}
    }
    \\
    \frac{
      a_{\lE}
      }{
      a_{\lN}
    }
    m_{2}^{\epsilon}
    &=
    \frac{
      1 - z_{\lE}^{0}
      }{
      1 - z_{\lN}^{0}
    }
    &
    \frac{
      a_{\lS}
      }{
      a_{\lW}
    }
    m_{2}^{\epsilon}
    &=
    \frac{
      1 - z_{\lS}^{0}
      }{
      1 - z_{\lW}^{0}
    }
  \end{align*}
  so for example
  \[
    \frac{
      a_{\lW}
      }{
      a_{\lN}
    }
    m_{1}^{-\epsilon}
    =
    \frac{
      1 - (b_{2} / b_{1} m_{1})^{\epsilon}
      }{
      1 - (b_{4} / b_{2})^{\epsilon}
    }
    .
    \qedhere
  \]
\end{corollary}

\begin{proof}
  Part of the conclusion of \cref{thm:shape parameters} is that 
  \[
    z_{j}^{1} = \frac{1}{1 - z_{j}^{0}} 
  \]
  for each corner \(j\).
  Taking ratios of these relations gives the claim.

  Alternatively, observe that for \(\epsilon = 1\) the claim is a relation between explicit rational functions defined by \cref{eq:region parameter,eq:segment parameter} on the set \(\admvar{X}\) of admissible shadow colorings of a positive crossing \(X\).
  It can therefore be checked by writing things out in coordinates and doing some algebra.
  The same argument works for \(\epsilon = -1\).
\end{proof}

The papers \cite{McphailSnyder2022hyperbolicstructureslinkcomplements, McphailSnyder2024octahedralcoordinateswirtingerpresentation} establish a connection between the braiding on quantum \(\sla\) at a root of unity and the octahedral coordinates.
This connection was developed to understand the holonomy \(R\)-matrices of \cite{McPhailSnyderAlgebra}.
In this context we think of the segment and region parameters as defining central characters of the Weyl algebra \(\weyl\) (\cref{sec:Modules}).
The outer \(R\)-matrix \(\mathcal{R}\) (\cref{sec:The braiding}) acts on these characters, and the lemma below says that the action agrees with the parameters:

\begin{lemma}[\protect{\cite[Theorem 2.11]{McphailSnyder2024octahedralcoordinateswirtingerpresentation}}]
  \label{thm:crossing parameter relations}
  Let \(X\) be a shadow-colored, decorated \(\slg\)-colored crossing.
  If it is positive the segment and region parameters of \(X\) are related by
  \begin{align}
      \label{eq:b-transf-positive-1}
      b_{3}
      &=
      \frac{m_2 b_2}{m_1}
      \left(
        1 - m_2 a_2 \left( 1 - \frac{b_2}{m_1 b_1} \right)
      \right)^{-1}
      \\
      \label{eq:b-transf-positive-2}
      b_{4}
      &=
      b_1
      \left(
        1 - \frac{m_1}{a_1}\left( 1 - \frac{b_2}{m_1 b_1} \right)
      \right)
      \\
    \label{eq:a-transf-positive}
      a_{\lE}
      &=
      \frac{1}{a_{\lW}} \left(
      a_{\lN} a_{\lS} - \frac{m_1 b_1}{b_2} \left(a_{\lN} - \frac{a_{\lW}}{  m_1}\right)\left(a_{\lS} - \frac{a_{\lW}}{m_2 }\right)
    \right)
  \end{align}
  while if it is negative they are related by
  \begin{align}
  \label{eq:b-transf-negative-1}
    b_{3}
      &=
      \frac{m_2 b_2}{m_1}
      \left(
        1 - \frac{a_2}{m_2} \left( 1 - \frac{m_1 b_1}{b_2} \right)
      \right)
      \\
  \label{eq:b-transf-negative-2}
      b_{4}
      &=
      b_1
      \left(
        1 - \frac{1}{m_1 a_1} \left( 1 - \frac{m_1 b_1}{b_2} \right)
      \right)^{-1}
      \\
    \label{eq:a-transf-negative}
    a_{\lE} &=
    \frac{
      1  
      }{
      a_{\lW}
    }
    \left(
    a_{\lN} a_{\lS} - \frac{b_2}{m_1 b_1}\left(a_{\lN} - m_1 a_{\lW} \right) \left( a_{\lS} - m_{2} a_{\lW} \right)
  \right)
  .
  \end{align}
\end{lemma}

When a shape parameter is \(0\), \(1\), or \(\infty\) it represents a geometrically degenerate tetrahedron whose vertices coincide.
The admissibility condition ensures that each \(z_{i}^{0}\) is not \(0\) or \(\infty\), but pinched with \(z_{i}^{0} = 1\) and \(z_{i}^{1} = \infty\) are possible.
Allowing these is useful both for technical reasons and to understand \(\qfunc{}\) in the geometrically degenerate limit.
When constructing quantum hyperbolic invariants (and when computing the Chern-Simons invariant) one frequently encounters functions%
\note{
  See \cref{sec:quantum dilogarithms} for the examples used in this paper.
}
\(f : \mathbb{C} \times \mathbb{Z} \to \mathbb{C}\) obeying a quasi-periodicity relation
\[
  f(z|k+1) = f(z|k) (1 - z)^{-1}.
\]
Each function \(f\) is assigned to a single tetrahedron, so in the degenerate limit \(z \to 1\) we can encounter poles.
In the octahedral decomposition such functions always occur in groups \(f(z_{3}) f(z_{4})/f(z_{1}) f(z_{2})\) and the \(z_{i}\) go to \(1\) simultaneously so the poles cancel.
To do this properly we need to keep track of \emph{how} the poles cancel, i.e.\ we need to blow up along the singular subvariety.
We can view \cref{thm:segment relations} as saying that the region variables \(a_{j}\) give a blowup of the \(b\)-gluing variety \(\mathfrak{B}_{D}\) \cite[Section 5]{McphailSnyder2022hyperbolicstructureslinkcomplements} along the subvariety of pinched colorings of \(D\).

While pinched crossings are geometrically degenerate both the classical and quantum Chern-Simons invariants \(\csfunc{}\) and \(\qfunc{}\) are still well-defined there: the poles cancel as above.
The contribution of a pinched crossing to the complex Chern-Simons invariant is trivial, but this is not the case for the quantum invariant.
For example, if \(X\) is a crossing with \(g_{1} = g_{2} = (-1)^{\nr -1} \in \slg\) then \(\qfunc{X}\) recovers Kashaev's \(R\)-matrix \cite{Kashaev1995} (\cref{thm:pinched braiding is defined}).

\subsection{Equivalence of tangles via diagrams}%
\label{sec:Equivalence of tangles and diagrams}

Let \(D\) be a colored diagram.
The rules defining shadow colorings are compatible with (framed, oriented) Reidemeister moves in the sense that applying a move to \(D\) uniquely defines the coloring of the new diagram \(D'\).
However, this coloring may not be admissible, and here we explain how to work around this.

\begin{figure}
  \centering
  \begin{subfigure}[t]{\textwidth}
    \centering
    \begin{tikzpicture}[line width = 1, scale = 0.5, xscale = 1.5, baseline={(current bounding box.center), line cap = rect}]
      \draw[->] (0,1) to [out = 240, in = 180] (1.5,0);
      \draw[white, line width=10] (-0.5,0) to [out = 00, in = -60] (1,1);
      \draw[-] (-0.5,0) to [out = 00, in = -60] (1,1);
      \draw[-] (1,1) to [out = 120, in = 60] (0,1);
      \begin{scope}[shift={(2,0)}]
        \draw[-] (-0.5,0) to [out = 00, in = -60] (1,1);
        \draw[-] (1,1) to [out = 120, in = 60] (0,1);
        \draw[white, line width=10] (0,1) to [out = 240, in = 180] (1.5,0);
        \draw[->] (0,1) to [out = 240, in = 180] (1.5,0);
      \end{scope}
    \end{tikzpicture}
    \caption{\(\reid{1f}\)}
  \end{subfigure}

  \begin{subfigure}[t]{0.3\textwidth}
    \centering
    \begin{tikzpicture}[line width=1, scale=1, xscale = 1.5, baseline={(current bounding box.center)}, line cap = rect]
      \tikzBraiding{+}{(0, 1)}{(0, 0)}{(1, 0)}{(1, 1)}
      \tikzBraidingEast{-}{(1, 1)}{(1, 0)}{(2, 0)}{(2, 1)}
    \end{tikzpicture}
    \caption{\(\reid{2a}\)}
  \end{subfigure}
  \begin{subfigure}[t]{0.3\textwidth}
    \centering
    \begin{tikzpicture}[line width=1, scale=1, xscale = 1.5, baseline={(current bounding box.center)}, line cap = rect]
      \tikzBraiding{-}{(0, 1)}{(0, 0)}{(1, 0)}{(1, 1)}
      \tikzBraidingEast{+}{(1, 1)}{(1, 0)}{(2, 0)}{(2, 1)}
    \end{tikzpicture}
    \caption{\(\reid{2b}\)}
  \end{subfigure}
  \\
  \begin{subfigure}[t]{0.3\textwidth}
    \centering
    \begin{tikzpicture}[line width=1, scale=1, xscale = 1.5, baseline={(current bounding box.center)}, line cap = rect] 
      \draw[-] (0,1) \br (1,0) ;
      \draw[white, line width=10, line cap=butt] (0,0) \br (1,1) ;
      \draw[<-] (0,0) \br (1,1) ;
      \draw[->] (1,0)  \br (2,1);
      \draw[white, line width=10, line cap=butt] (1,1)  \br (2,0);
      \draw[-] (1,1)  \br (2,0);
    \end{tikzpicture}
    \caption{\(\reid{2}[+-]\)}
  \end{subfigure}
  \begin{subfigure}[t]{0.3\textwidth}
    \centering
    \begin{tikzpicture}[line width=1, scale=1, xscale = 1.5, baseline={(current bounding box.center)}, line cap = rect] 
      \draw[<-] (0,1) \br (1,0) ;
      \draw[white, line width=10, line cap=butt] (0,0) \br (1,1) ;
      \draw[-] (0,0) \br (1,1) ;
      \draw[-] (1,0)  \br (2,1);
      \draw[white, line width=10, line cap=butt] (1,1)  \br (2,0);
      \draw[->] (1,1)  \br (2,0);
    \end{tikzpicture}
    \caption{\(\reid{2}[-+]\)}
  \end{subfigure}
  \begin{subfigure}[t]{\textwidth}
    \centering
    \begin{tikzpicture}[line width=1, scale=1, xscale = 1.5, baseline={(current bounding box.center)}, line cap = rect]
      \tikzBraiding{+}{(0, 2)}{(0, 1)}{(1, 1)}{(1, 2)}
      \tikzBraiding{+}{(1, 1)}{(1, 0)}{(2, 0)}{(2, 1)}
      \tikzBraiding{+}{(2, 2)}{(2, 1)}{(3, 1)}{(3, 2)}
      \draw[-] (1,2) to (2,2);
      \draw[-] (0,0) to (1,0);
      \draw[-] (2,0) to (3,0);
      \tikzBraiding{-}{(3, 2)}{(3, 1)}{(4, 1)}{(4, 2)}
      \tikzBraiding{-}{(4, 1)}{(4, 0)}{(5, 0)}{(5, 1)}
      \tikzBraidingEast{-}{(5, 2)}{(5, 1)}{(6, 1)}{(6, 2)}
      \draw[-] (4,2) to (5,2);
      \draw[-] (3,0) to (4,0);
      \draw[->] (5,0) to (6,0);
    \end{tikzpicture}
    \caption{\(\reid{3}\)}
    \label{fig:reid moves that vanish:3}
  \end{subfigure}
  \caption{
    We call replacing an appropriate number of parallel unlinked strands with one of the diagrams above (or vice-versa) a colored Reidemeister move.
  }
  \label{fig:reid moves that vanish}
\end{figure}

\begin{definition}
  \label{def:R move log-decorated}
  Let \(D\)  and \(D'\) be colored diagrams.
  We say they are related by a \defemph{colored Reidemeister move} if \(D'\) is obtained from \(D\) by replacing a set of parallel strands by one of the diagrams in \cref{fig:reid moves that vanish} or vice-versa.
  We say the move is  \defemph{admissible} if the shadow colorings of both diagrams are admissible.
\end{definition}

This description of  the \(\reid{3}\) move is more convenient than the standard one because it is easier to work with endomorphisms than general morphisms.
For example, the determinant of a vector space endomorphism can be canonically identified with a scalar (i.e.\ is basis-independent).
This is used in the proof of \(\reid{3}\) invariance for both \(\csfunc{}\) and \(\qfunc{}\) along with the following technical lemma:

\begin{lemma}
  \label{thm:reid3 is connected}
  \(\admvar{\reid{3}}\) is path-connected.
\end{lemma}

\begin{proof}
  The set \(\matset \subset \slg \times \pjsp\) of decorated matrices is a complex variety.
  If we include the degenerate shadow coloring with all colors \(0\) the space \(Y\) of all shadow colorings of \(D_{3}\) is \(\matset^{3} \times \mathbb{C}^{2}\).
  Clearly \(Y\) is path-connected.
  The set of \emph{admissible} colorings \(\admvar{\reid{3}}\) is obtained by removing finitely many complex hypersurfaces from \(Y\), so it is still path-connected.\note{A complex hypersurface has \emph{real} codimension \(2\).}
\end{proof}

\begin{definition}
  \label{def:gauge transformations}
  Let \(D\) be a colored diagram with arc colors \(i \mapsto (g_{i}, [v_{i}])\) and region colors \(j \mapsto u_{j}\).
  A \defemph{gauge transformation} of it is a coloring of the form
  \begin{itemize}
    \item[(A)] \(i \mapsto (h^{-1} g_i h, [v_{i} h])\) and \(j \mapsto h^{-1} u_j\)
    \item[(B)] \(i \mapsto (g_i, [v_{i}])\) and \(j \mapsto h^{-1} u_j\)
  \end{itemize}
  for some \(h \in \slg\).
  We say the new colorings are the images of \defemph{type (A)} and  \defemph{type (B)} gauge transformations  \defemph{by \(h\)}.
\end{definition}
It is not hard to see that the gauge transformation of a coloring is still a coloring and that the underlying representation \(\rho : \pi(D) \to \slg\) is conjugated under a type (A) transformation and unchanged under type (B).
Any two shadow colorings with the same underlying representation are (B) gauge-equivalent.

\begin{figure}
  \centering
  \subcaptionbox{ 
    \label{fig:gauge-transformation-A} \(D'\) is obtained from \(D\) by a type (A) gauge transformation by \(h\).
    }{
    \begin{tikzpicture}[line width=1, scale=2] 
      \begin{scope}[baseline = (D)]
        \node[draw,name = D, minimum height = 1cm, minimum width = 1.5cm] at (1,0) {\(D\)};
        \draw[line width = 3] (0,0) to (D.west);
        \draw[line width = 3] (D.east) to (2,0);
        \draw[->] (0,0.5) to (2, 0.5);
        \node[right] at (2, 0.5) {\(h\)};
      \end{scope}
      \draw[->, color = accent] (2.25, 0)  to (2.75,0);
      \begin{scope}[baseline = (D), shift = {(3, 0)}]
        \node[draw,name = D, minimum height = 1cm, minimum width = 1.5cm] at (1,0) {\(D'\)};
        \draw[line width = 3] (0,0) to (D.west);
        \draw[line width = 3] (D.east) to (2,0);
        \draw[white, line width = 6] (0,0.5) .. controls (0.5, 0.5) and (0.5, -1) .. (1, -1) .. controls (1.5, -1) and (1.5, 0.5) .. (2,0.5);
        \draw[->] (0,0.5) .. controls (0.5, 0.5) and (0.5, -1) .. (1, -1) .. controls (1.5, -1) and (1.5, 0.5) .. (2,0.5);
        \node[right] at (2, 0.5) {\(h\)};
      \end{scope}
    \end{tikzpicture}
    }
  \subcaptionbox{
    \label{fig:gauge-transformation-B} \(D''\) is obtained from \(D\) by a type (B) gauge transformation by \(h\).
  }{
    \begin{tikzpicture}[line width=1, scale=2] 
      \begin{scope}[baseline = (D)]
        \node[draw,name = D, minimum height = 1cm, minimum width = 1.5cm] at (1,0) {\(D\)};
        \draw[line width = 3] (0,0) to (D.west);
        \draw[line width = 3] (D.east) to (2,0);
        \draw[->] (0,0.5) to (2, 0.5);
        \node[right] at (2, 0.5) {\(h\)};
      \end{scope}
      \draw[->, color = accent] (2.25, 0)  to (2.75,0);
      \begin{scope}[baseline = (D), shift = {(3, 0)}]
        \draw[->] (0,0.5) .. controls (0.5, 0.5) and (0.5, -1) .. (1, -1) .. controls (1.5, -1) and (1.5, 0.5) .. (2,0.5);
        \draw[white, line width = 10] (0,0) to (2,0);
        \node[draw,name = D, minimum height = 1cm, minimum width = 1.5cm] at (1,0) {\(D''\)};
        \draw[line width = 3] (0,0) to (D.west);
        \draw[line width = 3] (D.east) to (2,0);
        \node[right] at (2, 0.5) {\(h\)};
      \end{scope}
    \end{tikzpicture}
  }
  \caption{
    Graphical description of gauge transformations.
    The thick strands represent bundles of incoming and outgoing segments (possibly empty) with arbitrary colorings.
  }
  \label{fig:gauge-transformation}
\end{figure}

The letters A and B are chosen because of the picture in \cref{fig:gauge-transformation}: a type (A) transformation comes from pulling a strand colored by \(h\) across and \textbf{A}bove the diagram, while for type (B) we pull across and \textbf{B}elow.
These are quite similar to the diagrammatic gauge transformations of \textcite[Section 3.3]{Blanchet2018}.

\begin{definition}
  \label{def:equivalent tangle diagrams}
  Two colored diagrams \(D, D'\) are \defemph{colored isotopic} if
  \begin{enumerate}
    \item their underlying tangles are framed oriented isotopic,
    \item the colorings give the same log-decorated representation of the underlying tangle, and
    \item they have the same shadow coloring of their topmost region (hence of every region).
  \end{enumerate}
\end{definition}

It is obvious that diagrams related by a sequence of admissible colored Reidemeister moves are colored isotopic.
Unlike the usual situation in knot theory the converse may not hold.
The problem is that even if \(D\) and \(D'\) are admissible every sequence of Reidemeister moves between them might involve an inadmissible diagram.
Following \textcite{Blanchet2018} we overcome this problem by extending our allowed diagram moves to include gauge transformations.
Because a gauge transformation by \(h\) can be represented by pulling a strand colored by \(h\) across a diagram (a sequence of \(\reid{3}\) moves) so we simply need to allow adding extra untangled strands.

\begin{definition}
  Let \(D\) be a colored tangle diagram.
  A \defemph{stabilization} of \(D\) is a disjoint union of \(D\) with a single strand \(I\) with an arbitrary admissible coloring, as in \cref{fig:stabilizations}.
  A \defemph{destabilization} is the inverse operation that removes a single unlinked strand from the top or bottom of a diagram.
\end{definition}

\begin{figure}
  \centering
  \subcaptionbox{
    \label{fig:stabilizations-top}
    \(I \du D\)
    }{
    \begin{tikzpicture}[line width=1, scale=2, baseline = (D)] 
      \node[draw,name = D, minimum height = 1cm, minimum width = 2cm] at (1,0) {\(D\)};
      \draw[line width = 3] (0,0) to (D.west);
      \draw[line width = 3] (D.east) to (2,0);
      \draw[->] (0,0.5) to (2,0.5);
    \end{tikzpicture}
  }
  \subcaptionbox{
    \label{fig:stabilizations-bottom}
    \(D \du I\)
    }{
    \begin{tikzpicture}[line width=1, scale=2, baseline = (D)] 
      \node[draw,name = D, minimum height = 1cm, minimum width = 2cm] at (1,0.5) {\(D\)};
      \draw[line width = 3] (0,0.5) to (D.west);
      \draw[line width = 3] (D.east) to (2,0.5);
      \draw[->] (0,0) to (2,0);
    \end{tikzpicture}
  }
  \caption{ Stabilizations of a diagram \(D\). The thick strands represent collections of incoming and outgoing segments.}
  \label{fig:stabilizations}
\end{figure}

\begin{theorem}
  \label{thm:diagrams linked by stabilization and R}
  Any colored isotopic \(\slg\)-colored log-decorated tangle diagrams \(D\) and \(D'\) are related by a stabilization, a sequence of admissible Reidemeister moves, and a destabilization.
\end{theorem}

\begin{proof}
  It suffices to prove a version of \cite[Theorem 5.4]{Blanchet2018}: if \(D\) and \(D'\) are colored isotopic then there is a common stabilization so that \(I \du D\) and \(I \du D'\) are lined by a sequence of admissible colored Reidemeister moves.

  To prove this we need to find a color to assign to \(I\).
  Parametrize possible colors by the set
  \[
    Y = \set{ (g, v) \given g \in \slg, v \in \mathbb{C}^{2}, vg = m v  \text{ for some } m \in  \mathbb{C} }.
  \]
  It is clearly an algebraic set in some \(\mathbb{C}^{n}\), but certain elements of \(Y\) are inadmissible.
  For example, we can exclude
  \[
    Z = \set{(g,v) \in Y \given v = (0,0)}
  \]
  since we do not allow \((0,0)\) as a region color.
  There are also excluded sets coming from the isotopy.
  We show that in total there are finitely many excluded sets \(W, W', W''\), each of which is a union of finitely many hypersurfaces.
  We can then choose
  \[
    h \in Y \setminus \left(  Z \cup W \cup W' \cup W'' \right).
  \]
  This set is nonempty because \(Y\) is not a finite union of hypersurfaces.

  Consider the diagram \(\widetilde{D}\) obtained by pulling \(I\) across \(I \du D\) as in a type (A) gauge transformation.
  Choose a sequence of Reidemeister moves
  \[
    I \du D =
    D_{0}
    \xrightarrow{\Omega_{1}}
    D_{1}
    \xrightarrow{\Omega_{2}}
    \cdots
    \xrightarrow{\Omega_{n}}
    D_{n}
    =
    \widetilde{D}
  \]
  linking \(I \du D \) and \(\widetilde{D}\).
  For each \(\Omega_{i}\) let \(X_{i}\) be the subset of points \(h \in Y\) where  the move \(\Omega_{i}\) is inadmissible.
  Set \(W = \bigcup_{i} X_{i}\).
  It is clear that \(W\) is a finite union of hypersurfaces, and if \(h \not \in W\) (or \(Z\)) then we can gauge-transform \(D\) by \(h\).

  There is a similar set \(W'\) from connecting \(I \du D'\) to \(\widetilde{D}'\) via some chosen sequence of moves.
  Finally there is a set \(W''\) coming from the isotopy between \(\widetilde{D}\) and \(\widetilde{D}'\).
  It is clear each is a finite union of hypersurfaces.
\end{proof}

\subsection{Longitude eigenvalues}%
\label{sec:Longitude eigenvalues}

\begin{figure}
  \centering
  \begin{tikzpicture}[line width=1, baseline=30, scale=1, xscale = 1.5] 
    \coordinate (1) at (0,1);
    \coordinate (2) at (0,0);
    \coordinate (3) at (1,0);
    \coordinate (4) at (1,1);
    \draw[->] (2) node[left] {\(+\)} \br (4) node[right] {\(-\)};
    \draw[white, line width=10] (1) \br (3);
    \draw[->] (1) node[left] {\(-\)} \br (3) node[right] {\(+\)};
  \end{tikzpicture}
  \quad \quad
  \begin{tikzpicture}[line width=1, baseline=30, scale=1, xscale = 1.5] 
    \coordinate (1) at (0,1);
    \coordinate (2) at (0,0);
    \coordinate (3) at (1,0);
    \coordinate (4) at (1,1);
    \draw[->] (1) node[left] {\(+\)} \br (3) node[right] {\(-\)};
    \draw[white, line width=10] (2) \br (4);
    \draw[->] (2) node[left] {\(-\)} \br (4) node[right] {\(+\)};
  \end{tikzpicture}
  \caption{Signs for the contributions of the segment parameters to the squared longtidue.}
  \label{fig:crossing-log-dec-signs}
\end{figure}

A framed oriented link \(L\) has meridians \(\mer_{i}\) and longitudes \(\lon_{i}\) for each component \(i\).
A decorated representation assigns eigenvalues \(m_{i}\) to each \(\rho(\mer_{i})\), and because representatives of \(\mer_{i}\) and \(\lon_{i}\) from the same peripheral subgroup commute it also  determines longitude eigenvalues \(\ell_{i}\).
For tangles it is less clear how to define longitudes without passing to the fundamental groupoid.
However, for shadow-colored tangle diagrams we can still assign a number \(\ell_{i}^{2}\) to each component that behaves like the square of the longitude eigenvalue, as in \cref{thm:qfunc defined and properties}\ref{thm:qfunc defined and properties:log-decoration}.

Let \(D\) be an admissible colored diagram.
Each segment \(k\) of \(D\) meets two, one or zero crossings of \(D\) depending on whether it is a boundary or internal segment.
Let \(\epsilon_{k}\) be the sign \(\pm 1\) determined by \cref{fig:crossing-log-dec-signs} if \(k\) begins at a crossing and \(0\) if it begins at the boundary of \(D\).
Similarly, let \(\epsilon_{k}'\) be the sign (or \(0\)) determined by the endpoint of \(k\).
Write \(a_{k\dn}\) for the region parameter lying below \(k\).
Observe that \(|\epsilon_{k}'| - |\epsilon_{k}'|\)  is \(-1\)  if \(k\) starts at the boundary of \(D\), \(1\) if it ends at the boundary, and is \(0\) if both or neither.
\begin{definition}
  \label{def:longitude eigenvalues}
  For each component \(i\) of \(D\) let \(S_{i}\) be the set of diagram segments in component \(D\).
  The \defemph{squared longitude} of the \(i\)th component is
  \[
    \ell_{i}^{2}(D)
    \defeq
    \prod_{k \in S_{i}} b_{k}^{\epsilon_{k} + \epsilon_{k}'} a_{k \dn}^{|\epsilon_{k}'| - |\epsilon_{k}'|}
  \]
\end{definition}

\begin{example}
  \label{ex:longitude}
  Abbreviating \(g_{i} = (g_{i}, [v_{i}])\), for
  \begin{equation*}
    D=
    \begin{tikzpicture}[line width=1, scale=0.75, xscale = 1.5, baseline={(current bounding box.center)}] 
      \coordinate (1) at (0,1);
      \coordinate (2) at (0,0);
      \coordinate (3) at (1,0);
      \coordinate (4) at (1,1);
      \tikzBraidingEast{+}{(1)}{(2)}{(3)}{(4)}
      \node[left] at (1) {\(g_{1}\)};
      \node[left] at (2) {\(g_{2}\)};
      \node[right] at (3) {\(g_{1}\)};
      \node[right] at (4) {\(g_{2} \qn g_{1}\)};
      \node[above] at (0.5, 1) {\(u\)};
    \end{tikzpicture}
  \end{equation*}
  we have
  \begin{align*}
    \ell_{1}^{2}(D)
    &=
    \frac{
      1
      }{
      b_{1} a_{\lW}
    }
    b_{3} a_{\lS}
    =
    \frac{
      v_{1} u 
      }{
      v_{1} g_{1}^{-1} g_{2} g_{1} u 
    }
    \frac{
      \det( g_{2} g_{1} u, \evec{2})
      }{
      \det( g_{1} u, \evec{2})
    }
    \\
    \ell_{2}^{2}(D)
    &=
    \frac{
      b_{2}
      }{
      a_{\lS} 
    }
    \frac{
      a_{\lE}
      }{
      b_{4}
    }
    =
    \frac{
      v_{2} u 
      }{
      v_{2} g_{1} u
    }
    \frac{
      \det( g_{1}^{-1} g_{2} g_{1} u, \evec{2})
      }{
      \det( g_{2} g_{1} u, \evec{2})
    }
  \end{align*}
\end{example}

\begin{lemma}
  \label{thm:longitude eigenvalues make sense}
  \begin{thmenum}
    \item
      \label{thm:longitude eigenvalues make sense:preserved}

      The squared longitudes are preserved by admissible Reidemeister moves.
    \item
      \label{thm:longitude eigenvalues make sense:closed}
      If \(i\) is a closed component of \(D\) then \(\ell^{2}_{i}\) has a well-defined square root \(\ell_{i}\) given by
      \begin{equation}
        \label{eq:longitude-closed-rule}
        \ell_{i}(D) = 
        \prod_{k} b_{k}^{\eta_{k}}
        \text{ where }
        \eta_{k}
        \defeq
        \begin{cases}
          1 & \text{if segment \(k\) is over-under,}
          \\
          -1 & \text{if it is under-over, and}
          \\
          0 & \text{otherwise}
        \end{cases}
      \end{equation}
      and  \(\ell_{i}(D)\) is the eigenvalue of the blackboard-framed longitude distinguished by the decoration.
      \qedhere
  \end{thmenum}
\end{lemma}

\begin{proof}
  \ref{thm:longitude eigenvalues make sense:preserved}
  It suffices to check \(\ell_{i}^{2}(R) = 1\) for every diagram \(R\) in \cref{fig:reid moves that vanish}.
  This is trivial for the \(\reid{1f}\) and  \(\reid{2}\) diagrams and elementary for \(\reid{3}\): the claim is an equality about specific rational functions on \(\admvar{\reid{3}}\), so we can write these out in coordinates and check.

  \ref{thm:longitude eigenvalues make sense:closed}
  If \(i\) does not intersect the boundary it is easy to see that \(\epsilon_{k} + \epsilon_{k}' = 2 \eta_{k}\) and \(|\epsilon_{k}'| - |\epsilon_{k}'| = 0\), so the square of \eqref{eq:longitude-closed-rule} agrees with \(\ell_{i}^{2}\).

  In \cite[Theorem 4.4]{McphailSnyder2022hyperbolicstructureslinkcomplements} the author used the geometry of the octahedral decomposition to show that \(\ell_{i}\) is an eigenvalue of the blackboard-framed longitude.
  More precisely, there is a detailed proof that \(\ell_{i}^{2}\) is a square of the longitude eigenvalue and a sketch of how to check the sign by extending known results in the boundary-parabolic case.
  We can close this gap using \cref{eq:segment parameter}.

  Fix a segment \(k\) of component \(i\).
  Let \(x_{k}\) be the corresponding Wirtinger generator and \(y_{k}\) the blackboard-framed longitude commuting with \(x_{k}\).
  Following \(i\) around \(D\) starting at \(k\) we cross under arcs with Wirtinger generators \(x_{j_{1}}, \dots, x_{j_{n}}\).
  By a standard result in knot theory
  \[
    y_{k}
    =
    x_{j_{1}}^{\sigma_{1}} \dots x_{j_{n}}^{\sigma_{n}}
  \]
  where the \(\sigma_{\nu}\) are the signs of the crossings.
  For example, the longitude of the figure eight knot in \cref{fig:figure-eight-arcs} is given in \cref{eq:figure-eight longitude}.
  By definition
  \[
    v_{k} \rho(x_{k}) = m_{i}^{-1} v_{k}
    \text{ and }
    v_{k} \rho(y_{k}) = \widetilde{\ell}_{i}^{-1} v_{k}
  \]
  temporarily writing \(\widetilde{\ell}_{i}\) for the actual longitude eigenvalue;
  our claim is that \(\tilde{\ell}_{i} = \ell_{i}(D)\).
  Let \(u\) be the color of the region above \(k\).
  By using \cref{ex:longitude} and a similar computation for negative crossings, it is not hard to see that the contributions of the over-arcs to \eqref{eq:longitude-closed-rule} cancel out and the contributions of the under-arcs telecsope, so that
  \[
    \ell_{i}(D) =
    \frac{
      v_{k} u
      }{
      v_{k} \rho(y_{k}) u
    }
    =
    \widetilde{\ell}_{i}.
    \qedhere
  \]
\end{proof}

\section{Invariants from models of \texorpdfstring{\(\tcat\)}{T}}
\label{sec:invariants from models}
We will define \(\qfunc{}\) by using a version of the Reshetikhin-Turaev construction \cite{Reshetikhin1990} due to \citeauthor{Kashaev2005} \cite{Kashaev2005}.
We have a topological category \(\tcat\) of tangle diagrams colored with geometric data.
If we can specify how to interpret elementary tangles (which generate \(\tcat\)) as morphisms of an algebraic category \(\mathcal{C}\) and these morphisms are invariant under Reidemeister moves then we get functorial tangle invariants.
Following \textcite{Blanchet2018}%
\note{
  More precisely, our models correspond to what they call a \defemph{regular representation} of the underlying biquandle, while our pre-models correspond to their \defemph{Yang-Baxter models}.
  We have changed terminology to preserve the term ``representation'' for \(\slg\)-representations.
}
we assume \(\mathcal{C}\) is a pivotal category and call the required data a \defemph{model}
of \(\tcat\) in \(\mathcal{C}\).
Roughly speaking of a model is a solution of the parametrized Yang-Baxter equation (i.e.\ colored braid relation) valued in \(\mathcal{C}\) with parameters the colorings of \(\tcat\).

In this section we define \(\tcat\) and models of it in a pivotal category \(\mathcal{C}\) and prove that these give tangle invariants.
When the quantum dimensions of the relevant objects of the target category \(\mathcal{C}\) vanish (as they do for our main example \(\qfunc{}\)) these tangle invariants will be \(0\) on any link, so we explain how to use a standard method to obtain nontrivial link invariants.
We conclude the section by giving some simple examples of models valued in \(\vect\) closely related to quandle cocycles.
More interesting examples are the classical Chern-Simons invariant \(\csfunc{}\) defined in \cref{sec:The Chern-Simons invariant of a tangle} and its quantization \(\qfunc{}\) defined in \cref{sec:construction of invariant}.

\subsection{The tangle category}
\label{sec:tcat}
Usually in quantum topology we view (oriented, framed) tangles as a monoidal category:
\begin{description}
  \item[objects] points
  \item[morphisms] tangles
  \item[monoidal product] disjoint union of points (and tangles)
  \item[composition] tangle composition (union along boundary points)
\end{description}
This extends to tangles colored by a quandle \(\matset\) in the obvious way: now the points and tangle (diagrams) are colored by elements of \(\matset\).

To keep track of the octahedral shape parameters and define our invariants we need to also introduce region labels, the shadow colors \(\shadset\).
Because disjoint union of tangles needs to preserve these the monoidal product is no longer globally defined: there are domains and codomains, just as for tangle composition.
To handle these we consider a \(2\)-category of tangles obtained a sort of delooping (\cref{def:delooping}) fibered over a set \(\shadset\) with an action of \(\matset\).
Now the objects are region labels, the \(1\)-morphisms are colored points transforming between the region labels, and the \(2\)-morphisms are tangles.
Because we need to restrict to admissible colorings we work with a category of tangle diagrams, not tangles (\cref{rem:morphisms are diagrams}).

\begin{definition}
  \label{def:tcat}
  \(\tcat\) is the strict \(2\)-category with
  \begin{description}
    \item[objects] elements of \(\shadset = \mathbb{C}^{2} \setminus {0}\) (shadow colors)
    \item[\(1\)-morphisms] lists of oriented points with log-decorated matrices and a choice of shadow coloring.
      If \(g = (g, [v])\) is a decorated matrix and  \(\mu\) is a log-meridian, then \((+, g, \mu, u )\) is a morphism \(u \to gu\) and similarly \((-, g, \mu, u)\) is a morphism \(u \to g^{-1} u\).
      These generate \(1\)-morphisms under composition.
    \item[\(1\)-morphism composition] disjoint union of points (and tangles) along compatible region colors.
    Two \(1\)-morphisms \(S : u \to u', S' : u' \to u''\) (lists of colored, oriented points) are composable when the bottom region color of \(S\) agrees with the top color of \(S'\), and then the composition \(S \du S'\) is the disjoint union of points.
    \item[\(2\)-morphisms] oriented, framed, admissibly shadow-colored, log-decorated tangle diagrams up to planar isotopy.%
      \note{
        Recall that a \defemph{planar isotopy} is an isotopy of the underlying graph of the tangle diagram.
      }
      We call these \defemph{colored diagrams} for short.
    \item[\(2\)-morphism composition] union along boundaries with matching colors and orientations.
      Two \(2\)-morphisms \(D : S \to S', D': S' \to S''\) are composable when the colors and orientations of their boundary points (i.e.\@ the codomain and domain \(1\)-morphisms) \(S'\) match.
      The underlying diagram of the composition \(D D'\) is the usual composition of tangles, and the compatibility condition means \(D D'\) inherits its coloring from \(D\) and \(D'\).
  \end{description}
  Our convention is to draw the composition of \(1\)-morphisms \emph{vertically} and of \(2\)-morphisms \emph{horizontally}.
  We compose morphisms left-to-right.
\end{definition}

\begin{example}
  \label{ex:2-morphism braiding}
  Abbreviate
  \begin{align*}
    g_{1} &= (g_{1}, [v_{1}], \mu_{1})
    \\
    g_{2} &= (g_{1}, [v_{1}], \mu_{2})
    \\
    g_{2} \qn g_{1} &= (g_{1}^{-1} g_{2} g_{1}, [v_{2} g_{1}], \mu_{2})
  \end{align*}
  Choosing an admissible \(u \in \shadset\), \((+, g_{1})\) is a \(1\)-morphism \(u \to g_{1} u\) and \((+, g_{2})\) is a \(1\)-morphism \( g_{1} u \to g_{2} g_{1} u\).
  Their composition is well-defined, and is a \(1\)-morphism \(u \to g_{2} g_{1} u\).
  The crossing diagram
  \[
    \begin{tikzpicture}[line width=1, baseline=30, scale=1.5] 
      \draw[->] (0,0) node[left] {\(g_{2}\)} \br (1.5,1) ;
      \draw (1.5,1) node[right] { \(  g_{2} \qn g_{1} \) };
      \draw[white, line width=10] (0,1) node[left] {} \br (1.5,0) node[right] {};
      \draw[->] (0,1) node[left] {\(g_{1}\)} \br (1.5,0) ;
      \draw (1.5,0) node[right] {\(g_{1}\)};
      \draw (0, 0.5) node { \(g_{1} u\) };
      \draw (2, 0.5) node { \(g_{1}^{-1} g_{2} g_{1} u\) };
      \draw (0.75, 0) node { \(g_{2} g_{1} u\) };
      \draw (0.75, 1) node { \(u\) };
    \end{tikzpicture}
  \]
  is a \(2\)-morphism of \(\tcat\) between the \(1\)-morphisms
  \(
    g_{1} g_{2} : u \to g_{2} g_{1} u
  \)
  and
  \(
    (g_{2} \qn g_{1}) g_{1} : u \to g_{2} g_{1} u
  \), leaving the orientations implicit.
\end{example}

\begin{example}
  \label{ex:2-morphism cap}
  Abbreviate \(g = (g, [v], \mu)\) as before.
  Picking a shadow color \(u\), we can view \((-,g)\) as a morphism \(u \to g^{-1} u\) and \((+,g)\) as a morphism \(g^{-1} u \to u\).
  The cap diagram
  \[
    \begin{tikzpicture}[line width=1, baseline=30, scale=1.5] 
      \draw[->, looseness = 2] (0,0) node[left] {\(g\)} to[out=00,in=00]  (0,1) node[left] {\(g\)} ;
      \draw (0,1.25) node {\(u\)};
      \draw (0,.5) node[left] {\(g^{-1}u\)};
    \end{tikzpicture}
  \]
  is a \(2\)-morphism of \(\tcat\) from the composition \((-, g)(+, g) : u \to u \) to the identity  \(1\)-morphism \(u \to u\).
\end{example}

\begin{remark}
  \label{rem:morphisms are diagrams}
  The \(2\)-morphisms of \(\tcat\) are tangle \emph{diagrams}, not tangles: we do not consider two diagrams related by a Reidemeister move to be the same \(2\)-morphism of \(\tcat\).
  This is because we want to restrict to admissibly colored diagrams and this condition is not necessarily preserved under Reidemeister moves.
\end{remark}

\subsection{Pivotal categories}%
\label{sec:Pivotal categories}

We interpret the cups and caps of \(\tcat\) as evaluation and coevaluation morphisms of our category \(\mathcal{C}\), which in turn are specified by a pivotal structure.%
\note{
  Usually for quantum invariants we would specify a ribbon structure.
  For technical reasons we instead work in a pivotal category and then show that the braiding is compatible with the pivotal structure, as in \cite[Section 6.5]{Blanchet2018}.
}
This is a standard definition in quantum topology: see \cite{EGNO, GPMBook} for details.

Let \(\mathcal{C}\) be a (strict) monoidal category.
We will always assume that \(\mathcal{C}\) is linear over \(\mathbb{C}\) with \(\End_{\mathcal{C}}(\tsunit) = \mathbb{C}\) and \(\tsunit = \mathbb{C}\).
Suppose that \(\mathcal{C}\) has duals, so that each object \(V\) of \(\mathcal{C}\) has a dual object \(V^{*}\) and there are maps
\begin{gather*}
  \evup{V} \colon  V^{*} \otimes V \to \mathbb{C}
  \\
  \coevup{V} \colon  \mathbb{C} \to V \otimes V^{*} 
\end{gather*}
Our names for these match our (slightly nonstandard) graphical conventions:
\begin{align*}
  &\coevup{V} 
  =
  \begin{tikzpicture}[line width=1, baseline=10, scale=1] 
    \draw[->, looseness = 1.5] (0,0) node[right] {} to [out =180, in=180] (0,1) node[right] {\(V\)};
  \end{tikzpicture}
  &
  &\evup{V} 
  =
  \begin{tikzpicture}[line width=1, baseline=10, scale=1] 
    \draw[->, looseness = 1.5] (0,0) node[left] {} to [out =0, in=0] (0,1) node[left] {\(V\)};
  \end{tikzpicture}
\end{align*}
We want to interpret downward-oriented caps and cups consistently.
To do this we need an additional structure:

\begin{definition}
  \label{def:pivotal structure}
  A \defemph{pivotal structure} on \(\mathcal{C}\) is a natural family of isomorphisms
  \(
    \pi_{V} \colon V \to V^{**}
  \)
  compatible with the monoidal structure.
\end{definition}

We can now define \(\evdown{V} : V \otimes V^{*} \to \tsunit\) by 
\[
  \evdown{V} \defeq  \evup{V^{*}} (\pi_{V} \otimes \id_{V^{*}})
\]
and similarly
\[
  \coevdown{V} \defeq (\id_{V^{*}} \otimes \pi_{V}^{-1}) \coevup{V^{*}} 
\]
Together the maps are
\begin{equation}
  \label{eq:cup and cap conventions}
  \begin{aligned}
    &\coevup{V} 
    =
    \begin{tikzpicture}[line width=1, baseline=10, scale=1] 
      \draw[->, looseness = 1.5] (0,0) node[right] {} to [out =180, in=180] (0,1) node[right] {\(V\)};
    \end{tikzpicture}
    &&&
    &\evup{V} 
    =
    \begin{tikzpicture}[line width=1, baseline=10, scale=1] 
      \draw[->, looseness = 1.5] (0,0) node[left] {} to [out =0, in=0] (0,1) node[left] {\(V\)};
    \end{tikzpicture}
    \\
    &\coevdown{V} 
    =
    \begin{tikzpicture}[line width=1, baseline=10, scale=1] 
      \draw[<-, looseness = 1.5] (0,0) node[right] {\(V\)} to [out =180, in=180] (0,1) node[right] {};
    \end{tikzpicture}
    &&&
    &\evdown{V} 
    =
    \begin{tikzpicture}[line width=1, baseline=10, scale=1] 
      \draw[<-, looseness = 1.5] (0,0) node[left] {\(V\)} to [out =0, in=0] (0,1) node[left] {};
    \end{tikzpicture}
  \end{aligned}
\end{equation}

\begin{example}
  \label{ex:pivotal structure on vect}
  Let \(\vect\) be the category of finite-dimensional vector spaces over \(\mathbb{C}\).
  The dual object of \(V\) is the usual \(V^{*} \defeq \operatorname{Hom}_{\mathbb{C}}(V, \mathbb{C})\), and the maps are
  \begin{gather*}
    \evup{V} \colon f \otimes v \mapsto f(v)
    \\
    \coevup{V} \colon 1 \mapsto \sum_{i} v_{i} \otimes v^{i}
  \end{gather*}
  where \(i\mapsto v_{i}\) is a basis of \(V\) and \(i \mapsto v^{i}\) is the dual basis.
  The standard identification
  \[
    \pi_{V}(x) = (f \mapsto f(v))
  \]
  is a pivotal structure on \(\vect\), so the downward maps are
  \begin{gather*}
    \evdown{V} \colon v \otimes f \mapsto f(v)
    \\
    \coevup{V} \colon 1 \mapsto \sum_{i} v^{i} \otimes v_{i}
  \end{gather*}
\end{example}

For quantum invariants we want to work instead in the category \(\modcat{H}\) of finite-dimensional modules over a Hopf algebra \(H\) over \(\mathbb{C}\).
If \(V\) is a \(H\)-module the dual space is \(V^{*} = \operatorname{Hom}_{\mathbb{C}}(V, \mathbb{C})\) with \(H\)-action
\[
  h \cdot f \colon v \mapsto f(S(h) \cdot v)
\]
where the \defemph{antipode} \(S\) is part of the data of \(H\).
In general
\[
  v \mapsto (f \mapsto f(v))
\]
will \emph{not} be a map of \(H\)-modules \(V \to V^{**}\) because we may not have \(S^{2} = \id\).%
\note{
  Such Hopf algebras are called \defemph{involutive}.
  \(\qgrplong\) is not involutive.
}

\begin{definition}
  An invertible element \(\varpi \in H\) is called a \defemph{pivot} if
  \[
    \varpi h \varpi^{-1} = S^{2}(h)
  \]
  for all \(h \in H\).
  For \(V\) an \(H\)-module we define
  \begin{equation}
    \label{eq:pivot double dual}
    \pi_{V}
    \colon
    V \to V^{**}
    ,
    v \mapsto ( f \mapsto f(\varpi v) )
  \end{equation}
\end{definition}

It is an instructive exercise to work out that \(\pi_{V}\) is a map of \(H\)-modules, and more generally to verify the example below:

\begin{example}
  \label{ex:pivot gives pivotal}
  If \(\varpi\) is a pivot for \(H\), then the maps \eqref{eq:pivot double dual} define a pivotal structure on \(\modcat{H}\).
  The induced evaluation and coevaluation maps are
  \begin{align*}
    &\evup{V} \colon f \otimes v \mapsto f(v)
    &
    &\evdown{V} \colon v \otimes f \mapsto f(\varpi v)
    \\
    &\coevup{V} \colon 1 \mapsto \sum_{i} v_{i} \otimes v^{i}
    &
    &\coevdown{V} \colon 1 \mapsto \sum_{i} v^{i} \otimes \varpi^{-1} v_{i}
  \end{align*}
\end{example}

\subsection{Invariants from models}

\begin{definition}
  A \defemph{pre-model} \(\mathcal{F}\) of \(\tcat\) valued in a pivotal category \(\mathcal{C}\) is the minimal data needed to interpret a colored tangle diagram \(D\) of \(\tcat\) as a morphism of \(\mathcal{C}\).
  It consists of a family
  \(
    \mathcal{F}(g, u)
  \)
  of objects \(\mathcal{C}\) indexed by a log-decorated matrix \(g\) and a shadow color \(u\) and a family of invertible morphisms between them:
  \begin{align*}
    \mathcal{F}\leftfun(
    \begin{tikzpicture}[line width=1, scale=0.75, xscale = 1.5, baseline={(current bounding box.center)}] 
      \coordinate (1) at (0,1);
      \coordinate (2) at (0,0);
      \coordinate (3) at (1,0);
      \coordinate (4) at (1,1);
      \tikzBraidingEast{+}{(1)}{(2)}{(3)}{(4)}
      \node[left] at (1) {\(g_{1}\)};
      \node[left] at (2) {\(g_{2}\)};
      \node[right] at (3) {\(g_{1}\)};
      \node[right] at (4) {\(g_{4}\)};
      \node[above] at (0.5, 1) {\(u\)};
    \end{tikzpicture}
    \rightfun)
    \colon
    \mathcal{F}(g_{1}, u) 
    \otimes
    \mathcal{F}(g_{2}, g_{1} u) 
    \to
    \mathcal{F}(g_{4}, u) 
    \otimes
    \mathcal{F}(g_{1}, g_{4} u) 
  \end{align*}
\end{definition}

\begin{figure}
  \begin{equation*}
    \begin{tikzpicture}[line width = 1, scale = 1, yscale = 1.5, baseline={(current bounding box.center)}]
      \coordinate (1) at (1,1);
      \coordinate (2) at (0,1);
      \coordinate (3) at (0,0);
      \coordinate (4) at (1,0);
      \draw[->] (2) to [out = 270, in = 90] (4);
      \draw[white, line width=10] (1) \br (3);
      \draw[->] (1)  to [out = 270, in = 90] (3);
      \node[above right] at (1) {\(1\)};
      \node[above left] at (2) {\(2\)};
      \node[below left] at (3) {\(3\)};
      \node[below right] at (4) {\(4\)};
    \end{tikzpicture}
    =
    \begin{tikzpicture}[line width=1, scale=1, xscale = 1.5, baseline={(current bounding box.center)}] 
      \coordinate (1) at (0,1);
      \coordinate (2) at (0,0);
      \coordinate (3) at (1,0);
      \coordinate (4) at (1,1);
      \tikzBraidingEast{+}{(1)}{(2)}{(3)}{(4)}
      \coordinate (1a) at (0,2);
      \coordinate (1b) at (1,2);
      \draw[->, looseness = 1.5] (1a) to [out =180, in=180] (1);
      \draw[->] (1b) to (1a);
      \coordinate (3a) at (1,-1);
      \coordinate (3b) at (0,-1);
      \draw[->, looseness = 1.5] (3) to [out =0, in=0] (3a);
      \draw[->] (3a) to (3b);
      \node[right] at (1b) {\(1\)};
      \node[left] at (3b) {\(3\)};
      \node[left] at (2) {\(2\)};
      \node[right] at (4) {\(4\)};
    \end{tikzpicture}
  \end{equation*}
  \caption{
    Braidings that are not oriented left-to-right are determined by rotation (via planar isotopy) and the pivotal structure.
  }
  \label{fig:braiding rotated example}
\end{figure}

A pre-model defines a morphism \(\mathcal{F}(D, \mu)\) of \(\mathcal{C}\) for any colored diagram \(D\) in the standard way:
\begin{enumerate}
  \item 
    Put the diagram in Morse position (which can be done without applying any Reidemeister moves, i.e.\ by planar isotopy).
  \item
    Positively-oriented points and segments are assigned the objects \(\mathcal{F}(g, u)\).
  \item 
    Using the pivotal structure we assign negatively-oriented points and segments the appropriate dual modules and assign caps and cups are the evaluation and coevaluation morphisms.
  \item 
    Left-to-right braidings are assigned the braiding maps and their inverses, while the other braiding maps are determined by the pivotal structure as in \cref{fig:braiding rotated example}.
\end{enumerate}
This assignment is compatible with composition, or more formally is functorial and preserves tensor products.
Because \(\tcat\) is a \(2\)-category we need to state this slightly differently:

\begin{definition}
  \label{def:delooping}
  The \defemph{delooping} \(\delooping \mathcal{C}\) of a monoidal category \(\mathcal{C}\) is a \(2\)-category with
  \begin{description}
    \item[objects] a unique object \(\bullet\)
    \item[\(1\)-morphisms] the objects \(V\) of \(\mathcal{C}\) viewed as maps \(\bullet \to \bullet\)
    \item[\(1\)-morphism composition] of \(V, W : \bullet \to \bullet\) is their tensor product \(V \otimes W\).
    \item[\(2\)-morphisms] the morphisms \(f : V \to W\) of \(\mathcal{C}\) (linear maps intertwining the \(H\)-action)
    \item[\(2\)-morphism composition] composition \(\circ\) of functions
  \end{description}
  The relationships between \(\otimes\) and \(\circ\) are a special case of the distributive laws for a \(2\)-category.
\end{definition}

\begin{lemma}
  Any pre-model \(\mathcal{F}\) with values in \(\mathbb{C}\) defines a \(2\)-functor \(\mathcal{F} : \tcat \to \mathsf{B}\mathcal{C}\).
\end{lemma}

\begin{definition}
  \label{def:regular object}
  An object \(V\) of a monoidal category \(\mathcal{C}\) is \defemph{regular} \cite[Section 4.1]{Blanchet2018} if for all objects \(X, Y\) of \(\mathcal{C}\) the maps
  \begin{gather}
    \hom(X,Y) \to \hom(V \otimes X , V \otimes Y), f \mapsto \id_{V} \otimes f
    \\
    \hom(X,Y) \to \hom(X \otimes V , Y \otimes V), f \mapsto f \otimes \id_{V}
  \end{gather}
  are injective.
  In other words \(V\) is regular if tensoring with it is injective on morphisms.
\end{definition}

\begin{definition}
  We say a pre-model \(\mathcal{F}\) of \(\tcat\) valued in \(\mathcal{C}\) is a \defemph{model} if
  \begin{enumerate}
    \item the values of \(\mathcal{F}\) on colored points are regular objects of \(\mathcal{C}\), and
    \item the morphism assigned by \(\mathcal{F}\) to any of the Reidemeister move diagrams of \cref{fig:reid moves that vanish} is the identity map.
  \end{enumerate}
\end{definition}

\begin{corollary}
  \label{thm:models give invariants}
  For any model \(\mathcal{F}\) of \(\tcat\) and colored isotopic diagrams \(D\) and \(D'\)
  \[
    \mathcal{F}(D) = \mathcal{F}(D').
  \]
  Furthermore if \(D\) is a link diagram then \(\mathcal{F}(D)\) is gauge invariant.
\end{corollary}

\begin{proof}
  By \cref{thm:diagrams linked by stabilization and R} \(D\) and \(D'\) are linked by a stabilization, a sequence of admissible Reidemeister moves, and a (de)stabilization.
  We have
  \[
    \mathcal{F}(I \du D)
    =
    \mathcal{F}(I \du D')
  \]
  because \(\mathcal{F}\) is invariant under admissible Reidemeister moves, and then regularity implies
  \[
    \mathcal{F}(D)
    =
    \mathcal{F}(D')
    .
  \]
  The second claim follows from a simpler version of the argument in \cref{thm:link invariants from model} below.
\end{proof}

Given two tensor categories \(\mathcal{C}_{1}, \mathcal{C}_{2}\) over \(\mathbb{C}\) one can take their (Deligne) \defemph{tensor product} \(\mathcal{C}_{1} \boxtimes \mathcal{C}_{2}\) \cite[Section 1.11]{EGNO}.
We need two cases:
\begin{itemize}
  \item If each \(\mathcal{C}_{i} = \modcat{H_{i}}\) is the category of representations of a Hopf algebra \(H_{i}\) then  \(\mathcal{C}_{1} \boxtimes \mathcal{C}_{2} = \modcat{H_{1} \otimes_{\mathbb{C}} H_{2}}\).
  \item In particular for any \(\mathcal{C}\) we have \(\mathcal{C} \boxtimes \vect = \mathcal{C}\).
\end{itemize}
A similar operation applies to models, and this tensor product is useful when studying the relationship to the torsion in \cref{sec:torsion}.

\begin{definition}
  \label{def:product of models}
  If \(\mathcal{F}_{1}\) and \(\mathcal{F}_{2}\) are models valued in categories \(\mathcal{C}_{1}\), \(\mathcal{C}_{2}\) their \defemph{tensor product} \(\mathcal{F}_{1} \boxtimes \mathcal{F}_{2}\) is the model valued in \(\mathcal{C}_{1} \boxtimes \mathcal{C}_{2}\) whose value  on points is
  \[
    (\mathcal{F}_{1} \boxtimes \mathcal{F}_{2})(g, u)
    \defeq
    \mathcal{F}_{1}(g, u)
    \boxtimes
    \mathcal{F}_{2}(g, u)
  \]
  and similarly for the braidings.
  The trivial model \(\textsf{1}\) sending every point to \(\mathbb{C}\) and every crossing to the identity map is the unit for \(\boxtimes\), and a model \(\mathcal{F}\) is \defemph{invertible} if there is a \(\mathcal{G}\) with \(\mathcal{F} \boxtimes \mathcal{G} \) equivalent to \(\textsf{1}\).
\end{definition}

Both \(\csfunc{}\) and the cocycle invariants  \(\mathcal{F}_{\phi}\) of \cref{sec:Shadow cocycles} below are invertible models.
In general \(\qfunc{}\) is not invertible, but we explain in \cref{sec:recovering the chern-simons invariant} that \(\qfunc[1]{}\) and \(\csfunc{}\) are inverse models.

\subsection{Link invariants}%
\label{sec:Link invariants}

For some models like the Chern-Simons invariant we can simply evaluate on a link diagram to get link invariants.
However, the value of \(\qfunc{}\) on any link is \(0\) because of vanishing quantum dimensions.
(See \cref{sec:Properties of qfun} for a concrete explanation.)
To extract nontrivial invariants we use a standard trick.

\begin{definition}
  Let \(L\) be a link with a distinguished component \(K\).
  A \defemph{cut presentation} of \((L,K)\) is a tangle \(T\) with one incoming and outgoing component so that the closure of \(T\) is \(L\) and the open component corresponds to \(K\).
\end{definition}

\begin{definition}[\protect{\cite[Section 4.1]{Blanchet2018}}]
  \label{def:absolutely simple object}
  An object \(V\) of a pivotal category \(\mathcal{C}\) defined over \(\mathbb{C}\) is \defemph{absolutely simple} if the map \(\mathbb{C} \to \End_{\mathcal{C}}(V), z \mapsto z\id_{V} \) is bijective.
  In particular, a simple object is absolutely simple by Schur's Lemma.
  For an endomorphism \(f\) of an absolutely simple object \(V\) we write \(\sclbrak{f}\) for the scalar with
  \[
    f = \sclbrak{f} \id_{V}.
  \]
  A pre-model \(\mathcal{F}\) of \(\tcat\) valued in \(\mathcal{C}\) is \defemph{simple} if the values of \(\mathcal{F}\) on log-meridians of points are absolutely simple objects.
\end{definition}

Let \((L, K)\) be a link with a distinguished component, \(\rho\) a decorated representation of \(\extr{L}\), and \(\mu \) a log-decoration of \(\rho\).
Let \(\rho'\) be any representation of \(\extr{L}\) gauge-equivalent to \(\rho\).
Choose a cut presentation \(T\) of \((L, K)\) and an admissibly colored diagram \((D, \rho', u, \mu)\) of \(T\).
(It may be that \(\rho\) is not admissible for any diagram \(D\).)
If \(\mathcal{F}\) is simple, it makes sense to write
\begin{equation}
  \operatorname{F}(L, K ; \rho, \mu)
  \defeq
  \sclbrak{ \mathcal{F}(D, \rho', u, \mu) }
\end{equation} 

\begin{theorem}
  \label{thm:link invariants from model}
  Let \(\mathcal{F}\) be a simple model of \(\tcat\) in \(\mathcal{C}\).
  Then the scalar \(\operatorname{F}(L, K ; \rho, \mu)\) defined above is independent of the choice of:
  \begin{thmenum}
    \item
    \label{thm:link invariants from model:tangle}
      cut presentation \(T\) and diagram \(D\) of \(T\), and
    \item
    \label{thm:link invariants from model:gauge}
    representative \(\rho'\) of the gauge class of \(\rho\) and shadow coloring \(u\).
      \qedhere
  \end{thmenum}
\end{theorem}

\begin{proof}
  \ref{thm:link invariants from model:tangle}
  Let \((D, \rho, u, \mu)\) and  \((D', \rho, u, \mu)\) be colored diagrams of any two (potentially different) cut presentations of \((L, K; \rho, u)\).
  Here saying they have the same shadow coloring means they have the same shadow color of their topmost regions.
  By \cref{thm:tangle folk theorem} below \(D\) and \(D'\) are isotopic as tangles, so by \cref{thm:models give invariants} \(\mathcal{F}(D, \rho, u, \mu) = \mathcal{F}(D', \rho, u, \mu)\).

  \ref{thm:link invariants from model:gauge}
  Suppose \(D = (D, \rho, u, \mu)\) is a diagram of a cut presentation of \((L, K; \rho, \mu)\) and \(D' = (D', \rho', u', \mu)\) is a gauge transformation of \(D\).
  By definition \(
    \operatorname{F}(L, K; \rho, \mu)
    =
    \sclbrak{\mathcal{F}(D)}
  \)
  and we need to show that
  \[
    \sclbrak{\mathcal{F}(D)}
    =
    \sclbrak{\mathcal{F}(D')}.
  \]
  Assume \(D'\) is a type (A) gauge transformation of \(D\) as in \cref{fig:gauge-transformation-A}; an identical argument works for type (B) transformations.
  For some colors \(\chi, \psi\) we have
  \begin{align*}
    \sclbrak{\mathcal{F}(D)}
    \id_{\mathcal{F}(\psi)}
    \otimes
    \id_{\mathcal{F}(\chi)}
    =
    \mathcal{F}\leftfun(
    \begin{tikzpicture}[line width=1, scale=1, baseline = (D)] 
      \node[draw,name = D, minimum height = .5cm, minimum width = .75cm] at (1,0) {\(D\)};
      \draw (0,0) to (D.west);
      \draw[->] (D.east) to (2,0) node[right] {\(\chi\)};
      \draw[->] (0,0.5) to (2, 0.5);
      \node[right] at (2, 0.5) {\(\psi\)};
    \end{tikzpicture}
    \rightfun)
  \end{align*}
  Isotopy invariance gives
  \begin{align*}
    \mathcal{F}\leftfun(
    \begin{tikzpicture}[line width=1, scale=1, baseline = (D)]
      \begin{scope}[baseline = (D)]
        \node[draw,name = D, minimum height = .5cm, minimum width = .75cm] at (1,0) {\(D\)};
        \draw (0,0) to (D.west);
        \draw[->] (D.east) to (2,0) node[right] {\(\chi\)} ;
        \draw[->] (0,0.5) to (2, 0.5) ;
        \node[right] at (2, 0.5) {\(\psi\)};
      \end{scope}
    \end{tikzpicture}
    \rightfun)
    &=
    \mathcal{F}\leftfun(
    \begin{tikzpicture}[line width=1, scale=1, baseline = (D)]
      \node[draw,name = D, minimum height = .5cm, minimum width = .75cm] at (1,0) {\(D'\)};
      \draw (0,0) to (D.west);
      \draw[->] (D.east) to (2,0) node[right] {\(\chi\)};
      \draw[white, line width = 6] (0,0.5) .. controls (0.5, 0.5) and (0.5, -1) .. (1, -1) .. controls (1.5, -1) and (1.5, 0.5) .. (2,0.5);
      \draw[->] (0,0.5) .. controls (0.5, 0.5) and (0.5, -1) .. (1, -1) .. controls (1.5, -1) and (1.5, 0.5) .. (2,0.5);
      \node[right] at (2, 0.5) {\(\psi\)};
    \end{tikzpicture}
    \rightfun)
    \\
    &=
    \sclbrak{\mathcal{F}(D')}
    \mathcal{F}\leftfun(
    \begin{tikzpicture}[line width=1, scale=1, baseline = (D)]
      \draw[->] (0,0) to (2,0) node[right] {\(\chi\)};
      \draw[white, line width = 6] (0,0.5) .. controls (0.5, 0.5) and (0.5, -1) .. (1, -1) .. controls (1.5, -1) and (1.5, 0.5) .. (2,0.5);
      \draw[->] (0,0.5) .. controls (0.5, 0.5) and (0.5, -1) .. (1, -1) .. controls (1.5, -1) and (1.5, 0.5) .. (2,0.5);
      \node[right] at (2, 0.5) {\(\psi\)};
    \end{tikzpicture}
    \rightfun)
    \\
    &=
    \sclbrak{\mathcal{F}(D')}
    \mathcal{F}\leftfun(
    \begin{tikzpicture}[line width=1, scale=1, baseline = (D)]
      \node[name = D] (1,0) {};
      \draw[->] (0,0.5) to (2,0.5) node[right] {\(\psi\)};
      \draw[->] (0,0) to (2,0) node[right] {\(\chi\)};
    \end{tikzpicture}
    \rightfun)
    \\
    &=
    \sclbrak{\mathcal{F}(D')}
    \id_{\mathcal{F}(\psi)}
    \otimes
    \id_{\mathcal{F}(\chi)}.
  \end{align*}
  Because \(\mathcal{F}(\psi)\) and \(\mathcal{F}(\chi)\) are regular objects this implies
  \[
    \sclbrak{\mathcal{F}(D)}
    =
    \sclbrak{\mathcal{F}(D')}
    .
    \qedhere
  \]
\end{proof}

\begin{lemma}[Folk theorem]
  \label{thm:tangle folk theorem}
  Any two cut presentations of \((L,K)\) are isotopic as tangles.
\end{lemma}
I learned the idea behind this proof from \citeauthor{TJFMathOverflow} \cite{TJFMathOverflow}, which he learned from John H.\ Conway.
\begin{proof}
  Consider \(T\) as being inside a \(3\)-ball \(B\) inside \(S^3\).
  After closing up \(T\) outside \(B\) we obtain a link isotopic to \(L\).
  Invert \(S^3\) across the boundary of \(B\).
  We now have a link \(\tilde L\) with a bead (the image of the ball under the inversion) on a distinguished component \(\tilde K\), and inside the bead is a single untangled strand.
  Applying the same construction to \(T'\) gives a similar link \((\tilde L', \tilde K')\).
  Because \(T\) and \(T'\) have isotopic closures \((\tilde L, \tilde K)\) and \((\tilde L', \tilde K')\) must be isotopic as well; here we are using the fact that bead lies on the same component of each.
  Furthermore we can choose this isotopy to avoid the beads, which means that after inverting again we have the required isotopy between \(T\) and \(T'\).
\end{proof}

In general the invariant \(\operatorname{F}(L, K; \rho, \mu)\) will depend on the choice of component \(K\), as for the Hopf links in \cref{sec:Example: Hopf links}.
In many circumstances this can be eliminated by using \defemph{modified dimensions} \cite{Geer2009}.
These are a family of scalars \(\operatorname{d}(\mathcal{F}(g, u, \mu))\) parametrized by the objects assigned to points.
Set
\[
  \widetilde{\operatorname{F}}(L; \rho, \mu)
  \defeq
  \operatorname{d}(\mathcal{F}(g, u, \mu))
  \sclbrak{ \mathcal{F}(D) }
\]
where \(D\) is a diagram of some cut presentation of \((L, K)\) and  \((g, u, \mu)\) is the color of the boundary segments of \(D\).
The modified dimensions can be chosen so that this quantity is independent of the component \(K\).
Usually they are determined solely by the gauge class of \((g, u, \mu)\) or simply by \(\mu\), and gauge-independence of \(\operatorname{d}\) is a necessary condition for \(\widetilde{\operatorname{F}}\) to be gauge-independent.

Modified dimensions for the category used to compute \(\qinv{}\) are known \cite{Geer2017}, and we can re-derive these results by using the Hopf link computations in \cref{sec:Example: Hopf links}.
However, they only make sense for certain values of \(\mu\).
To simplify the statement of \cref{thm:qfunc defined and properties} we do not use modified dimensions in the definition of \(\qinv{}\).
We discuss modified dimensions further in \cref{sec:Relation to the Kashaev--Akutsu-Deguchi-Ohtsuki invariant}.

\subsection{Shadow cocycles}%
\label{sec:Shadow cocycles}
We can construct simple examples of models from shadow quandle cocycles.%
\note{
  Quandle cohomology was first studied by \textcite{Carter2003}.
  Our cocycles are a variant of the shadow cocycles of \textcite{Inoue2014}.
}
Here it is convenient to take a slightly different labeling convention.
Let \(\mathsf{A}_{2} \subset \shadset \times \matset^{2}\) be the subset of colors \((u, x, y)\) for which the crossing diagram
\begin{center}
  \begin{tikzpicture}[line width=1, scale=1, xscale = 1.5] 
    \coordinate (1) at (0,1);
    \coordinate (2) at (0,0);
    \coordinate (3) at (1,0);
    \coordinate (4) at (1,1);
    \draw[->] (2) \br (4);
    \draw[white, line width=10] (1) \br (3);
    \draw[->] (1) \br (3);
    \node[left] at (1) {\(x\)};
    \node[right] at (4) {\(y\)};
    \node[above] at (0.5,1) {\(u\)};
  \end{tikzpicture}
\end{center}
is admissible.
Let \(\phi \colon \mathsf{A}_{2} \to \mathbb{C}^{\times}\) be any function.
We can define a pre-model \(\mathcal{F}_{\phi}\) valued in \(\vect\) by assigning every point \(\mathbb{C}\) or its dual.
Identifying \(\operatorname{Hom}_{\mathbb{C}}(\mathbb{C} \otimes \mathbb{C}, \mathbb{C} \otimes \mathbb{C})\) with \(\mathbb{C}\), we define
\begin{align*}
  \mathcal{F}_{\phi}
  \colon
  \begin{tikzpicture}[line width=1, scale=1, xscale = 1.5, baseline={(current bounding box.center)}] 
    \coordinate (1) at (0,1);
    \coordinate (2) at (0,0);
    \coordinate (3) at (1,0);
    \coordinate (4) at (1,1);
    \draw[->] (2) \br (4);
    \draw[white, line width=10] (1) \br (3);
    \draw[->] (1) \br (3);
    \node[left] at (1) {\(x\)};
    \node[right] at (4) {\(y\)};
    \node[above] at (0.5,1) {\(u\)};
  \end{tikzpicture}
  &\mapsto
  \phi(u, x, y)
  \\
  \mathcal{F}_{\phi}
  \colon
  \begin{tikzpicture}[line width=1, scale=1, xscale = 1.5, baseline={(current bounding box.center)}] 
    \coordinate (1) at (0,1);
    \coordinate (2) at (0,0);
    \coordinate (3) at (1,0);
    \coordinate (4) at (1,1);
    \draw[->] (1) \br (3);
    \draw[white, line width=10] (2) \br (4);
    \draw[->] (2) \br (4);
    \node[left] at (1) {\(x\)};
    \node[right] at (4) {\(y\)};
    \node[above] at (0.5,1) {\(u\)};
  \end{tikzpicture}
  &\mapsto
  \phi(u, y, x)^{-1}
\end{align*}
These values are chosen so that \(\mathcal{F}_{\phi}\) satisfies the \(\reid{2a}\) and \(\reid{2b}\) moves.
Because the pivotal structure on \(\vect\) is trivial \(\mathcal{F}_{\phi}\) will also satisfy the sideways moves \(\reid{2}[-+]\) and  \(\reid{2}[+-]\).
We say that \(\phi\) is a \defemph{generic \(2\)-cocycle} if it satisfies
\begin{gather}
  \label{eq:shadow cocycle condition}
  \phi(u, x, y \qn x)
  \phi((y \qn x) u , x, z \qn x )
  =
  \phi(xu, y, z \qn y)
  \phi( ((z \qn y) \qn x) u , x, y \qn x)
  \\
  \label{eq:shadow framing condition}
  \phi(u, x, x)
  =
  \phi(x^{-1}u, x, x)
\end{gather}

\begin{proposition}
  \label{thm:cocycles give models}
  \(\mathcal{F}_{\phi}\) is a model of \(\tcat\) if and only if \(\phi\) is a generic \(2\)-cocycle.
\end{proposition}

\begin{proof}
  The definition of \(\mathcal{F}_{\phi}\) on negative crossings ensures it takes the value \(1\) on the \(\reid{2}\) diagrams.

  It is convenient to split the \(\reid{3}\) diagram into two parts.
  One half is sent to
  \begin{align*}
    \mathcal{F}_{\phi}\colon
    &
    \begin{tikzpicture}[line width=1, scale=1, xscale = 1.5, baseline={(current bounding box.center)}, line cap = rect]
      \tikzBraidingEast{+}{(0, 2)}{(0, 1)}{(1, 1)}{(1, 2)}
      \tikzBraidingEast{+}{(1, 1)}{(1, 0)}{(2, 0)}{(2, 1)}
      \tikzBraidingEast{+}{(2, 2)}{(2, 1)}{(3, 1)}{(3, 2)}
      \draw[-] (1,2) to (2,2);
      \draw[-] (0,0) to (1,0);
      \draw[-] (2,0) to (3,0);
      \node[left] at (0,2) {\(x\)};
      \node[left] at (0,1) {\(y\)};
      \node[left] at (0,0) {\(z\)};
      \node[below] at (2,1) {\(z \qn x\)};
      \node[right] at (3,1) {\(y \qn x\)};
      \node[right] at (3,2) {\((z \qn y) \qn x\)};
      \node at (1.5, 2.5) {\(u\)};
      \node at (1.5, 1.5) {\((y \qn x)u\)};
    \end{tikzpicture}
    \\
    &\mapsto
    \phi(u, x, y \qn x)
    \phi((y \qn x) u , x, z \qn x )
    \phi( u , x, (z \qn y) \qn x ) 
  \end{align*}
  where we used the quandle relation \cref{eq:quandle distributivity} to simplify the coloring of the top-right segment.
  On the other hand,
  \begin{align*}
    \mathcal{F}_{\phi}\colon
    &
    \begin{tikzpicture}[line width=1, scale=1, xscale = 1.5, baseline={(current bounding box.center)}, line cap = rect]
      \tikzBraidingEast{+}{(0, 1)}{(0, 0)}{(1, 0)}{(1, 1)}
      \tikzBraidingEast{+}{(1, 2)}{(1, 1)}{(2, 1)}{(2, 2)}
      \tikzBraidingEast{+}{(2, 1)}{(2, 0)}{(3, 0)}{(3, 1)}
      \draw[-] (0,2) to (1,2);
      \draw[-] (2,2) to (3,2);
      \draw[-] (1,0) to (2,0);
      \node[left] at (0,2) {\(x\)};
      \node[left] at (0,1) {\(y\)};
      \node[left] at (0,0) {\(z\)};
      \node[below] at (1,1) {\(z \qn y\)};
      \node[right] at (3,1) {\(y \qn x\)};
      \node[right] at (3,2) {\((z \qn y) \qn x\)};
      \node at (1.5, 2.5) {\(u\)};
      \node at (0.5, 1.5) {\(xu\)};
      \node at (2.5, 1.5) {\(((z \qn y) \qn x)u\)};
    \end{tikzpicture}
    \\
    &\mapsto
    \phi(u, x, (z \qn y) \qn x ) 
    \phi(xu, y, z \qn y)
    \phi( ((z \qn y) \qn x) u , x, y \qn x)
  \end{align*}
  and canceling the common factor of \(\phi(u, x, (z \qn y) \qn x )\) gives \eqref{eq:shadow cocycle condition}.

  Similarly for \(\reid{1f}\) we have
  \begin{equation*}
    \mathcal{F}_{\phi}
    \colon
    \begin{tikzpicture}[line width = 1, scale = 1, xscale = 1.5, baseline={(current bounding box.center)}]
      \draw[-] (-0.5,0) to [out = 00, in = 60] (1,-1);
      \draw[-] (1,-1) to [out = -120, in = -60] (0,-1);
      \draw[white, line width=10] (0,-1) to [out = -240, in = 180] (1.5,0);
      \draw[->] (0,-1) to [out = -240, in = 180] (1.5,0);
      \node[left] at (-0.5,0) {\(x\)};
      \node[right] at (1.5,0) {\(x\)};
      \node[above] at (0.5, 0) {\(u\)};
    \end{tikzpicture}
    \mapsto
    \phi(u, x,x)
  \end{equation*}
  and
  \begin{equation*}
    \mathcal{F}_{\phi}
    \colon
    \begin{tikzpicture}[line width = 1, scale = 1, xscale = 1.5, baseline={(current bounding box.center)}]
      \draw[-] (-0.5,0) to [out = 00, in = -60] (1,1);
      \draw[-] (1,1) to [out = 120, in = 60] (0,1);
      \draw[white, line width=10] (0,1) to [out = 240, in = 180] (1.5,0);
      \draw[->] (0,1) to [out = 240, in = 180] (1.5,0);
      \node[left] at (-0.5,0) {\(x\)};
      \node[right] at (1.5,0) {\(x\)};
      \node[above] at (0.5, 1.5) {\(u\)};
      \node at (0.5, 0.75) {\(x^{-1}u\)};
    \end{tikzpicture}
    \mapsto
    \phi(x^{-1}u, x,x)
  \end{equation*}
  because the shadow color inside the loop is \(x^{-1}u\).
\end{proof}

\begin{example}
  It is easy to see that
  \(
    M(u, g_{1}, g_{4}) = m_{1} m_{2}
  \)
  is a generic \(2\)-cocycle.
\end{example}

The squared longitude eigenvalue is not a generic cocycle because it depends on a choice of component, but if we take the product over all components we get a cocycle:

\begin{example}
  \[
    L(u, g_{1}, g_{4})
    =
    \frac{
      a_{\lE}
      }{
      a_{\lW}
    }
    \frac{
      b_{2}  b_{3}
      }{
      b_{1} b_{4} 
    }
    =
    \frac{
      \det(g_{4} u, \evec{2})
      }{
      \det(g_{1} u, \evec{2})
    }
    \frac{
      (v_{1} u) (v_{2} u)
      }{
      (v_{2} g_{1} u) (v_{1} g_{4} u)
    }
  \]
  is a generic \(2\)-cocyle by \cref{thm:longitude eigenvalues make sense}.
\end{example}

\section{The Chern-Simons invariant \texorpdfstring{\(\csfunc{}\)}{Iψ} of a tangle}
\label{sec:The Chern-Simons invariant of a tangle}

\textcite{Inoue2014} showed how to use shadow colorings and quandles to compute the complex Chern-Simons invariant of a link directly from its diagram.
Their method only works for links with boundary-parabolic representations \(\rho\).
Here we explain how to extend their method to tangles with general \(\slg\) representation.
We use this result both to motivate our interpretation of \(\qfunc{}\) as a quantization of \(\csfunc{}\) and as part of the proof that it satisfies the \(\reid{3}\) move.

\subsection{Points}
\label{sec:The Chern-Simons invariant of a tangle:points}
For \(a, b \in \mathbb{C}^{\times}\) let \(\qpf[]{a,b}\) be the vector space of functions on the set of logarithms of \(b\)
\[
  f \colon \set{\beta \in \mathbb{C} \given e^{\tu \beta} = b} \to \mathbb{C} 
\]
that are quasi-periodic with period \(1\) and constant \(a\):
\[
  f(\beta + 1) = a f(\beta).
\]
This vector space is obviously \(1\)-dimensional, and the family of vector spaces \((a, b) \mapsto \qpf[]{a,b}\) is a line bundle over  \(\mathbb{C}^{\times}\times \mathbb{C}^{\times}\) in a natural way.
We write \(\ket{a, \beta}\) for the element of \(\qpf[]{a,b}\) sending \(\beta\) to \(1\).
Usually \(a\) will be clear and we simply write \(\ket{a, \beta} = \ket{\beta}\).
Any  \(\ket{\beta}\) is a basis of the complex line \(\qpf[]{a, b}\), and the basis vectors are related by
\[
  \ket{\beta + 1} = a^{-1} \ket{\beta}.
\]
We write \(\bra{\beta}\) for the dual basis of \(\qpf[]{a,b}^{*}\), which has
\[
  \bra{\beta + 1} = a \bra{\beta}
  \text{ and }
  \braket{\beta'}{\beta}
  =
  a^{\beta' - \beta}
\]

\begin{definition}
  \label{def:csfunc value point}
  For a positively oriented point we set
  \[
    \csfunc{
      \begin{tikzpicture}[baseline=(g.base)]
        \draw (0,0.5) node {\(u\)};
        \draw[->] (0,0) node(g)[left] {\((g, [v])\)} to (0.1,0);
        \draw (0,-0.5) node {\(gu\)};
      \end{tikzpicture}
    }
    =
    \qpf[]{
      \frac{
        \det(u, \evec{2})
        }{
        \det( gu, \evec{2})
      }
      ,
      -
      \frac{
        v \evec{2}
        }{
        v u
      }
    }
    =
    \qpf[]{
      \frac{
        a^{\up}
        }{
        a^{\dn}
      }
      ,
      b
    }
  \]
  independent of the log-meridian \(\mu\).
  Here the \(a^{\up}, a^{\dn}\) are the region parameters of the top and bottom regions and \(b\) is the segment parameter.
\end{definition}

In more abstract language we have defined a line bundle \(\csfunc{\point} \to \admvar{\point}\) over the space \(\admvar{\point}\) of admissible shadow colorings of a point.
Our tangle invariants will take values in spaces of linear maps between these points.

We use the standard pivotal structure on \(\vect\) of \cref{ex:pivotal structure on vect}.
Thus the value of \(\csfunc{}\) on a negatively oriented point is the dual space
\[
  \csfunc{
    \begin{tikzpicture}[baseline=(g.base)]
      \draw (0,0.5) node {\(u\)};
      \draw[<-] (0,0) node(g)[left] {\((g, [v])\)} to (0.1,0);
      \draw (0,-0.5) node {\(g^{-1}u\)};
    \end{tikzpicture}
  }
  =
  \qpf[]{
    \frac{
      \det(g^{-1}u, \evec{2})
    }{
      \det( u, \evec{2})
    }
    ,
    -
    \frac{
      v \evec{2}
    }{
      v g^{-1} u
    }
  }^{*}
  =
  \qpf[]{
    \frac{
      a'
    }{
      a
    }
    ,
    b
  }^{*}
\]
The region parameters are consistent with this as in \cref{ex:2-morphism cap}.

\subsection{Crossings}
The value of \(\csfunc{}\) on a crossing is a particular normalization of the Chern-Simons invariant of an ideal octahedron, which we divide it into four ideal tetrahedra.
The value on each can be computed using a special function called the dilogarithm.
We discuss our variant
\[
  \ldil{\zeta}
  =
  \frac{
    \operatorname{Li}_2(e^{\tu \zeta})
  }{\tu}
  +
  \zeta \log(1 - e^{\tu \zeta})
\]
of the dilogarithm further in \cref{sec:quantum dilogarithms}.
Here we observe that \(e^{\ldil{}}\) is a meromorphic function so we do not need to consider any further branch cuts.

Define functions of complex parameters \(\beta = (\beta_{1}, \beta_{2}, \beta_{3}, \beta_{4})\) and \(\mu = (\mu_{1}, \mu_{2})\) by
\begin{gather}
  \label{eq:octahedral function positive}
  \begin{aligned}
      \octfunc{+}{\mu}{\beta}
      &\defeq
      \ldil{\beta_{4} - \beta_{1}}
      +
      \ldil{\beta_{2} - \beta_{3} + \mu_{2} - \mu_{1}}
      \\
      &\phantom{=}
      -
      \ldil{\beta_{4} - \beta_{3} + \mu_{2}}
      -
      \ldil{\beta_{2} - \beta_{1} - \mu_{1}}
      \\
      &\phantom{=}
      -\mu_{1}(\beta_{3} - \beta_{1})
      +\mu_{2}(\beta_{4} - \beta_{2})
      + \mu_{1}\mu_{2}
  \end{aligned}
  \\
  \label{eq:octahedral function negative}
  \begin{aligned}
      \octfunc{-}{\mu}{\beta}
      &\defeq
      \ldil{\beta_{3} - \beta_{4} + \mu_{1}}
      +
      \ldil{\beta_{1} - \beta_{2} - \mu_{2}}
      \\
      &\phantom{=}
      -
      \ldil{\beta_{2} - \beta_{4}}
      -
      \ldil{\beta_{3} - \beta_{2} + \mu_{1} - \mu_{2}}
      \\
      &\phantom{=}
      +\mu_{1}(\beta_{3} - \beta_{1})
      -\mu_{2}(\beta_{4} - \beta_{2})
      - \mu_{1}\mu_{2}
  \end{aligned}
\end{gather}

\begin{definition}
  \label{def:CS braiding}
  Let \(X = (X, \chi, \mu)\) be a colored crossing of sign \(\epsilon\).
  Write \(\chi_{i} = (a_{i}, b_{i})\) for the parameters determining the domain and codomain vector spaces \(\qpf[]{}\).
  We define
  \[
    \csfunc{X, \chi, \mu}
    \colon
    \csfunc{\chi_{1}}
    \otimes
    \csfunc{\chi_{2}}
    \otimes
    \to
    \csfunc{\chi_{4}}
    \otimes
    \csfunc{\chi_{3}}
  \]
  to be the linear map
  \begin{equation}
    \label{eq:CS braiding}
    \ket{\beta_{1}} \otimes \ket{\beta_{2}} 
    \mapsto
    \exp \leftfun( \octfunc{\epsilon}{\mu}{\beta} \rightfun)
    \ket{\beta_{4}} \otimes \ket{\beta_{3}} 
  \end{equation}
\end{definition}

As discussed below this braiding is still well-defined at a pinched crossing.
At such a crossing the segment parameters are \( b_{2} = b_{1} m_{1}\), \( b_{3} = b_{1} + m_{2}\),  and \(b_{4} = b_{2} \), so the natural choice of logarithms is 
\[
  \beta_{2} = \beta_{1} + \mu_{1}, \beta_{3} = \beta_{1} + \mu_{2}, \beta_{4} = \beta_{1}
\]
and then
\[
  \octfunc{\epsilon}{\mu_{1}, \mu_{2}}{\beta_{1}, \beta_{1} + \mu_{1}, \beta_{1} + \mu_{2}, \beta_{1}}
  =
  \epsilon \mu_{1} \mu_{2}
\]

\begin{theorem}
  \label{thm:CS braiding is defined}
  \(\csfunc{X, \chi, \mu}\) is a well-defined linear map for any coloring \((\chi, \mu)\).
  It depends on the log-meridians as
  \begin{equation}
    \label{eq:csfunc crossing mu dep}
    \csfunc{X, \chi, \mu_{1} + k_{1}, \mu_{2} + k_{2}}
    =
    \ell_{1}^{2 k_{1}}(X, \chi) \ell_{2 k_{2}}^{2}(X, \chi)
    \csfunc{X, \chi, \mu_{1}, \mu_{2}}
  \end{equation}
  for the longitudes \(\ell_{i}^{2}\) of \cref{def:longitude eigenvalues}.
  These maps are holomorphic sections of a line bundle over the space \(\admvar{X}\) of admissible colorings.
\end{theorem}

\begin{proof}
  By \cref{thm:dilog is meromorphic} the matrix coefficients \(e^{\octfunc{+}{}{}}, \) and \(e^{\octfunc{-}{}{}}\) are meromorphic functions on \(\mathbb{C}^{6}\).
  Thought of as functions on \(\mathbb{C}^{6}\) they can have poles when the logarithms
  \[
    \zeta_{\lN}^{0}
    =
    \beta_{4} - \beta_{1} ,
    \quad
    \zeta_{\lW}^{0}
    =
    \beta_{2} - \beta_{1} - \mu_{1}
    ,\quad
    \zeta_{\lS}^{0} = \beta_{2} - \beta_{3} + \mu_{2} - \mu_{1}
    , \quad
    \zeta_{\lE}^{0}
    =
    \beta_{4} - \beta_{3} + \mu_{2}
  \]
  of the shape parameters (\ref{eq:shape-N}--\ref{eq:shape-E}) are integers.
  We need to check that none of these poles occur on the subspace corresponding to admissible colorings of \(X\).
  As discussed in \cref{sec:shadow colorings}, the \(\zeta_{i}^{0}\) are integers (i.e.\ the \(z_{i}^{0}\) are \(1\)) only at pinched crossings, where all four are integers simultaneously.
  It is not hard to use \cref{eq:ldil-0} to show that the poles of the four \(e^{\ldil{}}\) terms cancel so \(e^{\octfunc{\epsilon}{}{}}\) is well-defined at every admissible coloring.

  Next we need to check that the values of \(\csfunc{X}\) have the right quasi-periodicity behavior.
  This follows from checking that the expression \eqref{eq:CS braiding} transforms correctly when we change the \(\beta_{i}\) by integers, and in turn this follows from \cref{thm:segment relations}.

  For example, suppose we have a positive crossing.
  Label the segment and region parameters \(b_{i}\) and \(a_{j}\) as in \cref{fig:crossing-regions}.
  Let \(\beta = (\beta_{1}, \dots, \beta_{4})\) be a choice of logarithms of the segment parameters and \(\beta'\) be the same choice except that  \(\beta_{1}' = \beta_{1} + 1\).
  Then by \cref{eq:ldil-0} 
  \begin{align*}
    e^{
      \octfunc{+}{\mu}{\beta'}
    }
    &=
    e^{
      \octfunc{+}{\mu}{\beta}
    }
    m_{1}
    \frac{
      1 - b_{2}/b_{1}m_{1}
      }{
      1 - b_{4}/b_{2}
    }
  \end{align*}
  while
  \[
    \ket{\beta_{1} + 1} = a_{\lW}/a_{\lN} \ket{\beta_{1}}.
  \]
  The required equation 
  \[
    \frac{a_{\lW}}{a_{\lN}}
    =
    m_{1}
    \frac{
      1 - b_{2}/b_{1}m_{1}
    }{
      1 - b_{4}/b_{2}
    }
  \]
  follows from \cref{thm:segment relations}.
  The other cases are similar.

  The dependence on the log-meridians follows from similar computations.
  For example, at a positive crossing replacing \(\mu_{1}\) with \(\mu_{1} + 1\) gives a factor of
  \[
    \frac{
      1 - z_{\lW}^{0}
      }{
      1 - z_{\lS}^{0}
    }
    m_{2}
    \frac{
      b_{1}
      }{
      b_{3}
    }
    =
    \frac{
      a_{\lS}
      }{
      a_{\lW}
    }
    \frac{
      b_{1}
      }{
      b_{3}
    }
  \]
  which by \cref{ex:longitude} is the value of \(\ell_{1}^{2}\) on a positive crossing.

  Finally, we explain why the braidings are sections of a line bundle.
  For any coloring \(\chi\) of \(X\) there is a vector space of linear maps
  \[
    \csfunc{\chi_{1}}
    \otimes
    \csfunc{\chi_{2}}
    \otimes
    \to
    \csfunc{\chi_{4}}
    \otimes
    \csfunc{\chi_{3}}
  \]
  depending on a choice of log-meridians that transform as in \eqref{eq:csfunc crossing mu dep}.
  Together these give a line bundle over \(\admvar{X}\) and \(\chi \mapsto \csfunc{X, \chi}\) is a section of this line bundle.
  Because the \(\exp \octfunc{\pm}{\mu}{\beta} \) are holomorphic functions of the parameters \(\mu, \beta\) it is a holomorphic section.
\end{proof}

As before we can think of think of the braiding as a section of a line bundle \(\csfunc{X} \to \admvar{X}\) over the space of admissible shadow colorings of a crossing.
The segment log-parameters \(\beta_{i}\) are ``local coordinates'' for the endpoints of then tangle, while the meridian-log parameters \(\mu_{ij}\) correspond to the boundary of a regular neighborhood of \(X\).

\subsection{Recovering the standard Chern-Simons invariant}
\label{sec:recovering Neumann}

Before we prove that \(\csfunc{}\) is a tangle invariant we explain why it recovers the Chern-Simons invariant.
Let \(\tau\) be an ideal triangulation of a cusped \(3\)-manifold \(M\).
Choosing shape parameters \(z_{i}^{j}\) indexed by the tetrahedra \(i\) of \(\tau\) satisfying Thurston's gluing equations determines a hyperbolic structure on \(M\).
\textcite{Neumann2004} explains how to compute the Chern-Simons invariant of \(M\) using the additional data of a \defemph{flattening}, which is a choice of logarithms of the \(z_{i}^{0}\) and \(z_{i}^{1}\) satisfying a logarithmic form of the gluing equations.
The usual gluing equations say the product of shapes around each edge class is \(1\), so the logarithmic gluing equations require that the sum of the logarithms is \(0\).
In our normalization the Chern-Simons invariant of a tetrahedron with flattening \(\zeta^{0}, \zeta^{1}\) is
\begin{equation}
  \label{eq:ldil Rogers}
  S(\zeta^{0} ,\zeta^{1})
  \defeq
  \ldil{\zeta^{0}} + \pi \ii \zeta^{0} \zeta^{1}
\end{equation}
not simply \(\ldil{\zeta^{0}}\).
The extra term \(\pi \ii \zeta^{0} \zeta^{1}\) has some factors that cancel out over the triangulation (as in the proof of \cref{thm:CS is an invariant}) and some that combine with \(\psi\) to give the quadratic terms in the octahedral functions \(\octfunc{\pm}{}{}\) of \cref{eq:octahedral function positive,eq:octahedral function negative}.

Let \(D = (D,\rho,u)\) be a shadow-colored diagram of a \emph{link} \(L\).
Suppose that no crossing of \(D\) is pinched.
Then if \(\tau_{D}\) is the associated octahedral decomposition of \(\extr{L}\) \cref{thm:shape parameters} says that the shapes (\ref{eq:shape-N}--\ref{eq:shape-E}) determined by \(\rho\) and \(u\) give a solution of Thurston's gluing equations for \(\tau_{D}\).
Choose log-meridians \(\mu_{i}\) for \(D\).
Choose logarithms \(\alpha_{j}\) of the region parameters of \(D\) and \(\kappa_{X}\) of the parameters \(K_{X}\) occurring at each crossing \(X\).
We can use these to determine log-parameters
\begin{align}
  \label{eq:flattening-N}
  \zeta_{\lN}^0
  &=
  \epsilon (\beta_{2'} - \beta_1)
  &
  \zeta_{\lN}^1
  &=
  ( \kappa_{X} - \alpha_N)
  \\
  \label{eq:flattening-W}
  \zeta_{\lW}^0
  &=
  \epsilon(\beta_{2} - \beta_{1} - \mu_1 )
  &
  \zeta_{\lW}^1
  &=
  ( \kappa_{X} - \alpha_{\lW} + \epsilon \mu_1)
  \\
  \label{eq:flattening-S}
  \zeta_{\lS}^0
  &=
  \epsilon(\beta_{2} - \beta_{1'} + \mu_2 - \mu_1 )
  &
  \zeta_{\lS}^1
  &=
  (\kappa_{X} - \alpha_{\lS} + \epsilon (\mu_1 - \mu_2))
  \\
  \label{eq:flattening-E}
  \zeta_{\lE}^0
  &=
  \epsilon(\beta_{2'} - \beta_{1'} + \mu_2 )
  &
  \zeta_{\lE}^1
  &=
  ( \kappa_{X} - \alpha_{\lE} - \epsilon \mu_2)
\end{align}
for each tetrahedron at each crossing \(X\), hence for all of \(\tau_{D}\).
It is obvious that \(e^{\tu \zeta_{i}^{j}} = z_{i}^{j}\) for the shape parameters (\ref{eq:shape-N}--\ref{eq:shape-E}).

\begin{lemma}
  \label{thm:flattening}
  The logarithms of the shape parameters defined above give a flattening of \(\tau_{D}\).
\end{lemma}

\begin{proof}
  The parameters \(a_{i}, b_{j}, m_{k}, K_{X}\) are a version of the deformed Ptolemy coordinates \cite{Garoufalidis2015,Zickert2016} of \(\tau_{D}\), so taking logarithms as above will automatically give a flattening.
  One can see this directly by using the same arguments as in \cite[Theorem 4.2]{McphailSnyder2022hyperbolicstructureslinkcomplements}.
\end{proof}

\begin{proposition}
  \label{thm:dilogarithm sums agree}
  The dilogarithm sum \(S(\tau_{D})\) associated to the flattened ideal triangulation \(\tau_{D}\) agrees with \(\csfunc{}\):
  \[
    \csfunc{D, \mu}
    =
    \exp \leftfun( S(\tau_{D}) + \psi \rightfun(D))
  \]
  where \(\psi\) is the correction term discussed in \cref{sec:intro Chern-Simons}.
\end{proposition}

\begin{proof}
  Given our flattening the ideal octahedron at a crossing of sign \(\epsilon\) has Chern-Simons invariant
  \[
    \epsilon
    \left[
      S(\zeta^{0}_{\lN}, \zeta_{\lN}^{1})
      +
      S(\zeta^{0}_{\lS}, \zeta_{\lS}^{1})
      -
      S(\zeta^{0}_{\lW}, \zeta_{\lW}^{1})
      -
      S(\zeta^{0}_{\lE}, \zeta_{\lE}^{1})
    \right]
  \]
  As discussed in the introduction we want to adjust by the correction term \(\psi\).
  Using the formula for \(\ell^{2}\) in \cref{sec:Shadow cocycles} we see that the log-longitudes for the parts of the link at the crossing are
  \begin{align*}
    \lambda_{1} 
    &=
    \frac{
      \epsilon (\beta_{3} - \beta_{1}) + \alpha_{\lS} - \alpha_{\lW}
      }{
      2  
    }
    \\
    \lambda_{2} 
    &=
    \frac{
      \epsilon (\beta_{2} - \beta_{4}) + \alpha_{\lE} - \alpha_{\lS}
      }{
      2  
    }
  \end{align*}
  Thus, after adding \(\psi \defeq \tu (\mu_{1} \lambda_{1} + \mu_{2} \lambda_{2})\) we can now compute that
  \begin{align*}
    &\epsilon \left[
      S(\zeta^{0}_{\lN}, \zeta_{\lN}^{1})
      +
      S(\zeta^{0}_{\lS}, \zeta_{\lS}^{1})
      -
      S(\zeta^{0}_{\lW}, \zeta_{\lW}^{1})
      -
      S(\zeta^{0}_{\lE}, \zeta_{\lE}^{1})
    \right]
    +
    \psi
    \\
    &=
    \epsilon \left[
      \ldil{\zeta^{0}_{\lN}}
      +
      \ldil{\zeta^{0}_{\lS}}
      -
      \ldil{\zeta^{0}_{\lW}}
      -
      \ldil{\zeta^{0}_{\lE}}
    \right]
    +
    \epsilon \pi \ii \left[
      \zeta_{\lN}^{0} \zeta_{\lN}^{1}
      +
      \zeta_{\lS}^{0} \zeta_{\lS}^{1}
      -
      \zeta_{\lW}^{0} \zeta_{\lW}^{1}
      -
      \zeta_{\lE}^{0} \zeta_{\lE}^{1}
    \right] + \psi
    \\
    &=
    \octfunc{\epsilon}{\mu}{\beta}
    +
    \epsilon \pi \ii
    \left[
      \alpha_{\lN} \zeta_{\lN}^{0}
      +
      \alpha_{\lS} \zeta_{\lS}^{0}
      -
      \alpha_{\lW} \zeta_{\lW}^{0}
      -
      \alpha_{\lE} \zeta_{\lE}^{0}
    \right]
    +
    \pi \ii \left[
      \mu_{1}(\alpha_{\lS}  - \alpha_{\lW})
      +
      \mu_{2}(\alpha_{\lE}  - \alpha_{\lS})
    \right]
  \end{align*}
  When summing this over every crossing in the diagram we get a logarithm of \(\csfunc{D, \mu}\) plus some extra terms: we explain why these cancel.

  The terms
  \(
      \mu_{1}(\alpha_{\lS}  - \alpha_{\lW})
  \)
  and
  \(
      \mu_{2}(\alpha_{\lE}  - \alpha_{\lS})
  \)
  cancel for the same reason the region parameters drop out of the product in \cref{thm:longitude eigenvalues make sense}.

  If we group the terms \(\pm \alpha_{} \zeta_{}^{0}\) by region each region \(j\) of \(D\) corresponds to a term of the form
  \(
    \alpha_{j} \sum_{k} \epsilon_{k} \zeta_{k}^{0}
  \)
  where \(k\) ranges over the corners of \(j\) and the \(\epsilon_{k}\) are signs depending on the orientations of the edges at the corner.
  As in \cite[Theorem 4.2]{McphailSnyder2022hyperbolicstructureslinkcomplements} one can work out that
  \(
    \sum_{k} \epsilon_{k} \zeta_{k}^{0} = 0
  \)
  is the gluing equation  for the edge class associated to the region \(j\).
  It holds by \cref{thm:flattening} and the coefficient of \(\alpha_{j}\) vanishes as required.
\end{proof}

\begin{example}
  \label{example:CS boundary-parabolic}
  Suppose \(\rho \in \repvar{L}\) is a boundary-parabolic representation of a link \(L\).
  In this case there is a unique decoration and the meridian and longitude eigenvalues are both \(\pm 1\).
  Thus \(\ell_{i}^{2} = 1\) for every component of \(L\), so by \cref{thm:CS is an invariant} below \(\csfunc{L, \rho, \mu}\) is independent of the choice of \(\mu\).
  The correction term \(e^{\psi}\) is a power of \(\ii\), and the complex number \(\tu \log \csfunc{L, \rho} \in \mathbb{C}/ \pi^{2} \mathbb{Z}\) is exactly the \(\pslg\) complex volume \(\operatorname{cvol}\) of \textcite{Inoue2014}.
  \(\log \csfunc{L, \rho}\) is defined modulo \(\tu\) and not \(\pi \ii/2\) because we added the correction term \(\psi\).
  This lift is related to using a \(\slg\) (not \(\pslg\)) representation and using the dilogarithm \eqref{eq:ldil Rogers} corresponding to an extension \cite{zbMATH05220947} of the Bloch group used by \textcite{Neumann2004}.
\end{example}

\begin{example}
  \label{example:CS Dehn filling}
  Suppose Dehn surgery along the framed, oriented link \(L\) gives a \(3\)-manifold \(M\).
  If \(\rho\) maps the distinguished longitudes \(\lon_i\) of \(L\) to the identity matrix it induces a representation \(\rho : \pi_{1}(M) \to \slg\).
  Again this means that \(\ell_{i}^{2} = 1\) for every component of \(L\) and \(\csfunc{L, \rho, \mu}\) does not depend on \(\mu\).
  Now \(2 \pi \log \csfunc{L, \rho}\) is the complex Chern-Simons invariant of \((M, \rho)\) as in \cite[Theorem 14.5]{Neumann2004}.
  Note that even though \(\ell_{i}^{2} = 1\) the filling correction \(\psi\) is not necessarily \(0\) as the log-longitudes induced by a choice of logarithms \(\beta_{j}\) of the segment parameters may be nonzero integers.
\end{example}

When \(\rho\) is a general representation of \(\extr{L}\) that may not admit a Dehn filling we interpret \(\csfunc{L, \rho, \mu}\) as the complex Chern-Simons invariant of a flat connection \(\theta\) with holonomy \(\rho\) on \(\extr{L}\) whose behavior on \(\partial \extr{L}\) is determined by \(\mu\).
To make this interpretation precise we would need to work out precisely how the log-meridians determine boundary conditions for flat connections and do some integrals.
This should be possible by extending the methods of \textcite{Kirk1993} and \textcite{Marche2012}.
For links some unpublished work of Goerner and Zickert defines a similar invariant of link exteriors [Zickert, personal communication, 2021].

\subsection{\texorpdfstring{\(\csfunc{}\)}{Iψ} is a tangle invariant}

\begin{theorem}
  \label{thm:CS is an invariant}
  Let \(D\) be a tangle diagram and \((\rho, u)\) an admissible shadow coloring of \(D\).
  Then for any choice of log-meridians \(\mu\)
  \[
    \csfunc{D, \rho, u, \mu}
  \]
  is an isotopy invariant of \(D\), so we can write it as a function of the underling colored tangle \((T, \rho, u, \mu)\).
  It is gauge-invariant: if \((\rho', u')\) is obtained from \((\rho, u)\) by a gauge transformation, then
  \(
    \csfunc{T, \rho', u', \mu}
    =
    \csfunc{T, \rho, u, \mu}
    .
  \)
  It depends on choice of the log-meridians as
  \begin{equation}
    \label{eq:CS log meridian dep}
    \csfunc{T, \rho, u, \mu'}
    =
    \csfunc{T, \rho, u, \mu}.
    \prod_{i} \ell_{i}^{2 \Delta \mu_{i}}(T, \rho, u)
  \end{equation}
  where \(\Delta \mu_{i}= \mu_{i}' - \mu_{i}\).
\end{theorem}

\begin{proof}
  The claim about log-meridian dependence follows from \cref{thm:CS braiding is defined} and the definition of \(\ell_{i}^{2}\).
  By \cref{thm:models give invariants} the other claims follow from showing that \(\csfunc{}\) is a model of \(\tcat\).
  It is obviously regular and simple, so we need to prove it sends the Reidemeister move diagrams to identity maps.
  It is elementary to check this for the \(\reid{1f}\), \(\reid{2}\), and \(\reid{2}[+-]\) diagrams.

  The \(\reid{3}\) diagram is nontrivial and follows from the five-term relation for the lifted dilogarithm.
  For any shadow coloring \((\rho, u)\) of \(\reid{3}\) we have \(\ell_{i}^{2} = 1\) for every component, so \(\csfunc{}\) is independent of the choice of log-meridian.
  Suppose that \((\rho, u)\) is not pinched.
  Using the same reasoning as in the proof of \cref{thm:dilogarithm sums agree} we conclude that \(\log \csfunc{\reid{3}, \rho, u}\) is equal to the sum of dilogarithms \cref{eq:ldil Rogers} associated to the octahedral decomposition \(\tau_{\reid{3}}\).
  Because the underlying tangle of the \(\reid{3}\) diagram is isotopic to three unlinked strands there is a series of \(3\)-\(2\) moves (which the lifted dilogarithm \eqref{eq:ldil Rogers} is invariant under) that cancel all of these.
  We conclude that \(\csfunc{\reid{3}, \rho, u} = 1\) for every non-pinched shadow coloring \((\rho, u)\) of \(\reid{3}\).

  As discussed above \(\csfunc{\reid{3}, \rho, u, \mu}\) is independent of the choice of \(\mu\).
  Its value at each \((\rho, u)\) is an endomorphism of a one-dimensional vector space, so we can canonically identify it with a scalar and thus identify \((\rho, u) \mapsto \csfunc{\reid{3}, \rho, u}\) with a holomorphic function on \(\admvar{\reid{3}}\).
  By \cref{thm:reid3 is connected} this space is path-connected.
  Furthermore  \(\csfunc{\reid{3}}\) is constant and equal to \(1\) on the full-dimensional subset of non-pinched shadow colorings.
  By analytic continuation we conclude \((\rho, u) \mapsto \csfunc{\reid{3}, \rho, u}\) is the constant function \(1\), which is the same as saying \(\csfunc{\reid{3}}\) is the identity map for every coloring of  \(\reid{3}\).
\end{proof}

\section{The quantum invariant \texorpdfstring{\(\qfunc{}\)}{𝒵ψ}}
\label{sec:construction of invariant}

In this section we define \(\qfunc{}\) and \(\qinv{}\) using the representation theory of the quantum group \(\qgrplong\).
We use the relationship between the braiding of \(\qgrplong\), our presentation of it in terms of a Weyl algebra, and the octahedral coordinates (i.e.\ the segment and region parameters) in an essential way.
These are discussed in greater detail in \cite{McphailSnyder2022hyperbolicstructureslinkcomplements,McphailSnyder2024octahedralcoordinateswirtingerpresentation,McPhailSnyderAlgebra}.
Here we emphasize the parts of the story needed to understand the definition of \(\qfunc{}\) and its relationship with Chern-Simons theory.

We begin by defining the quantum group \(\qgrplong\) and an embedding of a Weyl algebra \(\phi : \weyl \to \qgrplong\).
All the \(\qgrplong\)-modules used in this paper are \(\weyl\)-modules with the \(\qgrplong\) action defined by \(X \cdot v \defeq \phi(X) \cdot v\).
At \(q = \omega\) a \(\nr\)th root of unity \(\weyl\) has a natural family \((a, b, \mu) \mapsto \qpf{a, b, \mu}\) of \(\nr\)-dimensional modules with action given by \(\omega\)-difference operators.
The elements of \(\qpf{}\) are quasi-periodic functions.

We define the value of \(\qfunc{}\) on a point to be one of the modules \(\qpf{a,b, \mu}\) with parameters determined by the coloring.
Braidings between these modules were previously defined by the author and Reshetikhin \cite{McPhailSnyderAlgebra}.
The braidings admit a simpler description in terms of the modules \(\qpf{a,b, \mu}\) because quasi-periodic functions naturally capture the algebraic behavior of quantum dilogarithms.
This perspective also shows how to recover (the inverse of) \(\csfunc{}\) by setting \(\nr = 1\).
In this section we focus on the braidings at geometrically nondegenerate crossings: the \defemph{pinched} case is described in \cref{sec:The pinched limit}.

The values of \(\qfunc{}\) on points and crossings and the pivotal structure of \(\qcat\) show how to associate  a morphism \(\qgrplong\)-modules to any admissible colored tangle diagram.
To show that this is an invariant we need to prove \(\qfunc{}\) respects colored Reiemeister moves (i.e.\ is a \defemph{model} of \(\tcat\)).
We delay the proof to \cref{sec:proofs}.
In the remainder of this section we discuss the properties of the tangle invariants \(\qfunc{}\) and their associated link invariants  \(\qinv{}\).
\Cref{sec:State-sum description} works out some concrete examples.

\subsection{The quantum group and the Weyl algebra}
\label{sec:algebras}

\begin{definition}
  \defemph{Quantum \(\mathfrak{sl}_{2}\)} is the algebra \(\qgrp[q] = \qgrplong[q]\) over \(\mathbb{C}[q^{1/2}, q^{-1/2}]\) with generators \(K, K^{-1}, E, F\) subject to the relations%
  \note{%
    Our normalization is different from the standard one but is equivalent over a ring where \((q^{1/2} - q^{-1/2})\) is invertible.
    It could also be called the \defemph{quantized function algebra} of \(\slg\) because in this normalization \(q = 1\) recovers the algebra of functions on the Poisson dual group \(\slg^{*}\).
    In the standard normalization \(q = 1\) recovers the enveloping algebra \(\mathcal{U}(\sla)\).
  }
  \begin{align*}
    KE &= q EK \\
    KF &= q FK \\
    [E,F] &= (q^{1/2} - q^{-1/2}) (K - K^{-1}).
  \end{align*}
  It is a Hopf algebra, with the coproduct
  \[
    \Delta(K)=K\otimes K, \ \ \Delta(E)=E\otimes K+1\otimes E, \ \ \Delta(F)=F\otimes 1+ K^{-1}\otimes F
\]
  and antipode
  \[
    S(K) = K^{-1}, \, \, S(E) = - EK^{-1}, \, \, S(F) = -KF.
  \]
  The center of \(\qgrp[q]\) is freely generated by the Casimir element
  \[
    \Omega = EF + q^{-1/2} K + q^{1/2} K^{-1}.
  \]
\end{definition}

Set \(q = \omega = e^{2 \pi i /\nr}\).
Unlike for generic \(q\) there is now a large central Hopf subalgebra \(Z_{0}\) generated by \(K^{\nr}, E^{\nr}, F^{\nr}\) \cite{de1991representations}.
The algebra \(Z_{0}\) is (the algebra of functions on) an algebraic group \(\slg^{*}\) (the \defemph{Poisson dual group}) closely related to \(\slg\).
Representations of \(\qgrp\) are naturally graded by \(\operatorname{Spec} Z_{0}\), which leads to our grading by decorated matrices and shadow colorings.
The full center is \(Z_{0}[\Omega]\) modulo a single relation given by a Chebyshev polynomial \cite[eq.\ 27]{Blanchet2018}.

\begin{definition}
  \label{def:weight modules}
  A \defemph{weight module} is a finite-dimensional \(\qgrp\)-module on which the center acts diagonalizably.
  We denote the category of weight modules by \(\qcat\).
  It is a monoidal category, and we make it a pivotal category by choosing the pivot \(\varpi = K^{1- \nr}\).
\end{definition}

\begin{remark}
  \label{rem:pivot}
  The standard choice of pivot for  \(\qgrp[q]\) is \(\varpi = K\).
  It is known \cite[Appendix A.2]{zbMATH01735009} that to obtain a pivotal structure on \(\qgrp\) admitting a ribbon structure (in particular, respecting the \(\reid{1fr}\) move) one needs to choose \(\varpi = K^{1 - \nr}\).
  One explanation for this phenomenon is that our choice recovers the standard pivotal structure on \(\vect\) when \(\nr = 1\), just as \(\qfunc{}\) recovers (the inverse of) \(\csfunc{}\) when \(\nr = 1\) (\cref{sec:recovering the chern-simons invariant}).
\end{remark}

\begin{definition}
  The \defemph{extended Weyl algebra} \(\weyl[q]\) is the algebra over \(\mathbb{C}[q^{1/2}, q^{-1/2}]\) with invertible generators \(x,y, z\) subject to the relations
  \begin{align*}
    xy &= q yx
       &
    xz &= zx
       &
    yz &= zy
  \end{align*}
  We abbreviate \(\weyl = \weyl[\omega]\).
\end{definition}

\begin{lemma}
  \label{thm:phi is defined}
  There is an algebra homomorphism $\phi : \qgrp[q] \to \weyl[q]$ acting on generators as 
  \begin{align*}
    \phi(K) &=x,
            &
    \phi(E) &= q^{1/2}y(z-x),
            &
    \phi(F)&=y^{-1}(1-z^{-1}x^{-1})
  \end{align*}
  so that the Casimir is sent to
  \[
    \phi(\Omega) = q^{1/2} z + q^{-1/2} z^{-1}.
    \qedhere
  \]
\end{lemma}

We think of \(\phi\) as giving a presentation of  \(\qgrp[q]\) in terms of \(\weyl[q]\).
In this paper we work almost exclusively with \(\qgrp[q]\)-modules arising as \(\weyl[q]\)-modules:
if  \(V\) is a \(\weyl[q]\)-module, then we can make it a \(\qgrp[q]\)-module by
\[
  X \cdot v \defeq \phi(X) v
\]

\subsection{Points}%
\label{sec:Modules}

Next we define the value of \(\qfunc{}\) on a point, which will be a \(\qgrp\)-module.
It is natural to describe these via the natural action of \(\weyl\) on spaces of quasi-periodic functions as in \cref{sec:The Chern-Simons invariant of a tangle}.

\begin{definition}
  \label{def:Weyl modules N}
  For \(a, b \in \mathbb{C}^{\times} \) and \(\mu \in \mathbb{C}\) let \(\qpf{a,b, \mu}\) be the vector space of functions
  \[
    f \colon \set{\beta \in \mathbb{C} \given e^{\tu \beta} = b} \to \mathbb{C} 
  \]
  on the lattice of logarithms of \(b\) that are quasi-periodic with constant \(a\) and period \(\nr\):
  \[
    f(\beta + \nr) = a f(\beta).
  \]
  It is clear that
  \begin{align*}
    (x \cdot f)(\beta) &= f(\beta + 1)
    &
    (y \cdot f)(\beta) &= \omega^{\beta} f(\beta)
    &
    (z \cdot f)(\beta) &= \omega^{\mu} f(\beta)
  \end{align*}
  makes \(\qpf{a,b,\mu}\) a \(\weyl\)-module.
  Its isomorphism class is determined by \((a, b, \omega^{\mu})\), which we can identify with the central character \(\chi : Z(\weyl) \to \mathbb{C}\) defined by 
  \begin{align*}
    \chi(x^{\nr}) &= a
    &
    \chi(y^{\nr}) &= b
    &
    \chi(z) &= \omega^{\mu}.
  \end{align*}
\end{definition}
We write \(\ket{a, \beta}\) for the function \(\ket{a, \beta}(\beta) = 1\) and \(\ket{a, \beta}(\beta') = 0\) if \(\beta \not \equiv \beta' \pmod{\nr}\).
These satisfy 
\begin{equation}
  \label{eq:ket qp relation}
  \ket{a,\beta + \nr} = a^{-1} \ket{a, \beta}
\end{equation}
as
\(
  \ket{a, \beta + \nr}(\beta)
  =
  a^{-1}
  \ket{a, \beta + \nr}(\beta + \nr)
  =
  a^{-1}.
\)
Usually \(a\) is clear and we abbreviate \(\ket{a,\beta} = \ket{\beta}\).
For any \(\beta\) with \(e^{\tu \beta} = b\) the set \(\set{ \ket{\beta}, \ket{\beta + 1}, \dots, \ket{\beta + \nr - 1} }\) is a basis of \(\qpf{a,b, \mu}\).
In this basis the action of \(\weyl\) is given by
\begin{align*}
  x \cdot \ket{\beta} &= \ket{\beta - 1}
    &
  y \cdot \ket{\beta} &= \omega^{\beta} \ket{\beta}
    &
  z \cdot \ket{\beta} &= \omega^{\mu} \ket{\beta}
    .
\end{align*}
We conclude that \(\qpf{a, b, \mu}\) is a \(\nr\)-dimensional simple \(\weyl\)-module.%
\note{
  It is not always a simple \(\qgrp\)-module: see \cref{sec:Absolute simplicity of modules}.
}
The family of vector spaces \((a, b, \omega^{\mu}) \mapsto \qpf[]{a,b,\mu}\) is a rank \(\nr\) bundle over \((\mathbb{C}^{\times})^{3}\).

When defining \(\qfunc{}\) it is most natural to use the basis \(\ket{\beta}, \dots, \ket{\beta + \nr - 1}\) of the \(\weyl\)-module \(\qpf{a,b, \mu}\).
The earlier paper \cite{McPhailSnyderAlgebra} uses a slightly different basis of these modules.
If we make the additional choice of \(\alpha\) with \(e^{\tu \alpha} = a\) then the vectors
\begin{equation}
  \label{eq:basis Weyl old}
  \widehat{v}_{n} \defeq \omega^{-n \alpha} \ket{\beta + n}
\end{equation}
are a basis of \(\qfunc{a, b, \mu}\) indexed by \(\mathbb{Z}/\nr \mathbb{Z}\) with action
\begin{align*}
  x \cdot \widehat{v}_{n}
  &=
  \omega^{\alpha} \widehat{v}_{n-1}
  &
  y \cdot \widehat{v}_{n}
  &=
  \omega^{\beta + n} \widehat{v}_{n}
  &
  z \cdot \widehat{v}_{n}
  &=
  \omega^{\mu} \widehat{v}_{n}
  .
\end{align*}
While this action is manifestly periodic (not quasi-periodic) the additional choice of \(\alpha\) turns out to be inconvenient.
One can also consider the Fourier dual vectors
\begin{equation}
  \label{eq:weight basis}
  \wb{\alpha}
  \defeq
  \sum_{k=0}^{\nr - 1} \omega^{(\beta + k) \alpha} \ket{\beta + k}
\end{equation}
which have \(\wb{\alpha + \nr} = b \wb{\alpha}\) and
\begin{align*}
  x \cdot \wb{\alpha}
  &=
  \omega^{\alpha}
  \wb{\alpha}
  &
  y \cdot \wb{\alpha}
  &=
  \wb{\alpha + 1}
  &
  z \cdot \wb{\alpha}
  &=
  \omega^{\mu} \wb{\alpha}
  .
\end{align*}

\begin{remark}
  \label{rem:cyclic modules}
  Thinking of \(\qpf{a, b, \mu}\) as a \(\qgrp\)-module \( \set{ \wb{\alpha}, \dots, \wb{\alpha - \nr + 1} }\) is a weight basis because
  \[
    K \cdot \wb{\alpha}
    =
    x \cdot \wb{\alpha}
    =
    \omega^{\alpha} \wb{\alpha}
    .
  \]
  However, since
  \[
    E \cdot \wb{\alpha}
    =
    \omega^{\mu + 1/2}(1 - \omega^{\alpha - \mu})
    \wb{\alpha + 1}
  \]
  for generic \(\chi\) \(\ker E\) will be trivial and there is no natural choice of \defemph{highest} weight.
  When \(a \ne \omega^{\pm \nr \mu}\) \(E\) and \(F\) act invertibly on the \(\qgrp\)-module \(\qpf{a, b, \mu}\) and we say it is \defemph{cyclic}.
  The Casimir acts on \(\qpf{a, b, \mu}\) by \(\omega^{\mu + 1/2} + \omega^{-\mu - 1/2}\) and the parameter \(\mu\) is the analogue of a highest weight for \(\qpf{a,b,\mu}\), although when \(\qpf{a,b,\mu}\) is cyclic it is neither a half-integer nor an eigenvalue of \(K\).
\end{remark}

\begin{definition}
  \label{def:qfunc value point}
  The functor \(\qfunc{}\) assigns a positively oriented point the module
  \[
    \qfunc{
      \begin{tikzpicture}[baseline=(g.base)]
        \draw (0,0.5) node {\(u\)};
        \draw[->] (0,0) node(g)[left] {\((g, [v], \mu)\)} to (0.1,0);
        \draw (0,-0.5) node {\(gu\)};
      \end{tikzpicture}
    }
    =
    \qpf{
      \frac{
        \det( gu, \evec{2})
        }{
        \det( u, \evec{2})
      }
      ,
      -
      \frac{
        v \evec{2}
        }{
        v u
      }
      ,
      \mu
    }
    =
    \qpf{
      \frac{
        a^{\dn}
        }{
        a^{\up}
      }
      ,
      b
      ,
      \omega^{\mu}
    }
  \]
  where \(a^{\up}, a^{\dn}\) are the region parameters of the top and bottom regions, \(b\) is the segment parameter, and \(\mu\) is the log-meridian.
  (Note the rule for region parameters differs from \cref{def:csfunc value point}.)
  We abbreviate
  \[
    \qfunc{\chi}
    = 
    \qfunc{+, \chi}
    =
    \qpf{\chi}
  \]
  to mean the value on a colored point with associated character \(\chi = (a^{\dn}/a^{\up}, b, \omega^{\mu})\) as above.
\end{definition}

The direct sum
\[
  \qfunc{g, [v]} \defeq \bigoplus_{n = 0}^{\nr -1} \qfunc{g, [v], \mu + n}
\]
does not depend on a choice of log-meridian \(\mu\) because the \(\qgrp\)-modules \(\qpf{a, b, \mu}\) depend only on \(\omega^{\mu}\).
As for the classical invariant \(\csfunc{}\) can think of \(\qfunc{\point} \to \admvar{\point}\) as a bundle of \(\qgrp\)-modules over the space of admissible shadow colorings of a point.

To define \(\qfunc{}\) on negatively-oriented points we need to describe the dual modules.
We can make \(\qpf{a,b,\mu}^{*} \defeq \operatorname{Hom}_{\mathbb{C}}(\qpf{a,b, \mu}, \mathbb{C})\) a \(\weyl\)-module by choosing an anti-automorphism \(\sigma : \weyl \to \weyl\) and defining
\[
   (w \cdot f)(v) \defeq f(\sigma(w) \cdot v) \text{ for } w \in \weyl.
\]
The natural choice is
\[
  \sigma(x) = x^{-1},
  \quad
  \sigma(y) = \omega zy,
  \quad
  \sigma(z) = \omega^{-1} z^{-1}
\]
as then \(\phi(S(X)) = \sigma(\phi(X))\) so the dual of a \(\weyl\)-module is also the dual as a \(\qgrp\)-module.
Under this action the central character of \(\qpf{a, b, \mu}^{*}\) is
\[
  \chi^{-1} \defeq (a^{-1}, b , \omega^{-\mu - 1}).
\]

As before the pivotal structure means we assign a negatively oriented point the dual module:
\[
  \qfunc{
    \begin{tikzpicture}[baseline=(g.base)]
      \draw (0,0.5) node {\(u\)};
      \draw[<-] (0,0) node(g)[left] {\((g, [v], \mu)\)} to (0.1,0);
      \draw (0,-0.5) node {\(g^{-1}u\)};
    \end{tikzpicture}
  }
  =
  \qpf{
    \frac{
      \det( u, \evec{2})
    }{
      \det(g^{-1}u, \evec{2})
    }
    ,
    -
    \frac{
      v \evec{2}
    }{
      v g^{-1} u
    }
  }^{*}
  =
  \qpf{
    \frac{
      a^{\up}
    }{
      a^{\dn}
    }
    ,
    b,
    \mu
  }^{*}
\]
Again we abbreviate this by \(\qfunc{-,\chi}\).
If we define \(\bra{\beta} \in \qpf{a,b, \mu}^{*}\) by 
\(
  \bra{\beta}( f) = f(\beta)
\)
then
\begin{equation}
  \label{eq:bra relations}
  \bra{\beta + \nr} = a \bra{\beta}
  \text{ and }
  \braket{\beta'}{\beta}
  =
  \begin{cases}
    a^{(\beta' - \beta)/\nr} & \beta' \equiv \beta \mod{\nr \mathbb{Z}}
    \\
    0 & \text{otherwise}
  \end{cases}
\end{equation}
and \(\{\bra{\beta}, \dots, \bra{\beta + \nr - 1}\}\) is a basis of \(\qpf{a, b, \mu}\).

Because the pivotal element of \(\qgrp\) is \(\varpi = K^{1 - \nr}\) the cup and cap maps are given by
\begin{align}
  \label{eq:evup}
  \evup {\chi} &: \qfunc{{-},\chi} \otimes \qfunc{{+},\chi} \to \mathbb{C} & & \bra{\beta} \otimes \ket{\beta'} \mapsto \braket{\beta}{\beta'}
  \\
  \label{eq:coevup}
  \coevup {\chi} &: \mathbb{C} \to \qfunc{{+},\chi} \otimes \qfunc{{-},\chi} & &1 \mapsto \sum_{n = 0}^{\nr -1} \ket{\beta + n} \otimes \bra{\beta + n}
  \\
  \label{eq:evdown}
  \evdown {\chi} &: \qfunc{{+},\chi} \otimes \qfunc{{-},\chi} \to \mathbb{C} & &\ket{\beta} \otimes \bra{\beta'} \mapsto \braket{\beta'}{1 - \nr + \beta}
  \\
  \label{eq:coevdown}
  \coevdown {\chi} &: \mathbb{C} \to \qfunc{{-},\chi} \otimes \qfunc{{+},\chi} & &1 \mapsto \sum_{n = 0}^{\nr -1} \bra{\beta + n} \otimes \ket{\beta + n + \nr -1}
\end{align}
where \(\chi\) is the central character determined by the diagram coloring and the signs indicate orientations.
We summarize these rules graphically in \cref{eq:cup and cap state sums}.

\subsection{Crossings}%
\label{sec:The braiding}

It is frequently said that \(\qgrp[q]\) is a quasi-triangular Hopf algebra but strictly speaking this is false.
The problem is that the universal  \(R\)-matrix \(\mathbf{R}\) is not an element of \(\qgrp[q] \otimes \qgrp[q]\), but instead of an appropriate \(\hbar\)-adic completion \(\qgrp[\hbar] \mathbin{\widehat{\otimes}} \qgrp[\hbar]\) at \(q = e^{\hbar}\).
Despite this, when \(q\) is not a root of unity the action of \(\mathbf{R}\) converges and gives \(\modcat{\qgrp[q]}\) the structure of a braided category.
The key algebraic fact is that the generators \(E, F \in \qgrp[q]\) act nilpotently on finite-dimensional modules.

When \(q = \omega\) is a root of unity this is no longer the case: there are now cyclic modules where \(E\) and \(F\) act invertibly.
The action of  \(\mathbf{R}\) on such modules may fail to converge, and indeed \(\modcat{\qgrp}\) is \emph{not} a braided category in the usual sense.
\textcite{Kashaev2004} showed to how work around this and still extract a braiding from \(\mathbf{R}\).
The conjugation action of \(\mathbf{R}\) on  \(\qgrp[\hbar]^{\otimes 2}\) gives an outer automorphism that still makes sense when specializing \(\hbar\), even if \(\mathbf{R}\) does not.
This gives an \defemph{outer \(R\)-matrix}, an algebra morphism
\[
  \mathcal{R} :
  \qgrp^{\otimes 2} 
  \to
  \qgrp^{\otimes 2}[W^{-1}]
\]
where  \(W \in \qgrp^{\otimes 2}\) is central.
The fact that \(\mathbf{R}\) intertwines the coproduct and opposite coproduct and satisfies braid relations%
\note{%
  More precisely \(\tau \mathbf{R}\) satisfies braid relations, where  \(\tau(x \otimes y) = y \otimes x\) is the flip map.
}
immediately implies the same for \(\mathcal{R}\), so we have a braiding on the algebra \(\qgrp\).

To define link invariants we want to understand the action of the braiding not on \(\qgrp\) but on \(\qgrp\)-modules.
Suppose we have an \(R\)-matrix 
\[
  R : V_{1} \otimes V_{2} \to V_{3} \otimes V_{4}
\]
intertwining \(\mathcal{R}\).%
\note{
  A braiding \(c\) will then be given by \(\tau R\).
}
Because \(\mathcal{R}\) acts nontrivially on the center of \(\qgrp^{\otimes 2 }\) the module \(V_{3}\) will in general \emph{not} be isomorphic to \(V_{1}\), nor \(V_{4}\) to \(V_{2}\).
We must specify the \(\qgrp\)-module structures on the codomain as part of the data.
This means a braiding is really a family of braidings parametrized by geometric data, or in the language of \cite{Kashaev2005,Blanchet2018,McPhailSnyderAlgebra} a \defemph{holonomy braiding}.
It requires specifying both a family of modules \(\chi_{i} \mapsto V(\chi_{i})\) \emph{and} a family of braidings between them, or in other words a bundle of \(\qgrp\)-modules over the representation variety of a point and compatible bundles of \(\qgrp\)-intertwiners over the representation variety of a crossing.
In our language we call this a model of \(\tcat\).

Next we precisely describe the intertwining property.
Let \(V_{1}, V_{2}\) be weight modules (\cref{def:weight modules}) for \(\qgrp\).
The action of the center \(Z(\qgrp) \subset \qgrp\) defines characters \(\chi_{1}, \chi_{2} : Z(\qgrp) \to \mathbb{C}\) for each.%
\note{%
  When they are \(\qgrp\)-modules arising from \(\weyl\)-modules as in \cref{sec:Modules} the \(Z(\qgrp)\)-characters come from pulling back the \(Z(\weyl)\)-characters to \(Z(\qgrp)\).
  This is discussed in detail in \cite{McPhailSnyderAlgebra}.
}
A braiding is a map
\[
  c : V_{1} \otimes V_{2} \to V_{4} \otimes V_{3}
\]
intertwining the outer \(R\)-matrix \(\mathcal{R}\) in the sense that
\begin{equation}
  \label{eq:intertwining property of braiding}
  \check{\mathcal{R}}(u) \cdot c(v) = c( u \cdot v) \text{ for }
  u \in \qgrp^{\otimes 2}, v \in V_{1} \otimes V_{2}
\end{equation}
where \(\check{\mathcal{R}} = \tau \mathcal{R}\) for \(\tau(x \otimes y) = y \otimes x\).
The image modules \(V_{3}\) and \(V_{4}\) have characters \(\chi_{3}, \chi_{4}\) defined by
\begin{equation}
  \label{eq:R action on characters}
  (\chi_{3} \otimes \chi_{4}) \mathcal{R} = \chi_{1} \otimes \chi_{2}.
\end{equation}
Because \(\mathcal{R}\) takes values in a localization \(\qgrp^{\otimes 2}[W^{-1}]\) the new characters \(\chi_{3}\) and \(\chi_{4}\) are not always well-defined, and we say \(V_{1}, V_{2}\) are \defemph{admissible} if they are.
When the \(\chi_{i}\) are determined by a shadow coloring of the crossing this condition is equivalent to admissibility of the shadow coloring.

\begin{lemma}[\protect{\cite[Theorem 3.8]{McPhailSnyderAlgebra}}]
  \label{thm:uniqueness of R}
  For any pair of admissible modules \(\qpf{\chi_{1}}, \qpf{\chi_{2}}\) the space of braidings is \(1\)-dimensional.
  More formally, if
  \[
    c, c' : 
    \qpf{\chi_{1}} \otimes \qpf{\chi_{2}}
    \to
    \qpf{\chi_{4}} \otimes \qpf{\chi_{3}}
  \]
  are two braidings satisfying \eqref{eq:intertwining property of braiding} then there is a nonzero scalar \(\Upsilon\) with \(c' = \Upsilon c\).
\end{lemma}

\begin{proof}
  Write \(V_{i} = \qpf{\chi_{i}}\).
  Because the \(V_{i}\) are simple \(\weyl\)-modules \(\check{\mathcal{R}}\) induces an isomorphism
  \[
    \operatorname{End}_{\mathbb{C}}(V_{1} \otimes V_{2})
    \to
    \operatorname{End}_{\mathbb{C}}(V_{2'} \otimes V_{1'})
  \]
  of matrix algebras.
  Any such isomorphism is inner, given by conjugation by a map \(c : V_{1} \otimes V_{2} \to V_{2'} \otimes V_{1'}\).
  Such a map is the braiding, so it is nonzero and unique up to a scalar.
\end{proof}

In \cite{McPhailSnyderAlgebra} the author and Reshetikhin used the Weyl modules \(\qpf{\chi}\) to explicitly solve \cref{eq:intertwining property of braiding} for the matrix coefficients of the braiding.
Because these coefficients were computed using the basis \eqref{eq:basis Weyl old} they depend on choices of logarithms \(\alpha, \beta\) of the parameters \(a, b\).
These can be understood as a choice of local coordinates on the vector bundle \(\chi \mapsto \qpf{\chi}\).
One must understand how the braiding depends on these choices to make sure it is well-defined.

Here we take a somewhat different approach.
By using a different description of the modules \(\qpf{\chi}\) and changing the normalization of the braiding we obtain a manifestly well-defined map.
This greatly simplifies the description of our tangle invariants: a choice of \defemph{log-coloring} as in \cite[Definition 4.2]{McPhailSnyderAlgebra} is now built into the choice of basis of the \(\qpf{\chi}\).%
\note{
  Specifically, the segment log-parameters \(\beta_{i}\) are chosen as part of the basis of the modules  \(\qpf{\chi_{i}}\).
  The region log-parameters \(\alpha_{j}\) are eliminated by choosing a different normalization of the braiding;
  this causes a trivial change to the normalization of our tangle invariants that cancels out for any link.
}

The matrix coefficients of the braiding are defined using the quantum dilogarithm \(\pf{}\) discussed in \cref{sec:quantum dilogarithms}.
The most important property is the recurrence relation
\begin{equation}
  \label{eq:pf recurrence}
  \pf{\zeta + k} =
  \frac{
    \pf{\zeta}
  }{
    (1 - \omega^{\zeta + 1})
    \cdots
    (1 - \omega^{\zeta + k})
  }
\end{equation}
which shows that \(\pf{}\) is a meromorphic function on  \(\mathbb{C}\) that satisfies the quasi-periodicity relation
\begin{equation}
  \label{eq:pf quasi periodic}
  \pf{\zeta + \nr}
  =
  \frac{
    \pf{\zeta}
    }{
    1 - \omega^{\nr \zeta}
  }
  .
\end{equation}
We abbreviate
\begin{equation}
  \pft{\zeta}
  \defeq
  \pf{\zeta + \nr - 1} \omega^{(\nr - 1)\zeta}
\end{equation}
The \defemph{braiding kernels} are the meromorphic functions of six variables \(\mu = (\mu_{1}, \mu_{2})\) and \(\beta = (\beta_{1}, \beta_{2}, \beta_{3}, \beta_{4})\)
defined by 
\begin{align}
  \label{eq:braiding kernel positive}
  \brkern{+}{\mu}{\beta}
  &\defeq
  \frac{
    1
  }{
    \nr
  }
  \frac{
    \pf{\beta_{4} - \beta_{1}}
    \pf{\beta_{2} - \beta_{3} + \mu_{2} - \mu_{1}}
  }{
    \pf{\beta_{4} - \beta_{3} + \mu_{2}}
    \pft{\beta_{2} - \beta_{1} -  \mu_{1}}
  }
  \omega^{
    \mu_{1}(\beta_{3} - \beta_{1}) 
    +
    \mu_{2}(\beta_{2} - \beta_{4})
    -
    \mu_{1} \mu_{2}
  }
  \\
  \label{eq:braiding kernel negative}
  \brkern{-}{\mu}{\beta}
  &\defeq
  \frac{
    1
  }{
    \nr
  }
  \frac{
    \pft{\beta_{3} - \beta_{4} - \mu_{2}}
    \pf{\beta_{1} - \beta_{2} +  \mu_{1}}
  }{
    \pft{\beta_{1} - \beta_{4}}
    \pft{\beta_{3} - \beta_{2} + \mu_{1} - \mu_{2}}
  }
  \omega^{
    \mu_{1}(\beta_{1} - \beta_{3})
    +
    \mu_{2}(\beta_{4} - \beta_{2})
    +
    \mu_{1} \mu_{2}
  }
\end{align}

Let \(X\) be a colored crossing of sign \(\epsilon\) oriented from left to right.
Write \(\chi_{i}\) for the central \(\weyl\)-characters associated to its segments as in \cref{sec:Modules}.
Define linear maps
\begin{equation}
  \label{eq:qfunc braiding definition}
  \begin{gathered}
    \qfunc{X}
    \colon
    \qfunc{\chi_{1}}
    \otimes
    \qfunc{\chi_{2}}
    \to
    \qfunc{\chi_{4}}
    \otimes
    \qfunc{\chi_{3}}
    \\
    \ket{\beta_{1}} \otimes \ket{\beta_{2}}
    \mapsto
    \sum_{n_{3},n_{4} = 0}^{\nr -1}
    \brkern{\epsilon}{\mu_{1}, \mu_{2}}{
      \beta_{1}, \beta_{2},
      \beta_{3} + n_{3},
      \beta_{4} + n_{4}
    }
    \ket{\beta_{4}} \otimes \ket{\beta_{3}}
  \end{gathered}
\end{equation}
It is not clear that this is well-defined at a pinched crossing because the braiding kernels could have poles.
We establish that it is and give explicit formulas for the matrix coefficients in \cref{sec:The pinched limit}.

\begin{theorem}
  \label{thm:qfunc braiding}
  \Cref{eq:qfunc braiding definition} gives a braiding: for any positive colored crossing \(X\), 
  \begin{thmenum}
    \item 
      \label{thm:qfunc braiding:defined}
      \(\qfunc{X}\) is a well-defined morphism of \(\qgrp\)-modules,
    \item  
      \label{thm:qfunc braiding:intertwines}
      it intertwines \(\check{\mathcal{R}}\) as in \cref{eq:intertwining property of braiding}, and
    \item 
      \label{thm:qfunc braiding:invertible}
      it is invertible, with inverse given by a negative crossing with the appropriate coloring.
      \qedhere
  \end{thmenum}
\end{theorem}

\begin{proof}
  \ref{thm:qfunc braiding:defined}
  The matrix coefficients \eqref{eq:braiding kernel positive} of \(\qfunc{X}\) are meromorphic functions.
  Their poles occur when the arguments of the quantum dilogarithms \(\pf{}\) are integers, hence when \(X\) is pinched.
  In \cref{sec:The pinched limit} we show that at any coloring coming from a pinched crossing the poles cancel and \(\qfunc{X}\) still makes sense.
  We also give an explicit formula for the coefficients.

  To obtain a well-defined linear map between the spaces \(\qpf{\chi_{i}}\) we need to make sure \eqref{eq:qfunc braiding definition} respects the quasi-periodicity conditions.
  This follows from quasi-periodicity relations on the braiding kernel \(\brkern{+}{}{}\).
  For example, the left-hand side of \eqref{eq:qfunc braiding definition} is independent of \(\beta_{4}\) by definition.
  For the right-hand side \(\ket{\beta_{4} + n_{4} + \nr} = a_{4}^{-1} \ket{\beta_{4} + n_{4}}\) and 
  \begin{align*}
    \brkern{+}{\mu}{\beta_{1}, \beta_{2}, \beta_{3} + n_{3}, \beta_{4} + n_{4} + \nr}
    =
    \frac{
      1 - b_{4}/b_{3} m_{2}
      }{
      1 - b_{4}/b_{1}
    }
    m_{2}^{-1}
  \end{align*}
  using \eqref{eq:pf quasi periodic}.
  By \cref{thm:segment relations}
  \[
    \frac{
      1 - b_{4}/b_{3} m_{2}
      }{
      1 - b_{4}/b_{1}
    }
    m_{2}^{-1}
    =
    \frac{
      a_{\lE}
      }{
      a_{\lN}
    }
    =
    a_{4}
  \]
  so the right-hand side is independent of \(\beta_{4}\) as well.
  The other relations of \cref{thm:segment relations} and similar computations give the relations for the remaining parameters.

  \ref{thm:qfunc braiding:intertwines}
  This is the first main result of \cite{McPhailSnyderAlgebra}.
  The central characters at the crossing are compatible with \(\check{\mathcal{R}}\) by \cref{thm:crossing parameter relations}.
  Repeating the computation in \cite[Section 4.1]{McPhailSnyderAlgebra} using the bases \(\set{\ket{\beta + n} \given n = 0, \dots, \nr - 1}\) and applying \eqref{eq:pf quasi periodic} shows that \(\qfunc{X}\) satisfies \cref{eq:intertwining property of braiding}.

  \ref{thm:qfunc braiding:invertible}
  \(\qfunc{X}\) is invertible by \cref{thm:uniqueness of R}.
  We check that the negative braiding matrix defined here is its inverse in \cref{sec:The R2 move}.
\end{proof}

One interpretation of \cref{thm:qfunc braiding}\ref{thm:qfunc braiding:defined} is that, while \(\brkern{+}{}{}\) depends on the choice of logarithms \(\beta_{i}\) (not just the segment parameters \(b_{i}\) ) this dependence is purely local and can be built into the definition of the modules assigned to points.
Understanding the dependence on the log-meridians \(\mu_{i}\) is more complicated.
Recall from \cref{sec:Modules} that we can eliminate the explicit \(\mu\) dependence of \(\qfunc{\point}\) by defining
\[
  \qfunc{g, [v]} \defeq \bigoplus_{n = 0}^{\nr -1} \qfunc{g, [v], \mu + n}
\]
so that \(\qfunc{\point}\) is a bundle of \(\qgrp\)-modules over \(\admvar{\point}\), the space of admissible shadow colorings of a (positively-oriented) point.
There is a similar vector bundle of \(\qgrp\)-morphisms over the space \(\admvar{X}\) of admissible colorings of a crossing: the fiber over a coloring is the space of morphisms between the modules assigned to the boundary points.
We can then view \(\qfunc{X}\) as as section of this bundle: choosing \(\mu_{i}\) picks out submodules and \(\qfunc{X}\) acts as defined in \eqref{eq:qfunc braiding definition}.
However, because the braiding depends nontrivially on the log-meridians one needs to define this vector bundle in terms of quasi-periodic functions of the \(\mu_{i}\) similar to the modules \(\qpf{\chi}\).
Concretely, we have

\begin{lemma}
  \label{thm:qfunc braiding mu dependence}
  Let \(X = (X, \rho, u)\) be an admissible shadow-colored crossing with log-meridians \(\mu_{1}, \mu_{2}\).
  Then for any other log-meridians \(\mu_{i}' = k_{i} \nr + \mu_{i}\)
  \[
    \qfunc{X, \mu_{1}', \mu_{2}'}
    =
    \ell^{-2 k_{1}}_{1}(X)
    \ell^{-2 k_{2}}_{2}(X)
    \qfunc{X, \mu_{1}, \mu_{2}}
  \]
  in terms of the squared longitudes of \cref{sec:Longitude eigenvalues}.
\end{lemma}

\begin{proof}
  The hypothesis that \(\mu_{i}' - \mu_{i} \in \nr \mathbb{Z}\) ensures that \( \qfunc{X, \mu_{1}', \mu_{2}'} \) and \( \qfunc{X, \mu_{1}, \mu_{2}} \) are maps between the same modules.
  The result follows from similar computations as  \cref{thm:qfunc braiding}\ref{thm:qfunc braiding:intertwines}.
  For example, at a positive crossing we have
  \begin{align*}
    \brkern{+}{\mu_{1}, \mu_{2} + \nr k_{2}}{\beta}
    =
    \brkern{+}{\mu_{1}, \mu_{2} }{\beta}
    \frac{
      1 - b_{2}/b_{1}m_{1}
      }{
      1 - b_{2}m_{2}/b_{3}m_{1}
    }
    \frac{
      b_{3}
      }{
      b_{1}
    }
    =
    \frac{
      a_{\lW}
      }{
      a_{\lS}
    }
    \frac{
      b_{3}
      }{
      b_{1}
    }
    =
    \ell^{-2}_{1}(X)
  \end{align*}
  by \cref{thm:segment relations}.
\end{proof}

\begin{remark}
  Another perspective is to introduce an action of another Weyl algebra \(\mathcal{M}\) generated by \(z, w\) with \(zw = \omega wz\).
  In terms of basis vectors \(\ket{\beta, \mu} \in \qfunc{g, [v], \mu}\) of \(\qfunc{g, [v]}\) the generators act by
  \[
    z \cdot \ket{\beta, \mu}
    =
    \omega^{\mu} \ket{\beta, \mu}
    \quad
    w \cdot \ket{\beta, \mu}
    =
    \ket{\beta, \mu + 1}
  \]
  While \(\qfunc{X}\) intertwines the \(\qgrp\)-action the \(\mu_{i}\)-dependence means that two copies of \(\mathcal{M}\) (one for each component) act nontrivally on \(\qfunc{X}\).
  It is not yet clear what algebraic structure best describes this behavior, but it might be explained by a quantization of non-normalized cluster algebras \cite{zbMATH07000309}.
\end{remark}

\subsection{Recovering the classical Chern-Simons invariant \texorpdfstring{\(\csfunc{}\)}{Iψ}}%
\label{sec:recovering the chern-simons invariant}

Since our quantization parameter is \(\omega = e^{\tu /\nr}\) setting \(\nr = 1\) should correspond to the classical case.%
\note{
  One also interprets \(\nr \to \infty\) as the semiclassical limit.
  It again recovers \(\csfunc{}\) but in a much more complicated way: this is what the Volume Conjecture and \cref{conj:our volume} are about.
}
Here we support our interpretation of \(\qfunc{}\) as a quantization of \(\csfunc{}\) by  explaining in detail why \(\qfunc[1]{}\) recovers the inverse of  \(\csfunc{}\).

By definition the underlying vector space of the \(\weyl\)-module \(\qpf[1]{a, b, \mu}\) is \(\qpf[]{a, b}\) as defined in \cref{sec:The Chern-Simons invariant of a tangle:points}.
However because of our slightly different conventions on region variables we have
\[
    \qfunc[1]{
      \begin{tikzpicture}[baseline=(g.base)]
        \draw (0,0.5) node {\(u\)};
        \draw[->] (0,0) node(g)[left] {\((g, [v], \mu)\)} to (0.1,0);
        \draw (0,-0.5) node {\(gu\)};
      \end{tikzpicture}
    }
    =
    \qpf{
      \frac{
        \det( gu, \evec{2})
        }{
        \det( u, \evec{2})
      }
      ,
      -
      \frac{
        v \evec{2}
        }{
        v u
      }
    }
    =
    \qpf[]{
      \frac{
        a^{\dn}
        }{
        a^{\up}
      }
      ,
      b
    }
\]
but
\[
    \csfunc{
      \begin{tikzpicture}[baseline=(g.base)]
        \draw (0,0.5) node {\(u\)};
        \draw[->] (0,0) node(g)[left] {\((g, [v], \mu)\)} to (0.1,0);
        \draw (0,-0.5) node {\(gu\)};
      \end{tikzpicture}
    }
    =
    \qpf{
      \frac{
        \det( u, \evec{2})
        }{
        \det( gu, \evec{2})
      }
      ,
      -
      \frac{
        v \evec{2}
        }{
        v u
      }
    }
    =
    \qpf[]{
      \frac{
        a^{\up}
        }{
        a^{\dn}
      }
      ,
      b
    }
\]
These are inverses in the sense that there is an isomorphism
\[
  \qpf[]{a, b}
  \otimes_{\mathbb{C}}
  \qpf[]{a^{-1}, b}
  \to
  \mathbb{C}, \quad
  \ket{a, \beta} \otimes \ket{a^{-1}, \beta'}
  =
  a^{\beta' - \beta}
\]
Similarly, by \cref{eq:exact-qlf-value} \(\pf[\nr]{\zeta} = e^{- \ldil{\zeta}}\), so setting \(\nr = 1\) in the braiding kernels (\ref{eq:braiding kernel positive},\ref{eq:braiding kernel negative}) recovers the inverse octahedron functions (\ref{eq:octahedral function positive}, \ref{eq:octahedral function negative}):
\begin{equation*}
  \brkern{\pm}{}{}
  =
  e^{
    -
    \octfunc{\pm}{}{}
  }
\end{equation*}
This inverse matches the change in quasi-periodicity conditions from \(\qpf[]{a,b}\) to  \(\qpf[]{a^{-1}, b}\).
We conclude that the product \(\qfunc[1]{} \boxtimes \csfunc{}\) (\cref{def:product of models}) is trivial, which is what it means for them to be inverses.

\subsection{Tangle and link invariants}%
\label{sec:Properties of qfun}

Above we showed how to interpret the colored segments and crossings of \(\tcat\) as maps of \(\qgrp\)-modules.
In other words, we have defined a pre-model \(\qfunc{}\) of \(\tcat\) in \(\qcat\), hence a \(2\)-functor (as in \cref{def:delooping})
\[
  \qfunc{} \colon \tcat \to \mathsf{B}\qcat.
\]
Our claim is that \(\qfunc{D}\) depends only on the isotopy class of the colored diagram \(D = (D, \rho, u, \mu)\).
By \cref{thm:models give invariants} this follows from showing that \(\qfunc{}\) is a model of \(\tcat\) valued in \(\qcat\).
We do this in \cref{sec:proofs}, which gives our first main result:

\begin{bigtheorem}
  \label{thm:qfunc defined and properties}
  Let \(T\) be a oriented, framed tangle, \(\rho \in \repvar{T}\) a decorated representation, and \(\mu\) a choice of log-meridians for \(\rho\).
  \begin{thmenum}
    \item
    \label{thm:qfunc defined and properties:defined}
      Let \(D, D' \in \tcat\) be any two decorated \(\slg\)-colored diagrams of \((T, \rho)\).
      Suppose that the shadow colorings of \(D\) and \(D'\) induced by coloring their topmost region by the same vector \(u\) are both admissible.
      Then
      \[
        \qfunc{D, \rho, u, \mu} = \qfunc{D', \rho', u, \mu}
      \]
      so \(\qfunc{T, \rho, u, \mu} \defeq \qfunc{D, \rho, u, \mu}\) is a colored isotopy invariant, and in particular is independent of the choice of representative diagram.
      This invariant is functorial: it respects disjoint union and composition of tangles.
    \item
      \label{thm:qfunc defined and properties:log-decoration}
      Suppose  that \(\mu, \mu'\) are two different log-meridians for \(\rho\).
      Assume that each \(k_{i} = \mu_{i}' - \mu_{i}\) is a multiple of \(\nr\).
      Then
      \begin{equation}
        \label{eq:log dec dep meridian}
        \qfunc{T, \rho, u, \mu'}=
        \qfunc{T, \rho, u, \mu}
        \prod_{i} \ell_{i}^{-2 k_{i}}(T)
        .
      \end{equation}
      for the \(\ell_{i}^{2}(T)\) of \cref{def:longitude eigenvalues}.
    \item 
      \label{thm:qfunc defined and properties:framing}
      Suppose \(T'\) is the same tangle as \(T\) but with the framing of component \(i\) increased by \(k\), with the convention that the right-hand twist of \eqref{eq:twist right hand positive} has framing \(+1\).
      Then if that component has log-meridian \(\mu_{i}\),
      \[
        \qfunc{T', \rho, u, \mu}
        =
        \omega^{\mu_{i}(\mu_{i} + 1 - \nr)}
        \qfunc{T, \rho, u, \mu}
        .
        \qedhere
      \]
  \end{thmenum}
\end{bigtheorem}

\begin{proof}
  \ref{thm:qfunc defined and properties:defined}
  By \cref{thm:models give invariants} it suffices to show \(\qfunc{}\) is a model of \(\tcat\).
  We do this in \cref{sec:proofs}.

  \ref{thm:qfunc defined and properties:log-decoration}
  This follows immediately from \cref{thm:qfunc braiding mu dependence} and the definition of  \(\ell_{i}^{2}\) in \cref{sec:Longitude eigenvalues}.

  \ref{thm:qfunc defined and properties:framing}
  This follows from the computation in  \cref{sec:The R1f move} used to show \(\reid{1f}\) invariance.
\end{proof}

Our invariant is defined not just for tangles with \(\slg\) representations \(\rho\) but requires the extra data of a decoration (a trivialization of the bundle defined by \(\rho\) near the boundary components) and log-meridians (more boundary data for a flat connection representing \(\rho\)).
Part \ref{thm:qfunc defined and properties:log-decoration} shows that \(\qfunc{}\) depends nontrivially on the log-meridians, but in a simple way.
We have a similar conjecture for the dependence on the decoration.

\begin{definition}
  Let \((T, \rho, u, \mu)\) be a colored tangle.
  Let \(x\) be a meridian of the \(i\)th component of \(T\), and suppose \(\rho(x)\) has eigenspaces \([v], [w]\), with \([v]\) the space determined by the decoration.
  The \defemph{opposite log-decoration} of the \(i\)th component has eigenspace \([w]\) and log-meridian \(\nr - 1 - \mu\).
\end{definition}

\begin{conjecture}
  \label{conj:opposite decoration}
  \(\qfunc{}\) is unchanged when taking the opposite log-decoration.
  That is, if \((T, \tilde{\rho}, u, \tilde{\mu})\) is the tangle with the opposite log-decoration on some components, then 
  \(
    \qfunc{T, \tilde{\rho}, u, \tilde{\mu}} = \qfunc{T, \rho, u, \mu}
    .
  \)
\end{conjecture}

The statement makes sense because \(\qfunc{g, [v], \mu}\) and \(\qfunc{g, [w], \nr - 1 - \mu}\) are isomorphic as \(\qgrp\)-modules, so if we change decoration on an open component both sides of the equation have the same source and target.
In \cref{sec:Example: Hopf links} we show that this conjecture holds for Hopf links.

The definition of  \(\qfunc{T}\) makes sense when \(T\) is a link, but it vanishes uniformly.
This is a consequence of an algebraic fact: the quantum dimensions%
\note{
  The trace of the identity map is called a quantum dimension, just as the trace of the identity map on an ordinary vector space is the dimension.
}
of the modules \(\qfunc{\chi}\) vanish.
Concretely, consider the diagram
\begin{center}
    \begin{tikzpicture}[line width=1, baseline=10, scale=1] 
      \draw[<-, looseness = 1.5] (0,0)  to [out =180, in=180] (0,1) ;
      \draw[->, looseness = 1.5] (0,0)  to [out =0, in=0] (0,1) ;
      \node[above] at (0,1) {\(g, [v], \mu\)};
    \end{tikzpicture}
\end{center}
Its value under \(\qfunc{}\) is the linear map \(\mathbb{C} \to \mathbb{C}\)
\[
  1
  \mapsto
  \sum_{n = 0}^{\nr -1} \ket{\beta + n} \otimes \bra{\beta + n}
  \mapsto
  \sum_{n = 0}^{\nr -1} \braket{\beta + n}{\beta + n + 1 - \nr}
  =
  0
\]
i.e.\ the scalar \(0\).
As a consequence \(\qfunc{L, \rho, u}\) will be \(0\) for any link \(L\).
We can follow the procedure in \cref{sec:Link invariants} to produce nontrivial link invariants.

\begin{bigtheorem}
  \label{thm:qinv defined and properties}
  Let \(L\) be a (framed, oriented) link, \(\rho\) a decorated representation of \(\extr{L}\), \(\mu\) a choice of log-meridians for \(\rho\), and \(K\) a component of \(L\).
  Choose a diagram \(D\) of a cut presentation of \((L, K)\) (i.e. an open tangle diagram whose closure is \(L\) and whose open component is \(K\).)
  \begin{thmenum}
    \item
    \label{thm:qinv defined and properties:defined}
      Let \(\rho'\) be any representation gauge-equivalent to \(\rho\) and \(u\) any admissible shadow coloring of \((D, \rho)\).
      Then
      \[
        \qinv{L, K ; \rho, \mu}
        \defeq
        \sclbrak{\qfunc{T, \rho', u, \mu}}
        \in \mathbb{C}
      \]
      is an isotopy invariant of \((L, K; \rho, \mu)\) that depends only on the gauge class of \(\rho\).
    \item
    \label{thm:qinv defined and properties:holomorphic}
    \(\qinv{L, K; \rho, \mu}\) depends on the log-meridians as in \eqref{eq:log dec dep meridian}.
    More abstractly, \(\qinv{L, K}\) is a section of a rank \(\nr^{c}\) holomorphic vector bundle over the space \(\repvar{L}\) of decorated representations of \(L\), where \(c\) is the number of components of \(L\).
    \qedhere
  \end{thmenum}
\end{bigtheorem}

\begin{proof}
  \ref{thm:qinv defined and properties:defined}
  In \cref{sec:Absolute simplicity of modules} we show that \(\qfunc{}\) is a simple model and then we can use \cref{thm:link invariants from model}.

  \ref{thm:qinv defined and properties:holomorphic}
  The first claim is a corollary of \cref{thm:qfunc defined and properties}\ref{thm:qfunc defined and properties:log-decoration}.
  To view \(\qinv{}\) as a section of a bundle first recall the space \(\repvarlog{L}\) of log-decorated representations.
  Fixing a diagram \(D\) to compute \(L\) the parameters \(\beta_{i}\) and  \(\mu_{j}\) are local coordinates on \(\repvarlog{L}\).
  The matrix coefficients defining \(\qinv{L, K; \rho, \mu}\) are all holomorphic functions of these so \((\rho, \mu) \mapsto \qinv{L, K; \rho, \mu}\) is a holomorphic function on \(\repvarlog{L}\).
  For each decorated representation \(\rho\) of \(L\) consider the space \(\Lambda_{\nr}(\rho)\) of functions \(f\) of the log-meridians \(\mu_{i}\) satisfying
  \[
    f(\mu_{1} + \nr k_{1}, \dots, \mu_{c} + \nr k_{c}) 
    =
    f(\mu_{1} , \dots, \mu_{c} ) 
    \prod_{i} \ell_{i}^{- 2k_{i}}(L, \rho)
  \]
  This condition is well-defined because the squared longitude \(\ell_{i}^{2}(L, \rho)\) is a link invariant by \cref{thm:longitude eigenvalues make sense}.
  Each \(\Lambda_{\nr}(\rho)\) is a \(\nr^{c}\)-dimensional vector space and they clearly define a holomorphic vector bundle over \(\repvar{L}\).
  By \cref{thm:qfunc defined and properties}\ref{thm:qfunc defined and properties:log-decoration} \(\qinv{L, K}\) is a section of this bundle and by the paragraph above it is a holomorphic section.
\end{proof}

We explain how to compute \(\qfunc{L, K; \rho, \mu}\) in practice and work out some concrete examples in \cref{sec:state-sums and examples}.
In \cref{sec:torsion} we explain why \(\qinv[2]{}\) is closely related to the Reidemeister torsion, and in \cref{sec:Relation to the Kashaev--Akutsu-Deguchi-Ohtsuki invariant} we show \(\qinv{L, K; \rho}\) recovers the link invariants of \textcite{Kashaev1995} and \textcite{Akutsu1992} when \(\rho\) has reducible (in particular, abelian) image.

\section{State-sum description of the invariant and examples}
\label{sec:state-sums and examples}

\subsection{State-sum description}%
\label{sec:State-sum description}

Here we explain how to reduce the computation of \(\qinv{L, K; \rho, \mu}\) to an explicit tensor contraction.
This description also leads to the state integral presentation studied in \cite{McPhailSnyderState}.

Recall that \(L\) is a framed oriented link and \(K\) is a component of \(L\).
Choose a diagram \(D\) of a cut presentation of \((L,K)\), i.e.\ choose a tangle diagram with one incoming and outgoing boundary strand whose closure is  \(L\) so that the open component is \(K\).
Let \(E\) be the set of internal segments of \(D\).
We can think of points of \(\mathbb{C}^{E}\) as assigning a complex variable \(\beta_{i}\) to each internal segment.
Assign the boundary segments of \(D\) the same parameter \(\beta_{0}\), which we think of as fixed.
Similarly we assign a parameter \(\mu_{i}\) to each component of \(D\).
We have not yet picked \(\rho\): it will be determined by the values of the \(\beta_{i}\) and \(\mu_{j}\).

We define a complex-valued \defemph{action functional}%
\note{
  In a more physical convention we would instead write this as \(e^{\nr \mathcal{S}_{D}}\), as in \cite{McPhailSnyderState}.
}
\(\beta \mapsto \mathcal{S}_{D}(\mu, \beta_{0}|\beta)\) by taking a product of pieces associated to elementary diagrams.
We think of the action as a family of functions of the log-segment parameters \(\beta = i \mapsto \beta_{i} \in \mathbb{C}^{E}\) parametrized by the choice of log-meridians \(\mu_{j}\).

We assign crossings the braiding kernel \(\brkern{\pm}{}{}\) of matching sign:
\begin{equation}
  \label{eq:braiding state sum}
  \begin{tikzpicture}[line width = 1, scale = 1, xscale = 1.5, baseline={(current bounding box.center)}]
    \coordinate (1) at (0,1);
    \coordinate (2) at (0,0);
    \coordinate (3) at (1,0);
    \coordinate (4) at (1,1);
    \draw[->] (2) \br (4);
    \draw[->] (1)  \br (3);
    \node[left] at (1) {\(\beta_{1}\)};
    \node[left] at (2) {\(\beta_{2}\)};
    \node[right] at (3) {\(\beta_{3}\)};
    \node[right] at (4) {\(\beta_{4}\)};
    \node[below] at (0.5,-0.5) {\(\brkern{\pm}{\mu_{1}, \mu_{2}}{\beta_{1}, \beta_{2}, \beta_{3}, \beta_{4}}\)};
  \end{tikzpicture}
\end{equation}
Cups and caps correspond to argument shifts
\begin{equation}
  \label{eq:cup and cap state sums}
  \begin{aligned}
    &
    \begin{tikzpicture}[line width=1, baseline=10, scale=1] 
      \draw[->, looseness = 1.5] (0,0) node[right] {\(\beta\)} to [out =180, in=180] (0,1) node[right] {\(\beta\)};
    \end{tikzpicture}
    &&&
    &
    \begin{tikzpicture}[line width=1, baseline=10, scale=1] 
      \draw[->, looseness = 1.5] (0,0) node[left] {\(\beta\)} to [out =0, in=0] (0,1) node[left] {\(\beta\)};
    \end{tikzpicture}
    \\
    &
    \begin{tikzpicture}[line width=1, baseline=10, scale=1] 
      \draw[<-, looseness = 1.5] (0,0) node[right] {\(\beta + 1 - \nr\)} to [out =180, in=180] (0,1) node[right] {\(\beta\)};
    \end{tikzpicture}
    &&&
    &
    \begin{tikzpicture}[line width=1, baseline=10, scale=1] 
      \draw[<-, looseness = 1.5] (0,0) node[left] {\(\beta + \nr - 1\)} to [out =0, in=0] (0,1) node[left] {\(\beta\)};
    \end{tikzpicture}
  \end{aligned}
\end{equation}
which we could also think as being delta functions.
Together these mean the rotated braidings are assigned
\begin{gather}
  \begin{tikzpicture}[line width = 1, scale = 1, yscale = 1.5, baseline={(current bounding box.center)}]
    \coordinate (1) at (0,0);
    \coordinate (2) at (1,0);
    \coordinate (3) at (1,1);
    \coordinate (4) at (0,1);
    \draw[->] (2) to [out = 90, in = 270] (4);
    \draw[->] (1)  to [out = 90, in = 270] (3);
    \node[below left] at (1) {\(\beta_{1}\)};
    \node[below right] at (2) {\(\beta_{2}\)};
    \node[above right] at (3) {\(\beta_{3}\)};
    \node[above left] at (4) {\(\beta_{4}\)};
    \node[below] at (0.5,-0.5) {\(\brkern{\pm}{\mu_{1}, \mu_{2}}{\beta_{1}, \beta_{2}, \beta_{3}, \beta_{4}}\)};
  \end{tikzpicture}
  \\
  \begin{tikzpicture}[line width = 1, scale = 1, yscale = 1.5, baseline={(current bounding box.center)}]
    \coordinate (1) at (1,1);
    \coordinate (2) at (0,1);
    \coordinate (3) at (0,0);
    \coordinate (4) at (1,0);
    \draw[->] (2) to [out = 270, in = 90] (4);
    \draw[->] (1)  to [out = 270, in = 90] (3);
    \node[above right] at (1) {\(\beta_{1}\)};
    \node[above left] at (2) {\(\beta_{2}\)};
    \node[below left] at (3) {\(\beta_{3}\)};
    \node[below right] at (4) {\(\beta_{4}\)};
    \node[below] at (0.5,-0.5) {\(\brkern{\pm}{\mu_{1}, \mu_{2}}{\beta_{1} + 1 - \nr, \beta_{2}, \beta_{3} + 1 - \nr, \beta_{4}}\)};
  \end{tikzpicture}
  \\
  \begin{tikzpicture}[line width = 1, scale = 1, xscale = 1.5, baseline={(current bounding box.center)}]
    \coordinate (1) at (1,0);
    \coordinate (2) at (1,1);
    \coordinate (3) at (0,1);
    \coordinate (4) at (0,0);
    \draw[->] (2) to[out = 180, in = 00] (4);
    \draw[->] (1) to[out = 180, in = 00] (3);
    \node[right] at (1) {\(\beta_{1}\)};
    \node[right] at (2) {\(\beta_{2}\)};
    \node[left] at (3) {\(\beta_{3}\)};
    \node[left] at (4) {\(\beta_{4}\)};
    \node[below] at (0.5,-0.5) {\(\brkern{\pm}{\mu_{1}, \mu_{2}}{\beta_{1} + 1 - \nr, \beta_{2} + 1 - \nr, \beta_{3} + 1 - \nr, \beta_{4} + 1 - \nr}\)};
  \end{tikzpicture}
\end{gather}

\begin{definition}
  Let \(L\) be a link with set of components \(C\).
  Choose a diagram \(D\) of a cut presentation of \((L;K)\) with set of internal segments  \(E\).
  The rules above assign \(D\) an \defemph{action functional}
  \[
    \mathcal{S}_{D} : \mathbb{C}^{C} \times \mathbb{C} \times \mathbb{C}^{E} \to \mathbb{C}
  \]
  of the log-meridians \(\mu_{i}, i \in C\), boundary segment log-parameter \(\beta_{0}\), and internal segment log-parameters \(\beta_{j}, j \in E\).
\end{definition}

\begin{example}
  \label{ex:figure eight}
  The diagram below is a cut presentation of the figure eight knot.
  \begin{center}
    \begin{tikzpicture}[line width=1, scale=1, xscale=1.5]
      \draw[<-] (0,1) node[above] {\(\beta_{6}\)} \br (1,0) node[below] {\(\beta_{5}\)};
      \draw[white, line width=10] (0,0) node {} \br (1,1) node {};
      \draw[->] (0,0) node[below] {\(\beta_{0}\)} \br (1,1) node[above] {\(\beta_{1}\)};

      \draw[->] (1,1) node {} \br (2,0) node[below] {\(\beta_{2}\)};
      \draw[white, line width=10] (1,0) node {} \br (2,1) node {};
      \draw[<-] (1,0) node {} \br (2,1) node[above] {\(\beta_{4}\)};

      \draw[<-] (2,1) node {} \br (3,2) node {};
      \draw[white, line width=10] (2,2) node {} \br (3,1) node {};
      \draw[->] (2,2) node[above] {\(\beta_{6}\)} \br (3,1) node[below] {\(\beta_{7}\)};

      \draw[-] (3,1) node {} \br (4,2) node {};
      \draw[white, line width=10] (3,2) node {} \br (4,1) node {};
      \draw[<-] (3,2) node[above] {\(\beta_{3}\)} \br (4,1) node[below] {\(\beta_{2}\)};

      \draw (2,0) to (4,0) ;
      \draw[looseness = 1.5] (4,0) to [out =0, in=0] (4,1); 

      \draw (0,2) to (2,2) ;
      \draw[looseness = 1.5] (0,1) to [out =180, in=180] (0,2); 
      
      \draw[->] (-1,0) to (0,0);

      \draw[->] (4,2) node[above] {\(\beta_{0}\)} to (5,2);
    \end{tikzpicture}
  \end{center}
  Assigning variables to the segments as shown the action functional is
  \begin{align*}
    \mathcal{S}_{D}(\mu, \beta_{0}|\beta)
    &=
    \brkern{+}{}{
      \beta_{0}, \beta_{5}, \beta_{1}, \beta_{6}
    }
    \brkern{+}{}{
      \beta_{4} + 1 - \nr, \beta_{1}, \beta_{5} + 1 - \nr, \beta_{2}
    }
    \\
    &\phantom{=}\times
    \brkern{-}{}{
      \beta_{3} + 1 - \nr, \beta_{6}, \beta_{4} + 1 - \nr, \beta_{7}
    }
    \brkern{-}{}{
      \beta_{7}, \beta_{2}, \beta_{0}, \beta_{3}
    }
  \end{align*}
  where \(\beta = (\beta_{1}, \dots, \beta_{7})\) and each braiding kernel has \(\mu_{1} = \mu_{2} = \mu\).
\end{example}

Now let \(\rho\) be a decorated representation of \(L\).
Choose a nonzero vector \(u \in \shadset = \mathbb{C}^{2} \setminus \set{0}\) for one of the regions of \(D\) and color the others according to \cref{def:shadow coloring}; we now have a shadow coloring of \(D\).
If the region and segment parameters of \cref{def:chi parameters} are all nonzero, then our coloring is admissible.
If not, we can always modify \(\rho\) and/or \(u\) by a gauge transformation to an admissible representation.
This does not affect the value of the invariant.

\begin{lemma}
  \label{thm:state sum}
  Fix logarithms \(\beta_{j}\) of the segment parameters assigned by \((\rho, u)\), including a common \(\beta_{0}\) for the two boundary segments of \(D\).
  A \defemph{state} is a vector \(k \in \set{0, \dots, \nr - 1}^{E}\).
  The quantum link invariant is given by a sum over all states:
  \[
    \qinv{L, K; \rho, u}
    =
    \sum_{k \in \set{0, \dots, \nr - 1}^{E}} \mathcal{S}_{D}(\mu, \beta_{0} | \beta + k)
    \qedhere
  \]
\end{lemma}
\begin{proof}
  The matrix coefficients of the linear maps assigned to crossings, cups, and caps are given by the relevant components of \(\mathcal{S}_{D}\), so taking the sum above corresponds to composing these linear maps.
  The result is an endomorphism of the module assigned to the open segment (since we did not sum over its variable) given by
  \[
    F(\ket{\beta_{0}})
    =
    \left[ \sum_{k \in \set{0, \dots, \nr - 1}^{E}} \mathcal{S}_{D}(\mu, \beta_{0} | \beta + k) \right] \ket{\beta_{0}}
  \]
  Because the module is absolutely simple \(F\) is a scalar map and thus
  \[
    \qinv{L, K; \rho, u}
    =
    \sclbrak{F}
    =
    \sum_{k \in \set{0, \dots, \nr - 1}^{E}} \mathcal{S}_{D}(\mu, \beta_{0} | \beta + k)
    .
    \qedhere
  \]
\end{proof}

\subsection{Example: the figure-eight knot}%
\label{sec:Example: the figure-eight knot}

\begin{figure}[htpb]
  \begin{center}
    \begin{tikzpicture}[line width=1, scale=1, xscale=1.5]
      \draw[<-] (0,1) node[above] {} \br (1,0) node[below] {};
      \draw[white, line width=10] (0,0) node {} \br (1,1) node {};
      \draw[->] (0,0) node[below] {\(1\)} \br (1,1) node[above] {};

      \draw[->] (1,1) node {} \br (2,0) node[below] {};
      \draw[white, line width=10] (1,0) node {} \br (2,1) node {};
      \draw[<-] (1,0) node {} \br (2,1) node[above] {\(3\)};

      \draw[<-] (2,1) node {} \br (3,2) node {};
      \draw[white, line width=10] (2,2) node {} \br (3,1) node {};
      \draw[->] (2,2) node[above] {\(4\)} \br (3,1) node[below] {};

      \draw[-] (3,1) node {} \br (4,2) node {};
      \draw[white, line width=10] (3,2) node {} \br (4,1) node {};
      \draw[<-] (3,2) node[above] {} \br (4,1) node[below] {};

      \draw (2,0) to (4,0) ;
      \draw[looseness = 1.5] (4,0) to [out =0, in=0] (4,1); 
      \node[below] at (4,0) {\(2\)};

      \draw (0,2) to (2,2) ;
      \draw[looseness = 1.5] (0,1) to [out =180, in=180] (0,2); 
      
      \draw[->] (-1,0) to (0,0);

      \draw[->] (4,2) node[above] {\(1'\)} to (5,2);
    \end{tikzpicture}
  \end{center}
  \caption{A cut presentation of the figure eight knot with labeled arcs.}
  \label{fig:figure-eight-arcs}
\end{figure}
For small values of \(\nr\) we numerically compute the value of \(\qinv{4_1, \rho, \mu}\) for the figure eight knot \(4_1\) and all boundary-parabolic representations \(\rho : \pi_{1}(\extr{4_{1}}) \to \slg\).
Our examples suggest that \(\qinv{}\) depends nontrivially only on \(\rho\) as a \(\pslg = \operatorname{Isom}(\mathbb{H}^{3})\) representation: the dependence on the choice of lift to \(\slg\) is trivial and there is only one class of log-meridians that give a nonzero invariant.
\note{%
  For boundary non-parabolic representations we expect more complicated dependence on the log-meridian.
}
These show exponential growth consistent with the standard volume conjectures for quantum invariants.

To begin we find all boundary-parabolic representations of \(\extr{4_{1}}\).
This computation is well-known, but we repeat it here for completeness.

\begin{proposition}
  Up to gauge equivalence there are two boundary-parabolic representations \(\tilde{\rho}_{1}, \tilde{\rho}_{2} : \pi_{1}(\extr{4_{1}}) \to \pslg\), with \(\tilde{\rho}_{1}\) the holonomy of the complete hyperbolic structure.
  As for any boundary-parabolic representations they each have a unique decoration.
  Each \(\rho_{i}\) lifts to two \(\slg\) representations \(\rho_{i}, \rho_{i}'\) with eigenvalues
  \[
    \begin{array}{lllll}
      & \rho_{1} & \rho_{2} & \rho_{1}' & \rho_{2}' \\
      \hline
      m & 1 & 1  & -1 & -1 \\
      \ell & -1 & -1 & -1 & -1
    \end{array}
    \qedhere
  \]
\end{proposition}

\begin{proof}
  The tangle diagram \(D\) in \cref{fig:figure-eight-arcs} is a cut presentation of the figure-eight knot.
  To determine a representation \(\rho : \pi(D) \to \slg\) we need to find matrices \(g_{i} = \rho(x_{i}) \in \slg\) for each arc  (with \(g_{1} = g_{1'}\) so the closure is defined) satisfying the crossing relations
  \begin{align*}
    g_{4} &= g_{3} \qn g_{1}
          &
    g_{2} &= g_{1}  \qn g_{3}
    \\
    g_{2} &= g_{3}  \qn g_{4}
          &
    g_{4} &= g_{1}  \qn g_{2}
  \end{align*}
  where \(x \qn y \defeq y^{-1} x y\).
  Eliminating \(g_{2}\) and \(g_{4}\) reduces these to
  \begin{equation}
    \label{eq:figure eight quandle relations}
    \begin{aligned}
      g_{1} \qn g_{3} &= g_{3} \qn (g_{3} \qn g_{1})
      \\
      g_{3} \qn g_{1} &= g_{1} \qn (g_{1} \qn g_{3})
    \end{aligned}
  \end{equation}
  
  For a boundary-parabolic representation all the \(g_i\) have eigenvalues \(\pm 1\).
  Consider the case where the eigenvalues are \(1\).
  Since we only need to find \(\rho\) up to conjugacy, we may assume
  \[
    g_{1} = 
    \begin{bmatrix}
      1 & 0 \\
      1 & 1
    \end{bmatrix}
    \text{ and }
    g_{3} =
    \begin{bmatrix}
      1 & -\tau \\
      0 & 1
    \end{bmatrix}
  \]
  for some \(\tau \in \mathbb{C}\).
  It is now elementary to show that our matrices satisfy the relations \eqref{eq:figure eight quandle relations} if and only if \(\tau^2 + \tau + 1 = 0\).
  The representation \(\rho_{1}\) with \(\tau = e^{\tu /3}\) is a lift of the complete hyperbolic structure on \(S^{3} \setminus 4_{1}\).
  Write \(\rho_{2}\) for the representation with \(\tau = e^{4 \pi \ii/3}\).
  Up to gauge transformation these are the only boundary-parabolic representations of \(\extr{4_{1}}\) with meridian eigenvalue \(1\).
  However, we could equally well choose them to have meridian eigenvalue \(-1\): we write  \(\rho_{1}'\) and \(\rho_{2}'\) for the representations with
  \[
    g_{1} = 
    \begin{bmatrix}
      -1 & 0 \\
      -1 & -1
    \end{bmatrix}
    \text{ and }
    g_{3} =
    \begin{bmatrix}
      -1 & \tau \\
      0 & -1
    \end{bmatrix}
  \]
  As for any boundary-parabolic representation the \(\rho_{i}\) and \(\rho_{i}'\) all have a unique decoration, and one can check directly that they all have longitude eigenvalue  \(-1\).
  One way to do this is to compute that in the labeling of \cref{fig:figure-eight-arcs} the blackboard-framed longitude commuting with the meridian  \(x_{1}\) is
  \begin{equation}
    \label{eq:figure-eight longitude}
    x_{3} x_{4}^{-1} x_{1} x_{2}^{-1} .
  \end{equation}
  It is not hard to check that
  \[
    \rho_{i}'( x_{3} x_{4}^{-1} x_{1} x_{2}^{-1} ) = 
    \rho_{i}( x_{3} x_{4}^{-1} x_{1} x_{2}^{-1} ) = 
    \begin{bmatrix}
      -1 & 0 \\
      2 (\tau - 1) \tau  & -1
    \end{bmatrix}
    \qedhere
  \]
\end{proof}

\(\mu = k/2\) is a log-meridian of \(\rho_{i}\) or \(\rho_{i}'\) depending on the parity of \(k\), so we abuse notation and write
\[
  (\rho_{i}, k/2) =
  \begin{cases}
    (\rho_{i}, k/2) & k \text{ even} \\
    (\rho_{i}', k/2) & k \text{ odd} \\
  \end{cases}
\]
W can use the rules of \cref{sec:State-sum description} to compute \(\qinv{}\) by an explicit tensor contraction; we give some details at the end of this section.
Some values are given in \cref{table:figure eight values}.

\begin{remark}
  Experimentally, for \(\nr = 2, \dots, 15\) we see that \(|\qinv{4_1, \rho_{j}, k/2}| < 10^{-8}\) unless \(k \equiv -1 \mod \nr\) and similarly that
  \[
    |
    \qinv{4_1, \rho_{i}, -1/2}
    +
    \qinv{4_1, \rho_{i}, (\nr - 1)/2}
    |
    <
    10^{-8}
  \]
  Given this observation we only include \(\qinv{4_{1}, \rho_{i}, (\nr - 1)/2}\) in \cref{table:figure eight values}, since the other choices of log-meridian are trivial or differ by a sign.

  We also observe that \(|\Im \qinv{4_1, \rho_{j}, k/2}| < 10^{-8}\) in the tested range, so we only report the real parts.
  As discussed in \cref{rem:vol conj evidence}  the same thing happens for \(\csfunc{}\) where it is a consequence of the amphichirality of \(4_1\).
\end{remark}

\begin{conjecture}
  \label{conj:parabolic invariants}
  Let \(K\) be a knot, \(\tilde{\rho} : \pi_{1}(\extr{K}) \to \pslg\) a boundary-parabolic representation, and \(\mu \in \frac{1}{2} \mathbb{Z}\) a log-meridian specifying a lift \(\rho\) of \(\tilde{\rho}\) to \(\slg\) as above.
  Then
  \begin{thmenum}
  \item \label{conj:parabolic invariants:parity}
    there is an invertible (\cref{def:product of models}) link invariant \(\eta\) that takes values in \(\set{1, -1}\) so that
    \[
      \qinv{K, \rho, \mu + \nr /2}
      =
      \eta(K)
      \qinv{K, \rho, \mu}
    \]
  \item \label{conj:parabolic invariants:vanishing}
    and 
    \(
    \qinv{K, \rho, \mu}
    =
    0
    \text{ unless }
    2 \mu \equiv -1 \pmod{\nr}.
    \)
  \end{thmenum}
  Similar things should hold for links with boundary-parabolic representations.
\end{conjecture}

\ref{conj:parabolic invariants:parity} should be related to the \(1\)-dimensional \(\qgrp\)-module \(\Pi\) defined by
\[
  K \cdot 1 = -1, E \cdot 1 = F \cdot 1 = 0.
\]
Since the Casimir has \(\Omega \cdot 1 = -\omega^{1/2} - \omega^{-1/2} = \omega^{(\nr-1)/2} + \omega^{-(\nr -1)/2}\) we can think of \(\Pi\) as having log-meridian \(\mu = \nr /2\), and in the notation of the previous paragraph \(V_{\mu} \otimes \Pi \iso V_{\mu + \nr /2}\) as \(\qgrp\)-modules.
I expect that \(\eta\) is the (ordinary, not holonomy) Reshetikhin-Turaev invariant associated to \(\Pi\).

\ref{conj:parabolic invariants:vanishing} appears to be related to the vanishing of modified dimensions.
Let \(V_{\mu} = \qfunc{g, u, \mu}\) for a parabolic decorated matrix \(g\).
The Hopf link computations in \cref{sec:Example: Hopf links} show that the modified dimension%
\note{
  Specifically, the modified dimension relative to the usual projective cover of the tensor unit.
}
of \(V_{\mu}\) is only well-defined if \(2\mu \equiv -1 \pmod{\nr}\).
Specifically, as \(\mu \to k/2\) for \(k \not \equiv -1/2\), 
\[
  \operatorname{d}(\mu)
  =
  \frac{
    \nr
  }{
    \left\{ \mu + \tfrac{1}{2} \right\}_{\omega}
  }
  =
  \nr
  \frac{
    \omega^{\mu + 1/2} - \omega^{-\mu - 1/2}
    }{
    \omega^{\nr(\mu + 1/2)} 
    -
    \omega^{-\nr(\mu + 1/2)} 
  }
  \to
  \frac{1}{0}
\]
On the other hand, computations for \(4_{1}\) suggest \(\qinv{K, \rho', \mu} \to 0\) as \((\rho', \mu) \to (\rho, k/2)\).
It may be possible to interpret the limit
\[
  \lim_{\mu \to k/2} \operatorname{d}(\mu) \qinv{K, \rho', \mu}
\]
as a nontrivial knot invariant.

\begin{table}
  \centering
  \[
    \begin{array}{l c c c c}
      \nr & \qinv{4_1, \rho_{1} } & \frac{1}{\nr} \log \qinv{4_1, \rho_{1} } & \qinv{4_1, \rho_{2} } & \frac{1}{\nr} \log \qinv{4_1, \rho_{2} }  \\
      \hline
      2 & 4.16824 & 0.713747 & 5.75783 & 0.875280 \\
      3 & 11.6551 & 0.818580 & 13.8995 & 0.877285 \\
      4 & 25.2731 & 0.807435 & 27.8557 & 0.831759 \\
      5 & 48.5023 & 0.776322 & 51.1679 & 0.787023 \\
      6 & 86.9050 & 0.744136 & 89.4770 & 0.748997 \\
      7 & 149.148 & 0.714991 & 151.518 & 0.717244 \\
      8 & 248.535 & 0.689448 & 250.647 & 0.690506 \\
      9 & 405.289 & 0.667178 & 407.126 & 0.667680 \\ 
      10 & 649.947 & 0.647689 & 651.512 & 0.647930 \\
      11 & 1028.35 & 0.630520 & 1029.67 & 0.630635 \\
      12 & 1609.03 & 0.615282 & 1610.12 & 0.615339 \\
      13 & 2493.97 & 0.601664 & 2494.86 & 0.601691 \\
      14 & 3834.41 & 0.589412 & 3835.13 & 0.589426 \\
      15 & 5853.94 & 0.578325 & 5854.52 & 0.578331 \\
    \end{array}
  \]
  \caption{
    Approximate values of the quantum invariant for (lifts of) the boundary-parabolic \(\pslg\) representations \(\rho_{1}, \rho_{2}\) of the \(4_{1}\) knot exterior at the log-meridian \(\mu = (\nr - 1)/2\).
    Here \(\rho_{1}\) is the complete hyperbolic structure.
    Compare \(\log \csfunc{4_{1}, \rho_{1}, 0} = \operatorname{Vol}(4_{1})/2\pi + 0 \ii \approx  0.323066\).
  }
  \label{table:figure eight values}
\end{table}

\begin{remark}
  \label{rem:vol conj evidence}
  Choosing the lift of \(\rho_{1}\) with trace \(2\) the classical Chern-Simons invariant of \(\extr{4_1}\) is
  \[
    \log \csfunc{4_{1}, \rho_{1}, 0}
    =
    \operatorname{Vol}(4_{1})/2\pi + 0 \ii \approx  0.323066.
  \]
  The imaginary part vanishes because \(4_{1}\) is amphichiral.
  The invariant can be computed by adapting the procedure in \cref{sec:State-sum description} to \(\csfunc{}\).
  The geometric data is the same, but there are no state sums because \(\nr = 1\).
  The values in \cref{table:figure eight values} support 
  \[
    \frac{1}{\nr}
    \log 
    \qinv{4_1, \rho_{i}, (\nr - 1)/2}
    \to
    \log \csfunc{4_{1}, \rho_{1}, 0}
    \text{ as }
    \nr \to \infty
  \]
  for \emph{both} boundary-parabolic representations as in \cref{conj:our volume}.
\end{remark}

\begin{proof}[Details of the computation]
  To determine segment and region parameters we need to choose an admissible shadow coloring.
  However, because the left eigenspace of \(g_{1}\) is spanned by \((1,0)\) there is no such choice: the segment parameters \eqref{eq:segment parameter} will always be \(0\).
  We can fix this by taking a type (A) gauge transformation, i.e.\ by simultaneously conjugating \(g_{1}\) and \(g_{3}\) by some matrix \(h\).
  Any \(h \in \slg\)  that does not fix \((1,0)\) will work, and the choice of \(h\) does not affect the value of the invariant.
  Similarly there are many possible choices of shadow color \(u\) for the topmost region.
  Once it is chosen it determines shadow colors for the remaining regions, and a generic choice of \(u\) will give an admissible shadow coloring \((\rho, u)\) of \(D\).
  
  Choosing logarithms \(\beta_{i}\) of the segment parameters \cref{thm:state sum} says that
  \begin{equation*}
    \qinv{4_{1}, \rho_{i}, \mu}
      =
      \sum_{n_{1}, \dots, n_{7} = 0}^{\nr -1}
      \begin{aligned}
        &
        \brkern{+}{}{
          \beta_{0} , \beta_{5} + n_{5}, \beta_{1} + n_{1}, \beta_{6} + n_{6}
        }
        \\
        &\times
        \brkern{+}{}{
          \beta_{4} + n_{4} + 1 - \nr, \beta_{1} + n_{1}, \beta_{5} + n_{5} + 1 - \nr, \beta_{2} + n_{2}
        }
        \\
        &\times
        \brkern{-}{}{
          \beta_{3} + n_{3} + 1 - \nr, \beta_{6} + n_{6}, \beta_{4} + n_{4} + 1 - \nr, \beta_{7} + n_{7}
        }
        \\
        &\times
        \brkern{-}{}{
          \beta_{7} + n_{7}, \beta_{2} + n_{2}, \beta_{0} , \beta_{3} + n_{3}
        }
      \end{aligned}
  \end{equation*}
  using the action functional of \cref{ex:figure eight}.
\end{proof}

\subsection{Example: Hopf links}%
\label{sec:Example: Hopf links}
\begin{figure}
  \begin{align*}
    &\begin{tikzpicture}[line width = 1, scale = 1, xscale = 1.5]
      \coordinate (1) at (0,1);
      \coordinate (2) at (0,0);
      \coordinate (3) at (1,0);
      \coordinate (4) at (1,1);
      \coordinate (3) at (1,0);
      \coordinate (5) at (2,1);
      \coordinate (6) at (2,0);
      \draw[-] (2)  to [out=90,in=180] (4);
      \draw[white, line width=10] (1) to [out=00,in=180] (3);
      \draw[-] (1)  to [out=00,in=180] (3);
      \draw[-] (3)  to [out=00,in=180] (5);
      \draw[white, line width=10] (4) to [out=00,in=90] (6);
      \draw[-] (4)  to [out=00,in=90] (6);
      \draw[->, looseness = 1.5] (6) to [out=270,in=270] (2);
      \draw[->] (5) to ++(0.5,0);
      \draw[<-] (1) to ++(-0.5,0);
      \node[above] at (1) {\(g_{1}, [v]\)};
      \node[above] at (4) {\(g_{2}, [v]\)};
      \node[below] at (1, -1) {\(\rho\)};
    \end{tikzpicture}
    &
    &\begin{tikzpicture}[line width = 1, scale = 1, xscale = 1.5]
      \coordinate (1) at (0,1);
      \coordinate (2) at (0,0);
      \coordinate (3) at (1,0);
      \coordinate (4) at (1,1);
      \coordinate (3) at (1,0);
      \coordinate (5) at (2,1);
      \coordinate (6) at (2,0);
      \draw[-] (2)  to [out=90,in=180] (4);
      \draw[white, line width=10] (1) to [out=00,in=180] (3);
      \draw[-] (1)  to [out=00,in=180] (3);
      \draw[-] (3)  to [out=00,in=180] (5);
      \draw[white, line width=10] (4) to [out=00,in=90] (6);
      \draw[-] (4)  to [out=00,in=90] (6);
      \draw[->, looseness = 1.5] (6) to [out=270,in=270] (2);
      \draw[->] (5) to ++(0.5,0);
      \draw[<-] (1) to ++(-0.5,0);
      \node[above] at (1) {\(g_{1}, [v]\)};
      \node[above] at (4) {\(g_{2}, [w]\)};
      \node[below] at (1, -1) {\(\tilde{\rho}\)};
    \end{tikzpicture}
  \end{align*}
  \caption{
    Two commuting matrices \(g_{1}, g_{2} \in \slg\) give a representation of the Hopf link exterior.
    Fixing the decoration of the open component there are two choices of decoration of the closed component.
    (Here \(v, w\) are eigenvectors of \(g_{1}\) and \(g_{2}\).)
    We call the corresponding decorated representations \(\rho\) and \(\tilde{\rho}\).
  }
  \label{fig:Hopf links with decorations}
\end{figure}

Next we compute \(\qinv{}\) in closed form for Hopf links.
Unlike for boundary-parabolic representations we see that \(\qinv{}\) depends on the choice of decoration, and studying this dependence leads to \cref{conj:opposite decoration}.

Let \(L_{H}\) be the (positively linked, zero-framed) Hopf link and \(K_{1}\) one of its components.
The tangle diagram below is a presentation of \((L_{H}, K_{1})\).
We can parametrize representations \(\rho\) of (the exterior of) \(L_{H}\) by choosing matrices \(g_{1}, g_{2} \in \slg\) corresponding to meridians as
\begin{center}
  \begin{tikzpicture}[line width = 1, scale = 1, xscale = 1.5]
    \coordinate (1) at (0,1);
    \coordinate (2) at (0,0);
    \coordinate (3) at (1,0);
    \coordinate (4) at (1,1);
    \coordinate (3) at (1,0);
    \coordinate (5) at (2,1);
    \coordinate (6) at (2,0);
    \draw[-] (2)  to [out=90,in=180] (4);
    \draw[white, line width=10] (1) to [out=00,in=180] (3);
    \draw[-] (1)  to [out=00,in=180] (3);
    \draw[-] (3)  to [out=00,in=180] (5);
    \draw[white, line width=10] (4) to [out=00,in=90] (6);
    \draw[-] (4)  to [out=00,in=90] (6);
    \draw[->, looseness = 1.5] (6) to [out=270,in=270] (2);
    \draw[->] (5) to ++(0.5,0);
    \draw[<-] (1) to ++(-0.5,0);
    \node[above] at (1) {\(g_{1}\)};
    \node[left] at (2) {\(g_{1} g_{2} g_{1}^{-1}\)};
    \node[above] at (4) {\(g_{2}\)};
    \node[above] at (5) {\(g_{2}^{-1} g_{1} g_{2}\)};
  \end{tikzpicture}
\end{center}
It is clear that this is a well-defined coloring of the closure if and only if \(g_{1}\) and \(g_{2}\) commute.
If they do, then there are eigenvectors \(v, w\) and eigenvalues \(m_{1}, m_{2}\) with
\begin{align*}
  v g_{1} &= m_{1}^{-1} v & w g_{1} &= m_{1} w \\
  \\
  v g_{2} &= m_{2}^{-1} v & w g_{2} &= m_{1} w
\end{align*}
(Again we note the unfortunate inverse of \cref{def:decorated matrix}.)
Fixing the decoration of the open component to be \([v]\) there are two choices for the decoration of the closed component.
We call these \(\rho\) and \(\tilde{\rho}\) as in \cref{fig:Hopf links with decorations}.
They have the same underlying \(\slg\) representation but different decorations.%
\note{
  If \(\rho\) is boundary-parabolic then \(v\) and \(w\) are proportional and \(\rho = \tilde{\rho}\) as decorated representations.
}
Note that for  \(\rho\) both crossings are pinched and for \(\tilde{\rho}\) neither are.
A log-decoration for \(\rho\) is given by log-meridians \(\mu_{1}, \mu_{2}\) with \(e^{\tu \mu_{i}} = m_{i}\).
We choose a log-decoration for \(\tilde{\rho}\) with the same \(\mu_{1}\) for the open component and a log-meridian \(\tilde{\mu}_{2}\) with \(e^{\tu \tilde{\mu}_{2}} = m_{2}^{-1}\).

\begin{proposition}
  \label{thm:Hopf link values}
  \begin{align}
    \qinv{L_{H}, K_{1}; \rho, \mu_{1}, \mu_{2}}
    &=
    \omega^{
      2 \mu_{1} \mu_{2}
      - (\nr - 1) ( \mu_{1} + \mu_{2} + \tfrac{1}{2} )
    }
    \left\{ \mu_{1} + \tfrac{1}{2} \right\}_{\omega}
    \\
    \qinv{L_{H}, K_{1}; \tilde{\rho}, \mu_{1}, \tilde{\mu}_{2}}
    &=
    \omega^{
      - 2 \mu_{1} \tilde{\mu}_{2}
      + (\nr - 1) ( \mu_{1} + \tilde{\mu}_{2} + \tfrac{1}{2} )
    }
    \left\{ \mu_{1} + \tfrac{1}{2} \right\}_{\omega}
  \end{align}
  where
  \[
    \left\{ \lambda  \right\}_{\omega}
    \defeq
    \frac{
      \omega^{\nr \lambda} - \omega^{-\nr\lambda}
      }{
      \omega^{\lambda} - \omega^{-\lambda}
    }
    .
    \qedhere
  \]
\end{proposition}

As observed in \cref{sec:Properties of qfun} this shows that \(\qfunc{}\) is invariant under an appropriately-defined change of decoration.
We conjecture this holds in general.

\begin{proof}[Computation for \(\rho\)]
  Because both crossings are pinched we can choose a shadow-coloring so that they are \(E\)-nilpotent as in \cref{thm:diagonal coloring lemma}.
  We can therefore work in the highest-weight basis \(\set{\wb{\mu_{i} - k} \given k = 0, \dots \nr -1}\).
  In this basis the braiding coefficients are given by \cref{eq:colored-Jones-R-matrix}:
  \begin{equation*}
    c_{n_{1} n_{2}}^{n_{3} n_{4}} 
    =
    \delta_{n_{1} + n_{2}}^{n_{3} + n_{4}} \omega^{n_{1}'(n_{2} -2\mu_{2}) + \mu_{1} \mu_{2}}
    \frac{
      \qp{\omega^{-2\mu_{2}}}{n_{2}}
      \qp{\omega}{n_{1}}
    }{
      \qp{\omega^{-2\mu_{2}}}{n_{4}}
      \qp{\omega}{n_{4} - n_{2}}
      \qp{\omega}{n_{3}}
    }
  \end{equation*}
  Assign indices to the bases of each vector space as
  \begin{center}
    \begin{tikzpicture}[line width = 1, scale = 1, xscale = 1.5]
      \coordinate (1) at (0,1);
      \coordinate (2) at (0,0);
      \coordinate (3) at (1,0);
      \coordinate (4) at (1,1);
      \coordinate (3) at (1,0);
      \coordinate (5) at (2,1);
      \coordinate (6) at (2,0);
      \draw[-] (2)  to [out=90,in=180] (4);
      \draw[white, line width=10] (1) to [out=00,in=180] (3);
      \draw[-] (1)  to [out=00,in=180] (3);
      \draw[-] (3)  to [out=00,in=180] (5);
      \draw[white, line width=10] (4) to [out=00,in=90] (6);
      \draw[-] (4)  to [out=00,in=90] (6);
      \draw[->, looseness = 1.5] (6) to [out=270,in=270] (2);
      \draw[->] (5) to ++(0.5,0);
      \draw[<-] (1) to ++(-0.5,0);
      \node[above] at (1) {\(n_{1}\)};
      \node[left] at (2) {\(n_{2}\)};
      \node[below] at (3) {\(n_{3}\)};
      \node[above] at (4) {\(n_{4}\)};
      \node[above] at (5) {\(n_{5}\)};
      \node[right] at (6) {\(n_{6}\)};
    \end{tikzpicture}
  \end{center}
  Because we are working in a highest-weight basis it suffices to find the coefficient \(n_{1} = n_{5} = 0\) and the pivotal structure imposes \(n_{6} = n_{2}\).
  Since 
  \(
      \qp{\omega}{l}^{-1}
  \)
  is \(0\) if \(l < 0\) the braiding matrix of the second crossing forces \(n_{3} = 0\).
  Now the \(\delta\) function in the matrix of the first crossing gives \(n_{2} = n_{4}\), so writing \(c(1)\) and \(c(2)\) for the matrix coefficients of the two crossings we compute
  \begin{align*}
    &\qinv{L_{H}, K_{1}; \rho, \mu_{1}, \mu_{2}}
    \\
    &=
    \sum_{n = 0}^{\nr -1}
    c(1)_{0, n}^{0, n}
    c(2)_{n, 0}^{n, 0}
    \omega^{-(\nr - 1)(\mu_{2} - n )}
    \\
    &=
    \sum_{n = 0}^{\nr -1}
    \omega^{-(\nr - 1)(\mu_{2} - n )}
    \frac{
      \qp{\omega^{- 2\mu_{2}}}{n}
      \qp{\omega}{0}
      }{
      \qp{\omega^{-2 \mu_{2}}}{n}
      \qp{\omega}{0}
      \qp{\omega}{n}
    }
    \omega^{n(0 - 2 \mu_{1})}
    \frac{
      \qp{\omega^{-2\mu_{1}}}{n}
      \qp{\omega}{n}
      }{
      \qp{\omega^{-2\mu_{1}}}{0}
      \qp{\omega}{0}
      \qp{\omega}{n}
    }
    \\
    &=
    \sum_{n = 0}^{\nr -1}
    \omega^{-(\nr - 1)\mu_{2}}
    \omega^{- n(1 + 2 \mu_{1}) }
  \end{align*}
  Here the factor of \(K^{-(\nr -1)(\mu_{2} - n)}\) comes from the downward-oriented evaluation map.
  The conclusion follows from elementary algebra and
  \[
    \frac{
      1-  \omega^{\nr \gamma}
    }{
      1-  \omega^{\gamma}
    }
    =
    \sum_{n = 0}^{\nr -1} \omega^{n \gamma}.\qedhere
  \]
\end{proof}

\begin{proof}[Computation for \(\tilde{\rho}\)]
  Choose an admissible shadow \(u\) for the topmost region and compute the segment parameters of the diagram.
  It is not hard to see that they are given by 
  \begin{center}
    \begin{tikzpicture}[line width = 1, scale = 1, xscale = 1.5]
      \coordinate (1) at (0,1);
      \coordinate (2) at (0,0);
      \coordinate (3) at (1,0);
      \coordinate (4) at (1,1);
      \coordinate (3) at (1,0);
      \coordinate (5) at (2,1);
      \coordinate (6) at (2,0);
      \draw[-] (2)  to [out=90,in=180] (4);
      \draw[white, line width=10] (1) to [out=00,in=180] (3);
      \draw[-] (1)  to [out=00,in=180] (3);
      \draw[-] (3)  to [out=00,in=180] (5);
      \draw[white, line width=10] (4) to [out=00,in=90] (6);
      \draw[-] (4)  to [out=00,in=90] (6);
      \draw[->, looseness = 1.5] (6) to [out=270,in=270] (2);
      \draw[->] (5) to ++(0.5,0);
      \draw[<-] (1) to ++(-0.5,0);
      \node[above] at (1) {\(b\)};
      \node[left] at (2) {\(b z /m_{1}\)};
      \node[below] at (3) {\(b m_{2} \)};
      \node[above] at (4) {\(bz\)};
      \node[above] at (5) {\(b\)};
    \end{tikzpicture}
  \end{center}
  for
  \[
    b =
    -
    \frac{
      v \evec{2}   
      }{
      v u
    }
    \text{ and }
    z
    =
    \frac{
      vu
      }{
      v \evec{2}
    }
    \frac{
      w \evec{2}
      }{
      w u
    }
    .
  \]
  Here \(z\) is the shape parameter of the north tetrahedron of the left-hand crossing.
  Choosing \(\beta, \zeta\) with
  \(
    e^{\tu \beta} = b,
    e^{\tu \zeta} =  z
  \)
  we can assign the state-sum log-parameters as
  \begin{center}
    \begin{tikzpicture}[line width = 1, scale = 2, xscale = 1.5]
      \coordinate (1) at (0,1);
      \coordinate (2) at (0,0);
      \coordinate (3) at (1,0);
      \coordinate (4) at (1,1);
      \coordinate (3) at (1,0);
      \coordinate (5) at (2,1);
      \coordinate (6) at (2,0);
      \draw[-] (2)  to [out=90,in=180] (4);
      \draw[white, line width=10] (1) to [out=00,in=180] (3);
      \draw[-] (1)  to [out=00,in=180] (3);
      \draw[-] (3)  to [out=00,in=180] (5);
      \draw[white, line width=10] (4) to [out=00,in=90] (6);
      \draw[-] (4)  to [out=00,in=90] (6);
      \draw[->, looseness = 1.5] (6) to [out=270,in=270] (2);
      \draw[->] (5) to ++(0.5,0);
      \draw[<-] (1) to ++(-0.5,0);
      \node[above] at (1) {\(\beta\)};
      \node[left] at (2) {\(\beta + \zeta - \mu_{1} + n_{2}\)};
      \node[below] at (3) {\(\beta - \tilde{\mu}_{2} + n_{3}\)};
      \node[above] at (4) {\(\beta + \zeta + n_{4}\)};
      \node[above] at (5) {\(\beta\)};
      \node[right] at (6) {\(\beta + \zeta - \mu_{1} + n_{2} + 1 - \nr\)};
    \end{tikzpicture}
  \end{center}
  Writing
  \begin{align*}
    \zeta_{\lW} &= \zeta - 2\mu_{2}
    \\
    \zeta_{\lS} &= \zeta + 2 \tilde{\mu}_{s} - 2 \mu_{1}
    \\
    \zeta_{\lE} &= \zeta + 2 \tilde{\mu}_{2}
  \end{align*}
  the associated state sum is
  \begin{equation}
    \label{eq:hopf non pinched state sum}
    \begin{aligned}
      \frac{
        \omega^{-6 \mu_{1} \tilde{\mu}_{2}}
        }{
        \nr^{2} 
      }
      \sum_{n_{2}, n_{2}, n_{4} = 0}^{\nr -1}
      &
      \frac{
        \pf{\zeta + n_{4}}
        \pf{-\zeta - n_{4}}
        }{
        \pft{\zeta_{\lW} + n_{2}}
        \pf{-\zeta_{\lW} - n_{2} + \nr -1}
      }
      \\
      &\times
      \frac{
        \pf{\zeta_{\lS} + n_{3} - n_{2}}
        \pf{-\zeta_{\lS} - n_{3} + n_{2} + \nr - 1}
        }{
        \pf{\zeta_{\lE} + n_{4} - n_{3}}
        \pft{-\zeta_{\lE} - n_{4} + n_{3}}
      }
      \\
      &\times
      \omega^{
        2 \mu_{1} n_{3}
        +
        \tilde{\mu}_{2}(2 n_{2} - 2 n_{4} + 1 - \nr)
      }
    \end{aligned}
  \end{equation}
  Using the definition \eqref{eq:pft def} of \(\pft{}\) and
  \[
    \frac{
      \pf{\zeta}
      }{
      \pf{\zeta + \nr - 1}
    }
    =
    \frac{
      1 - \omega^{\nr \zeta}
      }{
      1 - \omega^{\zeta}
    }
  \]
  we can re-write the summand as a product of terms of the form \(\pf{\theta}\pft{-\theta}\).
  Applying the inversion relation \eqref{eq:pf inversion} shows that \eqref{eq:hopf non pinched state sum} is equal to
  \begin{equation}
    \label{eq:hopf non pinched state sum 2}
    \begin{aligned}
      \frac{
        \omega^{-6 \mu_{1} \tilde{\mu}_{2}}
        }{
        \nr^{2} 
      }
      \sum_{n_{2}, n_{2}, n_{4} = 0}^{\nr-1}
      &
      \frac{
        \Qfunc{\zeta + n_{4}}
        \Qfunc{\zeta_{\lS} + n_{3} - n_{2}}
        }{
        \Qfunc{\zeta_{\lE} + n_{4} - n_{3}}
        \Qfunc{-\zeta_{\lW} - n_{2}}
      }
      \\
      &\times
      \frac{
        1 - \omega^{-\nr \zeta  }
        }{
        1 - \omega^{- \zeta - n_{4} }
      }
      \frac{
        1 - \omega^{-\nr (\zeta - 2 \mu_{1}) }
        }{
        1 - \omega^{- \zeta + 2 \mu_{1} - n_{2} }
      }
      \\
      &\times
      \omega^{
        2 \mu_{1} n_{3}
        +
        \tilde{\mu}_{2}(2 n_{2} - 2 n_{4} + 1 - \nr)
        +
        (\nr -1)(\zeta + \zeta_{\lS} + n_{4} + n_{3} - n_{2})
      }
    \end{aligned}
  \end{equation}
  where \(\Qfunc{\zeta} = \omega^{- \zeta(\zeta + \nr -1)/2}\).
  Expanding the \(\Qfunc{}\) terms shows that the summand in \eqref{eq:hopf non pinched state sum 2} is equal to \(\omega^{n_{3}(n_{4} - n_{2} -1)}\) times terms not depending on \(n_{3}\), so summing over \(n_{3}\) first gives \(\nr\) times a Kronecker delta enforcing \(n_{4} = n_{2} + 1\).
  Imposing this condition and simplifying the \(\Qfunc{}\) terms gives
  \begin{equation}
    \label{eq:hopf non pinched state sum 3}
    \begin{aligned}
      &
      \frac{
        \omega^{-2 \mu_{1} \tilde{\mu}_{2}}
        }{
        \nr
      }
      \sum_{n_{2} = 0}^{\nr -1}
      \frac{
        1 - \omega^{-\nr \zeta  }
        }{
        1 - \omega^{- \zeta - n_{2} - 1 }
      }
      \frac{
        1 - \omega^{-\nr (\zeta - 2 \mu_{1}) }
        }{
        1 - \omega^{- \zeta + 2 \mu_{1} - n_{2} }
      }
      \omega^{
        (\nr -1)(\tilde{\mu}_{2} + \zeta +  n_{2} + 1)
      }
      \\
      &=
      \frac{
        \omega^{(\nr - 1 -2 \mu_{1}) \tilde{\mu}_{2} }
        }{
        \nr
      }
      \sum_{n_{2} = 0}^{\nr - 1}
      \frac{
        1 - \omega^{\nr \zeta  }
        }{
        1 - \omega^{ \zeta + n_{2} + 1 }
      }
      \frac{
        1 - \omega^{-\nr (\zeta - 2 \mu_{1}) }
        }{
        1 - \omega^{- \zeta + 2 \mu_{1} - n_{2} }
      }
    \end{aligned}
  \end{equation}
  Expanding in geometric series then taking the sum over \(n_{2}\) shows that this is equal to
  \begin{align*}
     &\frac{
        \omega^{(\nr - 1 -2 \mu_{1}) \tilde{\mu}_{2} }
        }{
        \nr
      }
      \sum_{n_{2}, s, t =0}^{\nr - 1}
      \omega^{s (\zeta + n_{2} + 1) - t(\zeta Z - 2 \mu_{1} + n_{2})}
      \\
      &=
        \omega^{(\nr - 1 -2 \mu_{1}) \tilde{\mu}_{2} }
      \sum_{ s = 0 }^{\nr - 1}
      \omega^{s (2 \mu_{1} + 1)}
  \end{align*}
  which gives the claim after further elementary algebra.
\end{proof}

\section{The pinched limit}
\label{sec:The pinched limit}
As discussed in \cref{def:shadow coloring} certain admissible colorings of a crossing are \defemph{pinched}, which means the ideal octahedron assigned to them is geometrically degenerate.
In \cref{sec:construction of invariant} we defined the braiding away from these crossings and claimed that it is well-defined at the singular limit.
Here we prove this claim and compute its matrix coefficients, translating some results of \cite{McPhailSnyderAlgebra} to our conventions.
We use these results to show \(\qfunc{}\) recovers the link invariants of \textcite{Kashaev1995} and \textcite{Akutsu1992}.

\subsection{The braiding at a pinched crossing}
\label{sec:pinched limit brading}

By definition, a shadow-coloring
\begin{center}
  \begin{tikzpicture}[line width = 1, scale = 1, xscale = 1.5]
    \coordinate (1) at (0,1);
    \coordinate (2) at (0,0);
    \coordinate (3) at (1,0);
    \coordinate (4) at (1,1);
    \draw[->] (1) \br (3) ;
    \draw[->] (2) \br (4) ;
    \node[left] at (1) {\(g_{1}, [v_{1}], m_{1}\)};
    \node[left] at (2) {\(g_{2}, [v_{2}], m_{2}\)};
    \node[above] at (0.5,1) {\(u\)};
  \end{tikzpicture}
\end{center}
of a crossing (of either sign) is \defemph{pinched} if \([v_{1}] = [v_{2}]\) as subspaces of \(\mathbb{C}^{2}\) (i.e.\ as points of \(\pjsp\)).
At any pinched crossing the segment parameters are given by
\begin{center}
  \begin{tikzpicture}[line width = 1, scale = 1, xscale = 1.5]
    \coordinate (1) at (0,1);
    \coordinate (2) at (0,0);
    \coordinate (3) at (1,0);
    \coordinate (4) at (1,1);
    \draw[->] (1) \br (3) ;
    \draw[->] (2) \br (4) ;
    \node[left] at (1) {\(b\)};
    \node[left] at (2) {\(m_{1} b\)};
    \node[right] at (3) {\(m_{2} b\)};
    \node[right] at (4) {\( b\)};
  \end{tikzpicture}
\end{center}
for some \(b \in \mathbb{C}^{\times}\).
As a consequence the arguments of the quantum dilogarithms in the braiding kernels (\ref{eq:braiding kernel positive},\ref{eq:braiding kernel negative}) lie in \(\mathbb{Z}\) where the poles and zeros of \(\pf{}\) are.

Given log-meridians \(\mu_{1}\) and \(\mu_{2}\) we are lead to consider the braiding matrix coefficients
\[
  \hat{c}_{n_{1} n_{2}}^{n_{3} n_{4}}
  \defeq
  \brkern{+}{\mu_{1}, \mu_{2}}{\beta + n_{1} , \beta + \mu_{1} + n_{2}, \beta + \mu_{2} + n_{3}, \beta + n_{4}}
\]
corresponding to the diagram state
\begin{center}
  \begin{tikzpicture}[line width = 1, scale = 1, xscale = 1.5]
    \coordinate (1) at (0,1);
    \coordinate (2) at (0,0);
    \coordinate (3) at (1,0);
    \coordinate (4) at (1,1);
    \draw[->] (1) \br (3) ;
    \draw[->] (2) \br (4) ;
    \node[left] at (1) {\(\beta + n_{1}\)};
    \node[left] at (2) {\(\beta + \mu_{1} + n_{2}\)};
    \node[right] at (3) {\(\beta + \mu_{2} + n_{3} \)};
    \node[right] at (4) {\( \beta + n_{4}\)};
  \end{tikzpicture}
\end{center}
Here \(\beta\) is some logarithm of \(b\).
The \(\hat{c}_{n_{1} n_{2}}^{n_{3} n_{4}}\) are independent of it.
We need to compute them by taking the limit of the meromorphic function \(\brkern{+}{}{}\).
For an integer \(n\), write \(\modb{n}\) for the unique integer with
\[
  0 \le \modb{n} < \nr \text{ and } n \equiv \modb{n} \mod{\nr \mathbb{Z}}
\]
and set
\[
  \gamma_{\nr}(n) \defeq \frac{n - \modb{n}}{\nr}
  .
\]

\begin{theorem}[\protect{\cite[Theorem 6.5]{McPhailSnyderAlgebra}}]
  \label{thm:pinched braiding is defined}
  With respect to the bases above the matrix coefficients of the braiding at a positive, pinched crossing are
  \begin{equation}
    \label{eq:pinched matrix coefficients}
    \begin{aligned}
      \hat{c}_{n_1 n_2}^{n_{3} n_{4} }
      &=
      \frac{1}{\nr}
      \theta_{n_1 n_2}^{n_{3} n_{4}}
      A_{n_1 n_2}^{n_{3} n_{4}}
      \\
      &\phantom{=}\times
      \omega^{
        \mu_{1}(n_{3} - n_{1})
        -
        \mu_{2}(n_{4} - n_{2})
        +
        n_{2} - n_{1}
      }
      \\
      &\phantom{=}\times
      \frac{
        \pf{\modb{n_{4} - n_1}}
        \pf{\modb{n_2 - n_{3}}}
      }{
        \pf{\modb{n_{4} - n_{3}}}
        \pf{\modb{n_2 - n_1 - 1}}
      }.
    \end{aligned}
  \end{equation}
  where
  \begin{equation}
    \label{eq:cutoff theta}
    \theta_{n_1 n_2}^{n_{3} n_{4}}
    =
    \begin{cases}
      1 & 
      \modb{ n_{4} - n_{3}}
      +
      \modb{ n_{2} - n_{3}}
      -
      \modb{ n_{4} - n_{3}}
      -
      \modb{ n_{2} - n_{1} - 1}
      =
      1
      \\
      0 & \text{otherwise}
    \end{cases}
  \end{equation}
  and
  \begin{equation}
    \label{eq:pinched matrix A}
    A_{n_1 n_2}^{n_{3} n_{4}}
    =
    \frac{
      a_{\lW}
      }{
      m_{1} a_{\lN}
    }
    \frac{
      (a_{\lW}/m_{1} a_{\lN})^{\gamma_{\nr}(n_{2} - n_{1} - 1)}
      (m_{2} a_{\lE} / a_{\lN})^{\gamma_{\nr}(n_{4} - n_{3})}
      }{
      (m_{2} a_{\lS}/m_{1} a_{\lN})^{\gamma_{\nr}(n_{2} - n_{3})}
    }
  \end{equation}
  in terms of the region parameters of the pinched crossing.
\end{theorem}

\begin{proof}
  Quasi-periodicity of \(\pf{}\) gives
  \[
    \pf{\zeta + n}
    =
    \pf{\zeta + \modb{n}}
    (1 - \omega^{\nr \zeta})^{-\gamma_{\nr}(n)}
  \]
  so we can write the braiding kernel of a generic crossing as
  \begin{equation}
    \label{eq:pinched limit kernel}
    \begin{aligned}
      \brkern{+}{\mu}{\beta}
      &=
      \frac{
        1
        }{
        \nr
      }
      \frac{
        \pf{\zeta_{\lN}^{0} + \modb{n_{4} - n_{1}}}
        \pf{\zeta_{\lS}^{0} + \modb{n_{2} - n_{3}}}
      }{
        \pf{\zeta_{\lE}^{0} + \modb{n_{4} - n_{3}}}
        \pf{\zeta_{\lW}^{0} + \modb{n_{2} - n_{1} - 1}}
      }
      \\
      &\phantom{=}\times
      \omega^{
        \mu_{1}(\beta_{3} - \beta_{1}) 
        +
        \mu_{2}(\beta_{2} - \beta_{4})
        -
        \mu_{1} \mu_{2}
        -
        (\nr -1) (\zeta_{W}^{0} + n_{2} - n_{1})
      }
      \\
      &\phantom{=}\times
      \frac{
        (z_{\lE}^{1})^{\gamma_{\nr}(n_{4} - n_{3})}
        (z_{\lW}^{1})^{\gamma_{\nr}(n_{2} - n_{1} + \nr - 1)}
        }{
        (z_{\lN}^{1})^{\gamma_{\nr}(n_{4} - n_{1})}
        (z_{\lS}^{1})^{\gamma_{\nr}(n_{2} - n_{3})}
      }
    \end{aligned}
  \end{equation}
  where \(\zeta_{\lN}^{0} = \beta_{4} - \beta_{1}\) and so on are the flattening parameters (\ref{eq:flattening-N}--\ref{eq:flattening-E}) and the \(z_{j}^{1}\) are the shape parameters (\ref{eq:shape-N}--\ref{eq:shape-E}).
  Since \(\gamma_{\nr}(k + \nr) = 1 + \gamma_{\nr}(k)\) the last line of the right-hand side of \eqref{eq:pinched matrix coefficients} can be expanded as
  \begin{equation}
    \label{eq:pinched matrix A terms}
    \begin{aligned}
      &K^{
        1 +
        \gamma_{\nr}(n_{4} - n_{3})
        +
        \gamma_{\nr}(n_{2} - n_{1} - 1)
        -
        \gamma_{\nr}(n_{4} - n_{1})
        -
        \gamma_{\nr}(n_{2} - n_{3})
      }
      \\
      &\times
      \frac{
        a_{\lW}
        }{
        m_{1}
      }
      \frac{
        (a_{\lW}/m_{1})^{\gamma_{\nr}(n_{2} - n_{1} - 1)}
        (a_{\lE} m_{2})^{\gamma_{\nr}(n_{2'} - n_{1'})}
        }{
        (a_{\lS}m_{2}/m_{1})^{\gamma_{\nr}(n_{2} - n_{1'})}
        (a_{\lN})^{\gamma_{\nr}(n_{2'} - n_{1})}
      }
    \end{aligned}
  \end{equation}
  where \(K = a_{\lN} / (1 - b_{4}/b_{1})\).

  Now take the pinched limit, which means that each \(\zeta_{j}^{0} \to 0\) and the \(a_{j}\) are fixed by the coloring of the diagram.
  The only potentially singular term is the power of \(K\) which has a zero of order
  \[
    1
    +
    \gamma_{\nr}(n_{2'} - n_{1'})
    +
    \gamma_{\nr}(n_{2} - n_{1} - 1)
    -
    \gamma_{\nr}(n_{2'} - n_{1})
    -
    \gamma_{\nr}(n_{2} - n_{1'})
  \]
  As in the proof of \cite[Theorem 5.6(a)]{McPhailSnyderAlgebra} one can argue that this quantity is never negative, so there are no poles.
  When it is positive the matrix coefficient is zero, which gives the condition in \eqref{eq:cutoff theta}.
  The other terms in \cref{eq:pinched matrix A terms} give \eqref{eq:pinched matrix A} after some rearranging.
  It is clear that the first two lines of the right hand side of \eqref{eq:pinched limit kernel} give the remaining terms in \eqref{eq:pinched matrix coefficients}.
\end{proof}

\begin{corollary}
  \label{thm:qfunc is holomorphic}
  The braiding maps are holomorphic sections of a complex vector bundle of \(\qgrp\)-intertwiners over the space of admissible colorings  of a crossing.
\end{corollary}

\begin{proof}
  Their matrix coefficients \(\brkern{\pm}{}{}\) are meromorphic functions of the local coordinates \(\beta_{i}, \mu_{j}\).
  The theorem above says that the poles of the quantum dilogarithms defining \(\brkern{\pm}{}{}\) cancel at every admissible coloring, so in fact they are holomorphic.
\end{proof}

In practice it is usually inconvenient to work directly with the matrix coefficients \eqref{eq:pinched matrix coefficients} because of the complicated cutoff term \eqref{eq:cutoff theta}.
If we choose an appropriate shadow coloring we can instead work in a highest weight basis.

\begin{definition}
  \label{def:E-nilpotent}
  We say that \(\qpf{a, b, \mu}\) is \defemph{\(E\)-nilpotent} if \(a = \omega^{\nr \mu}\) so that \(E\) acts nilpotently.
  We say a shadow-colored crossing is \defemph{\(E\)-nilpotent} if the modules assigned by \(\qfunc{}\) are all \(E\)-nilpotent.
\end{definition}

It is elementary to see that
\begin{center}
  \begin{tikzpicture}[line width = 1, scale = 1, xscale = 1.5]
    \coordinate (1) at (0,1);
    \coordinate (2) at (0,0);
    \coordinate (3) at (1,0);
    \coordinate (4) at (1,1);
    \draw[->] (1) \br (3) ;
    \draw[->] (2) \br (4) ;
    \node[left] at (1) {\(g_{1}, [v_{1}], m_{1}\)};
    \node[left] at (2) {\(g_{2}, [v_{2}], m_{2}\)};
    \node[above] at (0.5,1) {\(u\)};
  \end{tikzpicture}
\end{center}
is \(E\)-nilpotent if and only if \(g_{1} u = m_{1} u\) and \(g_{2} u = m_{2} u\) (i.e. \(u\) is a simultaneous left eigenvector with the appropriate eigenvalues).
In this case it is also an eigenvector of \(g_{1} \qn g_{2}\) and \(g_{1} \qn^{-1} g_{2}\) so we have \(u_{\lW} = m_{1} u_{\lN}\), \(u_{\lE} = m_{2} u_{\lN}\), and \(u_{\lS} = m_{1} m_{2} u_{\lN}\).
Thus at any \(E\)-nilpotent crossing the region parameters are
\begin{align}
  \label{eq:E-nilpotent a}
  a_{\lW} &= a_{\lN} m_{1}
  , &
  a_{\lE} &= a_{\lN} m_{2}
  , &
  a_{\lS} &= a_{\lN} m_{1} m_{2}.
\end{align}
If \(g_{1}\) and \(g_{2}\) commute they have a common right eigenspace and there is an \(E\)-nilpotent shadow coloring.
Such a crossing is necessarily pinched.

\begin{theorem}[\protect{\cite[Theorem 5.9]{McPhailSnyderAlgebra}}]
  \label{thm:braiding in highest weight}
  Let \(X\) be an \(E\)-nilpotent, pinched crossing with braiding
  \[
    \qfunc{X} :
    \qpf{m_{1}, b_{1}, \mu_{1}}
    \otimes
    \qpf{m_{2}, b_{2}, \mu_{1}}
    \to
    \qpf{m_{2}, b_{4}, \mu_{2}}
    \otimes
    \qpf{m_{1}, b_{3}, \mu_{1}}
  \]
  Each module has a highest-weight basis 
  \(
    \wb{\mu_{i}}, \dots, \wb{\mu_{i} - (\nr - 1)}
  \)
  in terms of the vectors \eqref{eq:weight basis}.
  The matrix coefficients
  \[
    \qfunc{X}( \wb{\mu_{1} - n_{1}} \otimes \wb{\mu_{1} - n_{4}} )
    =
    \sum_{n_{3} n_{4} = 0}^{\nr}
    c_{n_{1} n_{2}}^{n_{3}, n_{4}} 
    \wb{\mu_{2} - n_{4}} \otimes \wb{\mu_{1} - n_{3}}
  \]
  of the braiding with respect to this basis are
  \begin{equation}
    \label{eq:colored-Jones-R-matrix}
    c_{n_{1} n_{2}}^{n_{3} n_{4}} 
    =
    \delta_{n_{1} + n_{2}}^{n_{3} + n_{4}} \omega^{n_{1}'(n_{2} -2\mu_{2}) + \mu_{1} \mu_{2}}
    \frac{
      \qp{\omega^{-2\mu_{2}}}{n_{2}}
      \qp{\omega}{n_{1}}
    }{
      \qp{\omega^{-2\mu_{2}}}{n_{4}}
      \qp{\omega}{n_{4} - n_{2}}
      \qp{\omega}{n_{3}}
    }
  \end{equation}
\end{theorem}

The proof of this theorem is an involved computation using terminating \(q\)-series.
It can be understood as taking discrete Fourier transform, since the bases \(\ket{\beta}\) and  \(\wb{\mu}\) are related by 
\[
  \wb{\mu}
  =
  \sum_{k=0}^{\nr - 1} \omega^{(\beta + k) \mu} \ket{\beta + k}
\]
At \(\mu_{1} = \mu_{2} = \nr-1/2\) the matrix \eqref{eq:pinched matrix coefficients} recovers Kashaev's \(R\)-matrix \cite{Kashaev1995}, while \eqref{eq:colored-Jones-R-matrix} is the braiding matrix defining the \(\nr\)th colored Jones polynomial at \(q = \omega\).
This equivalence was first shown by \textcite{Murakami2001}.

\subsection{Recovering the Kashaev--Akutsu-Deguchi-Ohtsuki invariants}%
\label{sec:Relation to the Kashaev--Akutsu-Deguchi-Ohtsuki invariant}

Below we explain how to recover Kashaev's quantum dilogarithm invariant \cite{Kashaev1995} and the ADO invariants \cite{Akutsu1992} from \(\qfunc{L, K; \rho, \mu}\) in the limit where \(\rho\) is reducible (hence geometrically trivial).
For simplicity we consider only links, although there are similar results for tangles.

\begin{definition}
  Let \(L\) be a (framed, oriented) link.
  A representation \(\rho : \pi_{1}(\extr{L}) \to \slg\) makes \(\mathbb{C}^{2}\) a \(\pi_{1}(\extr{L})\)-module.
  We say \(\rho\) is \defemph{reducible} if \(\mathbb{C}^{2}\) is reducible as a \(\pi_{1}(\extr{L})\)-module and \defemph{diagonalizable} if it decomposes into two \(1\)-dimensional submodules.
  A decoration of a reducible representation is \defemph{standard} if it assigns every meridian the same eigenspace, which we can identify with the fixed \(\pi(\extr{L})\)-submodule.
\end{definition}

\begin{lemma}
  \label{thm:diagonal coloring lemma}
  Let \(\rho\) be a diagonalizable representation of \(L\) with a standard decoration.
  Then for any diagram \(D\) of \(L\) there is an admissible shadow coloring \((\rho, u)\) for which every crossing is pinched and \(E\)-nilpotent.
\end{lemma}

\begin{proof}
  Because \(\rho\) is diagonalizable it assigns the same matrix \(g_i\) to each arc in component \(i\) and there is a basis \(v,w\) of \(\mathbb{C}^{2}\) with \(g_i v = m_{i}^{-1} v\) and \(g_i w = m_{i} w\), where \(m_i\) are the distinguished eigenvalues.
  The standard decoration assigns every arc of component \(i\) the eigenspace \([v]\) so by definition it is pinched.
  Let \(u\) be the column vector dual to \(w\) with respect to the determinant pairing on \(\mathbb{C}^{2}\).
  Specifically, if
  \(w =
  \begin{bmatrix}
    c_{0} & c_{1}
  \end{bmatrix}
  \)
  we set
  \[
    u
    =
    \begin{bmatrix}
      -c_{1} \\ c_{0}
    \end{bmatrix}
  \]
  It is not hard to see that \(g_i u = m_i u\) for each \(i\).
  If we assign \(u\) to the topmost region of \(D\) then every region is colored by a nonzero multiple of \(u\) so we get an \(E\)-nilpotent shadow coloring.
  It is admissible because \(v, w\) is a basis, so
  \[
    v u 
    = 
    \det \begin{bmatrix}
      v \\ w
    \end{bmatrix}
    \ne 0
    .
    \qedhere
  \]
\end{proof}

Earlier we chose not to use modified dimensions to normalize \(\qinv{}\) to avoid some technical complications.
Here we do so to match the normalization of \cite{Costantino2012}.
We say a log-meridian \(\mu\) is \defemph{renormalizable} if
\[
  \mu \in \left( \mathbb{C} \setminus \frac{1}{2}\mathbb{Z} \right) \cup \left(\frac{\nr -1}{2} + \frac{\nr}{2} \mathbb{Z}\right)
\]
equivalently if the \defemph{modified dimension}
\begin{equation}
  \label{eq:moddim}
  \operatorname{d}(\mu)
  \defeq 
  \frac{
    \nr
  }{
    \omega^{(\nr-1)(\mu + \frac{1}{2})}
    +
    \omega^{(\nr-3)(\mu + \frac{1}{2})}
    +
    \cdots
    +
    \omega^{(1-\nr)(\mu + \frac{1}{2})}
  }
\end{equation}
is well-defined.
There are various ways to define modified dimensions, but they can be computed as ratios of open Hopf links.
In particular we have
\[
  \operatorname{d}(\mu)
  =
  \frac{
    \qinv{L_{H}, K_{1}; \rho, (\nr - 1)/2, \mu}
  }{
    \qinv{L_{H}, K_{1}; \rho, \mu, (\nr - 1)/2}
  }
  =
  \frac{
    \nr
    }{
    \left\{ \mu + \tfrac{1}{2} \right\}_{\omega}
  }
\]
in terms of the Hopf link invariants computed in \cref{thm:Hopf link values}.

\begin{theorem}
  \label{thm:KADO}
  Let \(L\) be a link, \(\rho\) a reducible representation with a standard decoration, and \(\mu\) a renormalizable choice of log-meridians.
  Let \(K_{i}\) be a component of \(L\) with log-meridian \(\mu_{i}\), and define
  \[
    \kado{L, \mu}
    \defeq
    \operatorname{d}(\mu_{i})
    \qinv{L, K_{i}; \rho, \mu}
  \]
  It is independent of the choice of component \(K_{i}\) and agrees with the invariant \(\mathsf{N}_{r}\) of trivalent \(\mathbb{C}\)-colored graphs defined by \textcite{Costantino2012}, identifying \(\mu\) with a cohomology class as in \cref{rem:cohomology class}.
  In their notation the integer parameter is \(r = \nr\) and our highest weights \(\mu\) correspond to their highest weights via
  \(
    \alpha = 2 \mu - (\nr -1).
  \)
\end{theorem}

\begin{proof}
  First assume \(\rho\) is diagonalizable.
  Choose a diagram \((D, \rho, u)\) of \((L,K)\) as in \cref{thm:diagonal coloring lemma} so that the braiding matrices at every crossing are given by \cref{eq:colored-Jones-R-matrix}.
  We claim that these agree with the action of the truncated \(R\)-matrix of \cite[equation (31)]{Costantino2012}.
  They must agree up to a scalar by \cref{thm:uniqueness of R} and one can check directly that this scalar is \(1\) by examining the action on the tensor product of two highest-weight vectors.
  Furthermore our pivotal structure agrees with the one given in \cite[equation (33)]{Costantino2012}, as does our modified dimension function.
  We conclude that \(\kado{}\) agrees with \(\mathsf{N}_{r}\), so in particular it is independent of the choice of component \(K_{i}\).

  Now consider the case where \(\rho\) is reducible but not necessarily diagonalizable.
  We may gauge-transform \(\rho\) so that, for each meridian \(\mer_{j}\),
  \[
    \rho(\mer_{j}) = 
    \begin{bmatrix}
      e^{\tu \mu_{j}} & 0 \\
      c_{j} & e^{- \tu \mu_{j}}
    \end{bmatrix}
  \]
  \(\qinv{L , K_{i}; \rho, \mu}\) is a holomorphic function of the parameters \(\mu_j\) and \(c_j\), which are coordinates on the reducible part of the log-decorated representation variety.
  Write \(\rho_{0}\) for the diagonal representation obtained by setting each \(c_k\) to \(0\).
  Our claim follows from showing
  \(
    \qinv{L, K_{i}; \rho, \mu}
    =
    \qinv{L, K_{i}; \rho_{0}, \mu}.
  \)
  Write \(\rho_{\kappa}\) for the image of \(\rho\) under a type \(A\) gauge transformation by a diagonal matrix with entries \(\kappa, \kappa^{-1}\).
  Observe that
  \[
    \begin{bmatrix}
      \kappa & 0 \\
      0 & \kappa^{-1}
    \end{bmatrix}^{-1}
    \begin{bmatrix}
      e^{\tu \mu_{j}} & 0 \\
      c_{j} & e^{-\tu \mu_{j}}
    \end{bmatrix}
    \begin{bmatrix}
      \kappa & 0 \\
      0 & \kappa^{-1}
    \end{bmatrix}
    =
    \begin{bmatrix}
      e^{\tu \mu_{j}} & 0 \\
      \kappa^{2} c_{j} & e^{-\tu \mu_{j}}
    \end{bmatrix}
  \]
  so \(\lim_{\kappa \to 0} \rho_{\kappa} = \rho_{0}\).
  Then by continuity and gauge invariance
  \[
    \qinv{L, K_{i}; \rho, \mu}
    =
    \qinv{L, K_{i}; \rho_{\kappa}, \mu}
    =
    \lim_{\kappa \to 0}
    \qinv{L, K_{i}; \rho_{\kappa}, \mu}
    =
    \qinv{L, K_{i}; \rho_{0}, \mu}.
    \qedhere
  \]
\end{proof}

\section{Mirror images and the Reidemeister torsion}
\label{sec:torsion}

The second main result of \cite{McPhailSnyder2020} relates the geometric quantum link invariant \(\operatorname{F}_{\nr}\) of \textcite{Blanchet2018} to the nonabelian Reidemeister torsion.
Because \(\qinv{}\) is a refinement of \(\operatorname{F}_{\nr}\) we get a relationship with the torsion.

\begin{definition}
  The \defemph{mirror image} \((\overline{L}, \overline{\rho}, \overline{\mu})\) of \((L, \rho, \mu)\) has
  \begin{itemize}
    \item \(\overline{L}\) the usual mirror image of \(L\) with  the opposite orientations and framings and
    \item \(\overline{\rho}\) the induced representation of \(\extr{\overline{L}}\).
      In terms of a diagram we invert all the matrix colors \(g_{i}\) to match the orientation reversal.
      We use the same eigenspaces as \(\rho\) to define the decoration of \(\overline{\rho}\), and define
    \item \(\overline{\mu}\) to be the log-decoration with \(\overline{\mu}_{i} = - \mu_{i} - 1\).
      \qedhere
  \end{itemize}
\end{definition}

\begin{bigtheorem}
  \label{thm:torsion relation}
  Let \(L\) be a link, \(K\) a component of \(L\), \(\rho \) a decorated representation of \(\extr{L}\), and \(\mu\) a choice of log-meridians for \(\rho\).
  Assume that no meridian \(x\) of \(L\) has \(\tr \rho(x) = 2\) (equivalently, that no \(\mu_{i}\) is an integer) which ensures that the Reidemeister torsion \(\tau(S^{3} \setminus L, \rho)\) of the complement twisted by \(\rho\) is well-defined as an element of \(\mathbb{C}/(\pm 1)\).
  Then
  \[
    \frac{1}{m + m^{-2} -2}
    \qinv[2]{L, K; \rho, \mu}
    \qinv[2]{\overline{L}, \overline{K}; \overline{\rho}, \overline{\mu}}
    =
    \tau(S^{3} \setminus L, \rho)
  \]
  where \(m = e^{\tu \mu_{K}}\) is the meridian eigenvalue of the component \(K\).
\end{bigtheorem}

Note that unlike in \cite[Theorem 2]{McPhailSnyder2020} there is no extra normalization factor other than the modified dimension \((m + m^{-1} - 2)^{-1}\).

To prove this theorem we need to explain how to take the mirror image of \(\qfunc{}\).
Recall that \(\qcat\) is the category of weight modules for \(\qgrp\).
Let \(\qcatb\) be the category of weight modules for the algebra \(\qgrp^{\operatorname{cop}}\), which is \(\qgrp\) with the opposite comultiplication.

\begin{definition}
  Define a model \(\qfuncb{}\) valued in \(\qcatb\) by
  \[
    \qfuncb{
      \begin{tikzpicture}[baseline=(g.base)]
        \draw (0,0.5) node {\(u\)};
        \draw[->] (0,0) node(g)[left] {\((g, [v], \mu)\)} to (0.1,0);
        \draw (0,-0.5) node {\(gu\)};
      \end{tikzpicture}
    }
    \defeq
    \qfunc{
      \begin{tikzpicture}[baseline=(g.base)]
        \draw (0,0.5) node {\(gu\)};
        \draw[<-] (0,0) node(g)[left] {\((g, [v], \mu)\)} to (0.1,0);
        \draw (0,-0.5) node {\(u\)};
      \end{tikzpicture}
    }
  \]
  and
  \[
    \qfuncb{
      \begin{tikzpicture}[line width=1, scale=1, xscale = 1.5, baseline={(current bounding box.center)}]
        \coordinate (1) at (0,1);
        \coordinate (2) at (0,0);
        \coordinate (3) at (1,0);
        \coordinate (4) at (1,1);
        \tikzBraidingEast{+}{(1)}{(2)}{(3)}{(4)};
        \node[left] at (1) {\(g_{1}\)};
        \node[left] at (2) {\(g_{2}\)};
        \node[right] at (3) {\(g_{1}\)};
        \node[right] at (4) {\(g_{1}^{-1} g_{2} g_{1}\)};
        \node at (0.5,1) {\(u\)};
      \end{tikzpicture}
    }
    \defeq
    \tau
    \qfunc{
      \begin{tikzpicture}[line width=1, scale=1, xscale = 1.5, baseline={(current bounding box.center)}]
        \coordinate (1) at (1,0);
        \coordinate (2) at (1,1);
        \coordinate (3) at (0,1);
        \coordinate (4) at (0,0);
        \draw[->] (1) to[out = 180, in = 00] (3);
        \draw[white, line width=10] (2) \br (4);
        \draw[->] (2) to[out = 180, in = 00] (4);
        \node[right] at (1) {\(g_{1}^{-1} g_{2} g_{1}\)};
        \node[right] at (2) {\(g_{1}\)};
        \node[left] at (3) {\(g_{2}\)};
        \node[left] at (4) {\(g_{1}\)};
        \node at (0.5,0) {\(u\)};
      \end{tikzpicture}
    }
    \tau
  \]
  where \(\tau(x \otimes y) = y \otimes x\) is the swap map.
\end{definition}

The value of \(\qfuncb{}\) on a crossing is not a morphism of \(\qgrp\)-modules, but it \emph{is} a morphism of \(\qgrp^{\operatorname{cop}}\)-modules, which is why it takes values in \(\qcatb\).
To prove that \(\qfuncb{}\) is a model take the proofs for \(\qfunc{}\) and reflect them.
The definition of \(\qfuncb{}\) makes it obvious that it computes the invariant of the mirror image link:

\begin{proposition}
  \label{thm:mirror computes mirror}
  Let \(\qinvb{}\) be the link invariant associated to \(\qfuncb{}\).
  Then
  \[
    \qinvb{L, K; \rho, \mu} = \qinv{\overline{L}, \overline{K}; \overline{\rho}, \overline{\mu}}.\qedhere
  \]
\end{proposition}

Now consider the model \(\mathcal{D}_{\nr} = \qfunc{} \boxtimes \qfuncb{} \), which we can think of as a \(\slg\)-graded quantum double of \(\qfunc{}\) as in \cite[Section 6]{McPhailSnyder2020}.
By \cref{thm:uniqueness of R} the braiding of \(\qfunc{}\) is determined up to a scalar by the outer \(R\)-matrix \(\mathcal{R}\) and a similar result holds for \(\qfuncb{}\) \cite[Section 4.1]{McPhailSnyderThesis} and \(\mathcal{R}^{-1}\).
These conditions similarly characterize the braiding of \(\mathcal{D}_{\nr}\) up to a scalar.

There is a natural way to choose this scalar.
For any coloring \(\chi\) of a point
\[
  \mathcal{D}_{\nr}(\chi) 
  = \qfunc{\chi} \boxtimes \qfuncb{\chi}
  = \qpf{\chi} \boxtimes \qpf{\chi}^{*}
\]
and there is a \defemph{canonical vector}
\[
  v(\chi) = \sum_{n = 0}^{\nr -1} \ket{\beta + n} \boxtimes \bra{\beta + n} \in \mathcal{D}_{\nr}(g, u)
\]
in each \(\mathcal{D}_{\nr}(\chi)\).
It is independent of the choice of \(\beta\) for the same reason the coevaluation of \(\qcat\) is.

\begin{lemma}
  \label{thm:double normalization}
  The braiding of \(\mathcal{D}_{\nr}\) preserves the canonical vectors: for any crossing \(X\),
  \[
    \mathcal{D}_{\nr}(X) \colon
    \mathcal{D}_{\nr}(\chi_{1})
    \otimes
    \mathcal{D}_{\nr}(\chi_{2})
    \to
    \mathcal{D}_{\nr}(\chi_{4})
    \otimes
    \mathcal{D}_{\nr}(\chi_{3})
  \]
  has
  \[
    \mathcal{D}_{\nr}(
      v(\chi_{1}) \otimes v(\chi_{2})
    )
    =
    v(\chi_{4}) \otimes v(\chi_{3})
    .
    \qedhere
  \]
\end{lemma}

\begin{proof}
  Call the colored diagram below \(D\):
  \begin{center}
    \begin{tikzpicture}[line width = 1, scale = 1, xscale = 1.5, baseline={(current bounding box.center)}]
      \coordinate (1) at (0,1);
      \coordinate (2) at (0,0);
      \coordinate (2d) at (0,-1);
      \coordinate (1d) at (0,-2);
      \draw[->] (2d) .. controls (-0.5,-1) and (-0.5, 0) .. (2);
      \draw[->] (1d) .. controls (-1,-2) and (-1, 1) .. (1);
      \node[right] at (1) {\(g_{1}\)};
      \node[right] at (2) {\(g_{2}\)};
      \node[right] at (1d) {\(g_{1}\)};
      \node[right] at (2d) {\(g_{2}\)};
      \node at (-1,1) {\(u\)};
    \end{tikzpicture}
  \end{center}
  Its image under \(\qfunc{}\) is a morphism
  \[
    \qfunc{D}
    :
    \mathbb{C} \to
    \qpf{\chi_{1}}
    \otimes
    \qpf{\chi_{2}}
    \otimes
    \qpf{\chi_{2}}^{*}
    \otimes
    \qpf{\chi_{1}}^{*}
  \]
  which as usual we identify with the vector \(\qfunc{D}(1)\).
  There is an obvious vector space isomorphism
  \[
    \begin{aligned}
      \Phi
      \colon
      &\qpf{\chi_{1}}
      \otimes
      \qpf{\chi_{2}}
      \otimes
      \qpf{\chi_{2}}^{*}
      \otimes
      \qpf{\chi_{1}}^{*}
      \\
      &\to
      (
        \qpf{\chi_{1}}
        \boxtimes
        \qpf{\chi_{1}}^{*}
      )
      \otimes
      (
        \qpf{\chi_{2}}
        \boxtimes
        \qpf{\chi_{2}}^{*}
      )
      =
      \mathcal{D}_{\nr}(\chi_{1}) 
      \otimes
      \mathcal{D}_{\nr}(\chi_{2}) 
    \end{aligned}
  \]
  and \(\Phi(\qfunc{D}) = v(\chi_{1}) \otimes v(\chi_{2})\) is the canonical vector.
  Now consider the diagram \(D'\)
  \begin{center}
  \begin{tikzpicture}[line width = 1, scale = 1, xscale = 1.5, baseline={(current bounding box.center)}]
      \coordinate (1) at (0,1);
      \coordinate (2) at (0,0);
      \coordinate (3) at (1,0);
      \coordinate (4) at (1,1);
      \tikzBraidingEast{+}{(1)}{(2)}{(3)}{(4)};
      \coordinate (3d) at (1,-2);
      \coordinate (4d) at (1,-1);
      \coordinate (2d) at (0,-1);
      \coordinate (1d) at (0,-2);
      \draw[->] (3d) to[out = 180, in = 00] (2d);
      \draw[white, line width=10] (4d) to[out = 180, in = 00] (1d);
      \draw[->] (4d) to[out = 180, in = 00] (1d);
      \node[right] at (3) {\(g_{3}\)};
      \node[right] at (4) {\(g_{4}\)};
      \node[right] at (3d) {\(g_{4}\)};
      \node[right] at (4d) {\(g_{3}\)};
      \draw[->] (2d) .. controls (-0.5,-1) and (-0.5, 0) .. (2);
      \draw[->] (1d) .. controls (-1,-2) and (-1, 1) .. (1);
      \node[above] at (1) {\(g_{1}\)};
      \node[above] at (2) {\(g_{2}\)};
      \node[above] at (1d) {\(g_{1}\)};
      \node[above] at (2d) {\(g_{2}\)};
      \node at (-1,1) {\(u\)};
  \end{tikzpicture}
  \end{center}
  There is a similar vector space identification \(\Phi'\) of the codomain of \(\qfunc{D'}\) and \(\mathcal{D}_{\nr}(\chi_{4}) \otimes \mathcal{D}_{\nr}(\chi_{3})\).
  Unraveling the definitions shows that the action of the braiding of \(\mathcal{D}_{\nr}\) on the canonical vector is \(\Phi'(\qfunc{D'})\):
  \[
    \mathcal{D}_{\nr}(X)(v(\chi_{1}) \otimes v(\chi_{2}))
    =
    \Phi'(\qfunc{D'})
    .
  \]
  But by isotopy invariance
  \[
    \Phi'(
      \qfunc{D'}
    )
    =
    \Phi' \circ
    \qfunc{
    \begin{tikzpicture}[line width = 1, scale = 1, xscale = 1.5, baseline={(current bounding box.center)}]
      \coordinate (1) at (0,1);
      \coordinate (2) at (0,0);
      \coordinate (2d) at (0,-1);
      \coordinate (1d) at (0,-2);
      \draw[->] (2d) .. controls (-0.5,-1) and (-0.5, 0) .. (2);
      \draw[->] (1d) .. controls (-1,-2) and (-1, 1) .. (1);
      \node[right] at (1) {\(g_{4}\)};
      \node[right] at (2) {\(g_{3}\)};
      \node[right] at (1d) {\(g_{4}\)};
      \node[right] at (2d) {\(g_{3}\)};
      \node at (-1,1) {\(u\)};
    \end{tikzpicture}
    }
    =
    v(\chi_{4}) \otimes v(\chi_{3}).
  \qedhere
  \]
\end{proof}

\begin{proof}[Proof of \cref{thm:torsion relation}]
  Now assume \(\nr = 2\).
  In \cite{McPhailSnyder2020} we define a model  \(\mathcal{T}\) which is characterized by the conclusion of \cref{thm:double normalization}, so we must have \(\mathcal{T} = \mathcal{D}_{2}\).
  \cite[Theorem 2]{McPhailSnyder2020} says that the link invariant \(\operatorname{T}\) associated to \(\mathcal{T}\) computes the torsion, or in our normalization that
  \[
    \frac{
      \operatorname{T}(L, K; \rho, \mu)
      }{
      m + m^{-2} - 2
    }
    =
    \tau(S^{3} \setminus L, \rho)
    .
  \]
  Now \cref{thm:mirror computes mirror} finishes the proof.
\end{proof}

\section{\texorpdfstring{\(\qfunc{}\)}{𝒵ψ} is a model}
\label{sec:proofs}

Here we complete the proof of \cref{thm:qfunc defined and properties} by proving that \(\qfunc{}\) is a model of \(\tcat\).
It is obviously regular, so we need to check it is simple and sends the Reidemeister move diagrams to identity maps.

\subsection{Absolute simplicity of modules}%
\label{sec:Absolute simplicity of modules}

\begin{theorem}
  Let \((g, [v], u)\) be an admissible shadow coloring of a point and \(\mu\) a log-meridian.
  Then the  \(\qgrp\)-module
  \[
    V = \qfunc{g, [v], u} = \qpf{a, b, \mu}
  \]
  is \defemph{absolutely simple:}
  every morphism \(f : V \to V\) of \(\qgrp\)-modules is a scalar multiple of the identity.
\end{theorem}

\begin{proof}
  A simple module over an algebraically closed field (here \(\mathbb{C}\)) is absolutely simple by Schur's Lemma.
  We show that in most cases \(V \) is a simple \(\qgrp\)-module and that when this fails it is still absolutely simple.
  The simplicity results are not new.
  Here we give an elementary direct proof.

  Choose \(\alpha, \beta\) with \(\omega^{\nr \alpha} = a\) and \(\omega^{\nr \beta} = b\).
  The vectors
  \[
    v_{n} \defeq \omega^{(n - \alpha) \beta} \wb{\alpha - n}
  \]
  depend only on \(n \mod \nr\).
  They give a basis \(\set{v_{n} \given n \in \mathbb{Z}/\nr \mathbb{Z}}\) with \(\weyl\)-action
  \begin{align*}
    x \cdot v_{n}
    &=
    \omega^{\alpha - n} v_{n}
    \\
    y \cdot v_{n}
    &=
    \omega^{\beta} v_{n - 1}
    \\
    z \cdot v_{n}
    &=
    \omega^{\mu} v_{n}
  \end{align*}
  so 
  \begin{align*}
    K \cdot \vb{n} &= \omega^{\alpha - n} \vb{n}
    \\
    E \cdot \vb{n} &=  \omega^{1/2} \omega^{\beta - \mu}(1 - \omega^{\alpha - \mu + n}) \vb{n - 1}
    \\
    F \cdot \vb{n} &= \omega^{-\beta }(1 - \omega^{- \alpha - \mu + n}) \vb{n + 1}
  \end{align*}
  Observe that
  \begin{itemize}
    \item \( E \cdot \vb{n} = 0\) for some \(n\) if and only if \(a = m^{-1}\), and 
    \item \( F \cdot \vb{n} = 0\) for some \(n\) if and only if \(a = m\).
  \end{itemize}

  Let \(W\) be a nonempty submodule of \(V\).
  Because the action of \(K\) is always diagonalizable \(W\) must contain some \(K\)-eigenspace \(\operatorname{span} \vb{n}\). 
  If \(a \ne m\), then \(E\) acts invertibly and the orbit of \(\vb{n}\) under \(E\) will span \(V\), so \(W = V\) and \(V\) must be simple.
  The same argument works if \(F\) if \(a \ne m^{-1}\).

  The proof now breaks into cases depending on the conjugacy class of \(g\).
  If \(g\) is a diagonalizable, nonscalar matrix, then \(m \not \in \set{ \pm 1}\) so we cannot have both \(a = m\) and \(a = m^{-1}\) and thus \(V\) is simple.
  Now suppose \(g\) is parabolic, i.e.\@ a non-scalar, non-diagonalizable matrix.
  Let \(\epsilon \in \set{\pm 1}\) be the eigenvalue of \(g\).
  We claim that if \((g, [v], u)\) is admissible then \(a \ne \epsilon\) so as above \(V\) is simple.
  To prove the claim, use the parametrization of parabolic matrices in \cite[Section 5.2]{Inoue2014} to write
  \[
    g = 
    \begin{bmatrix}
    \epsilon -x y  & x^{2} \\
-y^{2} & \epsilon + x y
    \end{bmatrix}
  \]
  for a nonzero vector \((x, y) \in \mathbb{C}^{2}\), which spans \([v]\) because \(g\) has only one eigenspace.
  We have \(a = \epsilon\) if and only if \(u\) is a left eigenvector of \(g\), but the left eigenspace of \(g\) is spanned by \((-y, x)\), so if \(u = \alpha(-y,x)\) then
   \[
    vu =
    \alpha
    \begin{bmatrix}
      x & y
    \end{bmatrix}
    \begin{bmatrix}
      -y \\ x
    \end{bmatrix}
    =
    0
  \]
  and the coloring is inadmissible.

  Finally consider the case \(g\) is a scalar matrix.
  We have \(g = \epsilon I_{2}\) for \(\epsilon = \omega^{\nr \mu} \in \set{\pm 1}\) so \(2 \mu \in \mathbb{Z}\) and \(a = m\).
  The isomorphism class of \(V\) depends only on \(\mu \mod \nr\), so we may assume \(2 \mu = k\) or  \(2 \mu = k + \nr\) for \( k \in \set{0, \dots, \nr- 1} \) and that \(\alpha = \mu\).
  Then
  \begin{align*}
    K \cdot v_{n} &= \omega^{\mu - n} v_{n}
    \\
    E \cdot v_{n} &= \omega^{1/2 + \beta - \mu} (1 - \omega^{-n}) v_{n-1}
    \\
    F \cdot v_{n} &= \omega^{- \beta} (1 - \omega^{n - k}) \vb{n+1}
  \end{align*}
  so \(\ker E\) is spanned by \(\vb{0}\) and \(\ker F\) by \(v_{k}\).
  If \(k = \nr - 1\) or \(2\nr -1\) then \(V\) is simple; it is the one of the Steinberg modules denoted \(\operatorname{St}\) or \(\overline{\operatorname{St}}\) by \citeauthor{Suter1994} \cite{Suter1994}.
  Otherwise there is a proper  submodule \(W\) spanned by
  \(
  v_{0}, \dots, v_{k} 
  \)
  which is \((k+1)\)-dimensional and simple.
  The quotient \(V / W\) is an \((N - k - 1)\)-dimensional simple module.

  While in this case \(V\) is not simple, it is still absolutely simple.
  Suppose \(f : V \to V\) is a \(\qgrp\)-module intertwiner.
  In order to commute with the action of \(K\) it must be a diagonal matrix in the basis \(v_{n}\).
  Then commuting with the action of \(E\) ensures it is a scalar matrix.
\end{proof}

\subsection{The \texorpdfstring{\(\reid{2}\)}{R₂} moves}%
\label{sec:The R2 move}
Recall the diagram \(\reid{2a}\) of \cref{fig:reid moves that vanish}.
Our goal is to show that, for any admissible coloring \((\rho, u)\) of \(\reid{2a}\) and any log-decoration \(\mu\),
\[
  \qfunc{\reid{2a}, \rho, u, \mu} = \id.
\]
This will also show \(\qfunc{\reid{2b}} = \id\) because the right inverse of a square matrix is also a left inverse.

Any shadow coloring of \(\reid{2a}\) is as below:
\begin{equation}
  \label{eq:R2+- shadow coloring}
  \begin{tikzpicture}[line width=1, baseline=30, scale=2] 
    \draw[-] (0,0) node[left] {\(( g_{2}, [v_{2}] )\)} \br (1.5,1) ;
    \draw[-] (1.5,1) to (2.5,1);
    \draw (2,1) node[above] { \( ( g_{1}^{-1} g_{2} g_{1}, [v_{2} g_{1}] ) \) };
    \draw[white, line width=10] (0,1) node[left] {} \br (1.5,0) node[right] {};
    \draw[-] (0,1) node[left] {\((g_{1}, [v_{1}])\)} \br (1.5,0) ;
    \draw[-] (1.5,0) to (2.5,0);
    \draw (2,0) node[below] {\((g_{1}, [v_{1}])\)};
    \draw[->] (2.5,1)  \br (4,0) node[right] {\((g_{2}, [v_{2}])\)};
    \draw[white, line width=10] (2.5,0) node[left] {} \br (4,1) node[right] {};
    \draw[->] (2.5,0)  \br (4,1) node[right] {\((g_{1}, [v_{1}])\)};
    \draw (0.75, 1) node[above] { \(u\) };
    \draw (0, 0.5) node[right] { \(g_{1} u\) };
    \draw (4, 0.5) node[left] { \(g_{1} u\) };
    \draw (2, 0.5) node { \(g_{1}^{-1} g_{2} g_{1} u\) };
    \draw (0.75, 0) node { \(g_{2} g_{1} u\) };
  \end{tikzpicture}
\end{equation}
We can compute the segment parameters using \cref{eq:segment parameter}.
Label them as
\begin{equation*}
  \begin{tikzpicture}[line width=1, baseline=30, scale=1.5] 
    \draw[-] (0,0) node[left] {\( b_{2} \)} \br (1.5,1) ;
    \draw[white, line width=10] (0,1) node[left] {} \br (1.5,0) node[right] {};
    \draw[-] (0,1) node[left] {\( b_{1} \)} \br (1.5,0) ;
    \draw[-] (1.5,1) to (2.5,1);
    \draw (2,1) node[above] { \( b_{4} \) };
    \draw[-] (1.5,0) to (2.5,0);
    \draw (2,0) node[below] {\( b_{3} \)};
    \draw[->] (2.5,1)  \br (4,0) node[right] {\( b_{2} \)};
    \draw[white, line width=10] (2.5,0) node[left] {} \br (4,1) node[right] {};
    \draw[->] (2.5,0)  \br (4,1) node[right] {\( b_{1} \)};
  \end{tikzpicture}
\end{equation*}
Assume that the crossing is not pinched; the pinched case follows by taking a limit.
Write \(\mu_{1}, \mu_{2}\) for the log-meridian of the top and bottom strands, respectively.
Choosing logarithms \(\beta_{i}\) of the segment parameters, the operator \(\mathcal{Z} = \qfunc{\reid{2a}, \rho, \mu, u}\) has
\begin{align*}
  \mathcal{Z}\leftfun( \ket{\beta_{1}} \otimes \ket{\beta_{2}} \rightfun)
  =
  \frac{
    1
    }{
    \nr^{2}
  }
  \sum_{n_{3}, n_{4} = 0}^{\nr-1}
  &
  \frac{
    \pf{\beta_{4} + n_{4} - \beta_{1}}
    \pf{\beta_{2} - \beta_{3} - n_{3} + \mu_{2} - \mu_{1}}
  }{
    \pf{\beta_{4} - \beta_{3} + n_{4} - n_{3}+ \mu_{2}}
    \pft{\beta_{2} - \beta_{1} -  \mu_{1}}
  }
  \\
  &\times
  \omega^{
    \mu_{1}(\beta_{3} +n_{3} - \beta_{1}) 
    +
    \mu_{2}(\beta_{2} - \beta_{4} - n_{4})
    -
    \mu_{1} \mu_{2}
  }
  \\
  &\times
  \frac{
    \pft{- \beta_{5} + \beta_{6} - \mu_{1}}
    \pf{- \beta_{3} - n_{3} + \beta_{4} + n_{4} + \mu_{2}}
    }{
    \pft{-\beta_{5}  + \beta_{4} + n_{4}}
    \pft{ - \beta_{3} - n_{3} + \beta_{6} + \mu_{2} - \mu_{1}}
  }
  \\
  &\times
  \omega^{
    \mu_{2}(\beta_{4} + n_{4} - \beta_{6}) 
    +
    \mu_{1}(\beta_{5} - \beta_{3}  - n_{3})
    +
    \mu_{1} \mu_{2}
  }
  \ket{\beta_{5}} \otimes \ket{\beta_{6}}
\end{align*}
where \(\beta_{5}\) is the logarithm of \(b_{1}\) associated to the outgoing segment and similarly for \(\beta_{6}\).
Note that \(\beta_{5} \equiv \beta_{1}, \beta_{6} \equiv \beta_{1} \pmod{\nr \mathbb{Z}}\).
After applying the obvious cancellations we have
\begin{align*}
  \mathcal{Z}\leftfun( \ket{\beta_{1}} \otimes \ket{\beta_{2}} \rightfun)
  =
  \frac{
    1
    }{
    \nr^{2}
  }
  \frac{
    \pf{\beta_{6} - \beta_{5} - \mu_{1}}
    }{
    \pft{\beta_{2} - \beta_{1} - \mu_{1}}
  }
  &
  \sum_{n_{3}= 0}^{\nr-1}
  \frac{
    \pf{\beta_{2} - \beta_{3} - n_{3} + \mu_{2} - \mu_{1}}
    }{
    \pft{\beta_{6} - \beta_{3} - n_{3} + \mu_{2} - \mu_{1}}
  }
  \\
  &\times
  \sum_{n_{4}= 0}^{\nr-1}
  \frac{
    \pf{\beta_{4} - \beta_{1} + n_{4}}
    }{
    \pft{\beta_{4} - \beta_{5} + n_{4}}
  }
  \ket{\beta_{5}} \otimes \ket{\beta_{6}}
\end{align*}
Applying \cref{eq:pf cancellation} twice and using \cref{eq:pf recurrence} gives
\begin{align*}
  \mathcal{Z}\leftfun( \ket{\beta_{1}} \otimes \ket{\beta_{2}} \rightfun)
  &=
  \frac{
    \pft{\beta_{6} - \beta_{5} - \mu_{1}}
    }{
    \pft{\beta_{2} - \beta_{1} - \mu_{1}}
  }
  \nrdirac{\beta_{6} - \beta_{2}}
  \nrdirac{\beta_{5} - \beta_{1}}
  \\
  &\phantom{=}\times
  (1 - b_{2}m_{2}/b_{3}m_{1})^{(\beta_{6} - \beta_{2})/\nr}
  (1 - b_{4}/b_{1})^{(\beta_{5} - \beta_{1})/\nr}
  \ket{\beta_{5}} \otimes \ket{\beta_{6}}
\end{align*}
In particular if \(n_{1}, n_{2} \in \set{0, \dots, \nr -1}\) we have
\[
  \bra{\beta_{1} + n_{1}} \otimes \bra{\beta_{2} + n_{2}}
  \mathcal{Z}\leftfun( \ket{\beta_{1}} \otimes \ket{\beta_{2}} \rightfun)
  =
  \begin{cases}
    1 & n_{1} = n_{2} = 0 \\
    0 & \text{otherwise}
  \end{cases}
\]
which shows that \(\mathcal{Z}\) is the identity operator as claimed.

\subsection{The sideways \texorpdfstring{\(\reid{2}\)}{R₂} moves}%
The sideways moves \(\reid{2a}\) and  \(\reid{2b}\) are similar to the \(\reid{2}\) move.
As before it suffices to check the \(\reid{2b}\) move.
Choose an admissible shadow coloring and log-meridians and assign logarithms of the segment parameters as
\begin{equation}
  \begin{tikzpicture}[line width=1, baseline=30, scale=1.5] 
    \draw[-] (0,1) node[left] {\(\beta_{2}\)} \br (1.5,0) ;
    \draw[white, line width=10] (0,0) \br (1.5,1) ;
    \draw[<-] (0,0) node[left] {\(\beta_{3}\)} \br (1.5,1) ;
    \draw[->] (1.5,0)  \br (3,1) node[right] {\(\beta_{2'}\)};
    \draw[white, line width=10] (1.5,1)  \br (3,0);
    \draw[-] (1.5,1)  \br (3,0) node[right] {\(\beta_{3'}\)};
    \draw (1.5,1) node[above] {\(\beta_{1}\)};
    \draw (1.5,0) node[below] {\(\beta_{4}\)};
  \end{tikzpicture}
\end{equation}
Let \(\mathcal{Z} = \qfunc{D, \rho, u, \mu}\) as before.
Using the expressions for the sideways \(R\)-matrices in \cref{sec:State-sum description} and applying some obvious cancellations we see that
\begin{align*}
  \mathcal{Z}(\ket{\beta_{2}} \otimes \bra{\beta_{3}})
  =
  \frac{
    1
    }{
    \nr^{2}
  }
  \sum_{n_{1}, n_{4} = 0}^{\nr-1}
  &\frac{
    \pf{\beta_{4} - \beta_{1} + n_{4} - n_{1}}
    \pf{\beta_{2} - \beta_{3} + \mu_{2} - \mu_{1}}
    }{
    \pf{\beta_{4} - \beta_{3} + \mu_{2} + \nr - 1 + n_{4}}
    \pft{\beta_{2} - \beta_{1} - \mu_{1} - n_{1}}
  }
  \\
  &\times
  \frac{
    \pft{\beta_{2'} - \beta_{1} - \mu_{1} - \nr + 1 - n_{1}}
    \pf{\beta_{4} - \beta_{3'} + \mu_{2} + n_{4}}
    }{
    \pft{\beta_{4} - \beta_{1} + \nr - 1 + n_{4} - n_{1}}
    \pft{\beta_{2'} - \beta_{3'} + \mu_{2} - \mu_{1}}  
  }
  \\
  &\times
  \omega^{\mu_{1}(\beta_{3} - \beta_{3'} ) + \mu_{2}(\beta_{2} - \beta_{2'})}
  \ket{\beta_{2'}} \otimes \bra{\beta_{3'}}
\end{align*}
Using the definition \eqref{eq:pft def} of \(\pft{}\) we can re-write this as
\begin{align*}
  \mathcal{Z}(\ket{\beta_{2}} \otimes \bra{\beta_{3}})
  =
  &\frac{
    \omega^{\mu_{1}(\beta_{3} - \beta_{3'} ) + \mu_{2}(\beta_{2} - \beta_{2'})}
    }{
    \nr^{2}
  }
  \frac{
    \pf{\beta_{2} - \beta_{3} + \mu_{2} - \mu_{1}}
    }{
    \pft{\beta_{2'} - \beta_{3'} + \mu_{2} - \mu_{1}}
  }
  \ket{\beta_{2'}} \otimes \bra{\beta_{3'}}
  \\
  &\times
  \sum_{n_{1}= 0}^{\nr-1}
  \frac{
    \pf{\beta_{2'} - \beta_{1} - \mu_{1} - n_{1}}
    }{
    \pft{\beta_{2} - \beta_{1} - \mu_{1} - n_{1}}
  }
  \\
  &\times
  \sum_{n_{4}= 0}^{\nr-1}
  \frac{
    \pf{\beta_{4} - \beta_{3'} + \mu_{2} + n_{4}}
    }{
    \pft{\beta_{4} - \beta_{3} + \mu_{2} + n_{4}}
  }
\end{align*}
As before we can use \cref{eq:pf cancellation} to reduce this to some periodic delta functions and quasi-periodicity constants and conclude that \(\mathcal{Z}\) is the identity operator.

\subsection{The \texorpdfstring{\(\reid{3}\)}{R₃} move}%
\label{sec:The R3 move}

Our proof follows \cite{McPhailSnyderAlgebra} but uses the language of shadow colorings and gives more detail.
We need to show that \(\qfunc{\reid{3}, \rho, u, \mu}\) is the identity map for every admissible shadow coloring \((\rho, u)\) and log-decoration \(\mu = (\mu_{1}, \mu_{2}, \mu_{3})\) of the diagram \(\reid{3}\) of \cref{fig:reid moves that vanish:3}.
By \cref{thm:longitude eigenvalues make sense}\ref{thm:longitude eigenvalues make sense:preserved} \(\ell_{i}^{2}(\reid{3}) = 1\) for every component, so by \cref{thm:qfunc braiding mu dependence} \(\qfunc{}\) depends only on the \(\nr\)th roots \(\omega^{\mu_{i}}\) of the meridian eigenvalues, not their logarithms \(\mu_{i}\).
We thus consider the \(\nr^{3}\)-fold covering space
\[
  \widetilde{\mathsf{A}}_{3} \to \admvar{\reid{3}}
\]
whose points are of an admissible shadow coloring \((\rho, u)\) of \(\reid{3}\) and a choice \(\omega^{\mu_{1}}, \omega^{\mu_{2}}, \omega^{\mu_{3}}\) of \(\nr\)th  roots.
By \cref{thm:reid3 is connected} \( \admvar{\reid{3}} \) is connected so it is clear that \(\widetilde{\mathsf{A}}_{3} \) is as well.

\begin{lemma}
  There is a holomorphic function \(\Upsilon : \widetilde{\mathsf{A}}_{3} \to \mathbb{C}\) so that
  \[
    \qfunc{\reid{3}, \tilde{a}} = \Upsilon(\tilde{a}) \id.\qedhere
  \]
\end{lemma}

\begin{proof}
  \(\qfunc{\reid{3}}\) is an endomorphism-valued function on \(\admvar{\reid{3}}\), and by the projective uniqueness of the braiding (\cref{thm:uniqueness of R}) it must be a scalar.
  This scalar is \(\Upsilon\), and it is holomorphic just as in the proof of \cref{thm:qfunc is holomorphic}.
\end{proof}

It remains to show that \(\Upsilon =1\) uniformly.
We can show that it is a root of unity by computing the determinant of \(\qfunc{}\).
More formally, we can obtain a new model \(\extpower \qfunc{}\) from \(\qfunc{}\) by replacing each module \(\qfunc{\chi}\) with its top exterior power  \(\extpower \qfunc{\chi}\) and the braiding morphisms by their induced actions on these.
If we choose bases of the modules (say, by choosing logarithms of the segment parameters) then we can compute the map
\[
  \extpower \qfunc{X}
  \colon
  \extpower \qfunc{\chi_{1}}
  \otimes
  \extpower \qfunc{\chi_{2}}
  \to
  \extpower \qfunc{\chi_{4}}
  \otimes
  \extpower \qfunc{\chi_{3}}
\]
by taking the determinant of the matrix of \(\qfunc{X}\) with respect to these bases.

The positive crossing map \(\extpower \qfunc{X}\) is a map between complex lines.
More concretely, it has a single matrix coefficient (since it is a map between \(1\)-dimensional vector spaces) \(M(\beta)\) depending on a choice of logarithms \(\beta = (\beta_{1}, \dots )\) of the segment parameters at the crossing.
Just like the braiding kernels \(\brkern{\pm}{}{}\) and octahedral functions \(e^{- \octfunc{\pm}{}{}}\) these satisfy quasi-periodicity relations.
Examining these shows that the relations for \(M(\beta)\) and \(e^{- \nr \octfunc{\pm}{}{}}\) are the same.
It therefore makes sense to say that
\begin{equation}
  \label{eq:determinant proportional to cs}
  \extpower \qfunc{X, \rho, u, \mu}
  =
  C(\rho, u, \mu)
  \left( \csfunc{X, \rho, u, \mu} \right)^{-\nr}
\end{equation}
for some scalar-valued function \(C\): both sides are elements of the same vector space of transformations between spaces of quasi-periodic functions.
Since \(C\) is a function defined on the space of colorings of a crossing, we can think of it as a pre-model \(\mathcal{F}_{C}\) valued in \(\vect\) as in \cref{sec:Shadow cocycles}, and we can re-write \eqref{eq:determinant proportional to cs} as
\[
  \extpower \qfunc{} = 
  \mathcal{F}_{C}
  \boxtimes
  \left(\csfunc{}\right)^{-\nr}.
\]
Because we know that \(\csfunc{}\) and thus \(\left(\csfunc{}\right)^{-\nr}\) are models this reduces showing that \(\extpower \qfunc{}\) is a model to checking the same for \(\mathcal{F}_{C}\).

To do this we compute \(C\) by taking determinants.
Translating \cite[Lemma 5.7]{McPhailSnyderAlgebra} to our normalization gives
\[
  \extpower \qfunc{X, \rho, u}
  =
  \left(
    \csfunc{X, \rho, u}
  \right)^{-\nr}
    \frac{
      \nr^{\nr^{2}}
    }{
      \operatorname{D}_{\nr}^{2 \nr}
    }
  \left[
    \frac{
      L(\rho, u)
    }{
      M(\rho, u)^{2} 
    }
  \right]^{\binom{\nr}{2}}
\]
so that
\[
  C(\rho, u)
  = 
    \frac{
      \nr^{\nr^{2}}
    }{
      \operatorname{D}_{\nr}^{2 \nr}
    }
  \left[
    \frac{
      L(\rho, u)
    }{
      M(\rho, u)^{2} 
    }
  \right]^{\binom{\nr}{2}}
\]
Here \(L\) and \(M\) are the generic cocycles defined in \cref{sec:Shadow cocycles} and
\(
  \operatorname{D}_{\nr}
  =
  \df{0}^{\nr}
  =
  \prod_{k = 1}^{\nr -1}
  (1 - \omega^{k})^{k}
\)
is a constant. 
Because \(C\) is a product of generic cocycles it is one as well, so 
\(
  \mathcal{F}_{C}
\)
is a model.

\begin{lemma}
  \(
    \Upsilon^{\nr^3} =  1
  \)
  uniformly.
\end{lemma}

\begin{proof}
  For any point \(\tilde{a} \in \textsf{A}_{3}\) the above computation shows
  \[
    \extpower \qfunc{\reid{3}, \tilde{a}}
    =
    \mathcal{F}_{C}(\reid{3}, \tilde{a})
    \boxtimes
    \left( \csfunc{\reid{3}, \tilde{a}} \right)^{-\nr}
    =
    \id
  \]
  because \(\mathcal{F}_{C}\) and \(\csfunc{}\) hence \((\csfunc{})^{-\nr}\) are models. 
  On the other hand we know that \(\qfunc{\reid{3}} = \Upsilon \id\) is a scalar endomorphism of a \(\nr^{3}\)-dimensional vector space, so
  \[
    \extpower \qfunc{\reid{3}, \tilde{a}}
    =
    \Upsilon(\tilde{a})^{\nr^{3}}
    \id.
    \qedhere
  \]
\end{proof}

We now know that \(\Upsilon\) is a continuous function on a connected space \(\widetilde{\mathsf{A}}_{3}\) valued in the discrete set of \(\nr^{3}\)th roots of unity.
We conclude it is constant, and it remains only to check that the constant is \(1\).
Consider the shadow coloring \((\rho_{0}, u_{0})\) of \(D_{3}\) where each strand is assigned the decorated matrix \(( (-1)^{\nr -1}, (1,1))\) and \(u_{0} = (1,0)\).
If we choose \(\omega^{\mu_{i}} = \omega^{(\nr - 1)/2}\) for each strand, then by \cref{thm:KADO} and \cref{eq:colored-Jones-R-matrix} the braiding at each crossing is the one defining the colored Jones polynomial at \(q = \omega^{1/2}\), equivalently the Kashaev invariant.
Since these satisfy the braid relation \(\qfunc{\reid{3}, \rho_{0} , u_{0}, (\nr - 1)/2}\) is the identity map and thus \(\Upsilon(\rho_{0}, u_{0}, \omega^{(\nr - 1)/2}) = 1\).

\subsection{The \texorpdfstring{\(\reid{1f}\)}{R₁ᶠʳ} move}%
\label{sec:The R1f move}

To show that \(\qfunc{\reid{1f}}\) is the identity we compute the left and right hand positive twists and confirm that they have the same value.

\begin{lemma}
  For any log-meridian \(\mu\),
  \begin{equation}
    \label{eq:twist right hand positive}
    \qfunc{
      \begin{tikzpicture}[line width = 1, scale = 0.75, xscale = 1.5, baseline={(current bounding box.center)}]
        \draw[-] (-0.5,0) to [out = 00, in = 60] (1,-1);
        \draw[-] (1,-1) to [out = -120, in = -60] (0,-1);
        \draw[white, line width=10] (0,-1) to [out = -240, in = 180] (1.5,0);
        \draw[->] (0,-1) to [out = -240, in = 180] (1.5,0);
        \node[left] at (-0.5,0) {\((g, [v], \mu)\)};
        \node[above] at (0.5, 0) {\(u\)};
      \end{tikzpicture}
    }
    =
    \omega^{\mu(\mu + 1 - \nr)}
    \id_{\qfunc{g, [v], \mu}}
  \end{equation}
\end{lemma}

\begin{proof}
  The crossing is always pinched and represents an endomorphism of an absolutely simple object, so the relevant variable assignment for the state-sum is
  \begin{center}
    \begin{tikzpicture}[line width = 1, scale = 1, xscale = 1.5, baseline={(current bounding box.center)}]
      \draw[-] (-0.5,0) to [out = 00, in = 60] (1,-1);
      \draw[-] (1,-1) to [out = -120, in = -60] (0,-1);
      \draw[white, line width=10] (0,-1) to [out = -240, in = 180] (1.5,0);
      \draw[->] (0,-1) to [out = -240, in = 180] (1.5,0);
      \node[left] at (-0.5,0) {\(\beta\)};
      \node[right] at (1.5,0) {\(\beta\)};
      \node[left] at (0,-0.75) {\(\beta + \mu + n \)};
      \node[right] at (1,-0.75) {\(\beta + \mu + n + 1 - \nr\)};
    \end{tikzpicture}
  \end{center}
  where we sum \(n\) from \(0\) to \(\nr - 1\) .
  The matrix coefficients of the braiding are computed in \cref{thm:pinched braiding is defined} for general pinched crossings.
  Here we can directly take the limit of \cref{eq:braiding kernel positive} to get
  \begin{align*}
    &\frac{
      1
      }{
      \nr
    }
    \sum_{n = 0}^{\nr -1}
    \frac{
      \pf{0} \pf{\nr - 1}
      }{
      \pft{n} \pf{\nr - 1 - n}
    }
    \omega^{\mu(1 - \nr + 2n + 2 \mu) - \mu^{2}}
    \\
    &=
    \frac{
      \omega^{\mu(1 - \nr + \mu)}
      }{
      \nr
    }
    \sum_{n = 0}^{\nr -1}
    \frac{
      \pf{0} \pf{\nr - 1}
      }{
      \pf{\nr - 1 + n} \pf{\nr - 1 - n}
    }
    \omega^{n(2 \mu + 1)}
  \end{align*}
  Since \(\pf{\nr}^{-1} = 0\) the summand vanishes unless \(n = 0\) and the state-sum is
  \[
    \frac{
      \omega^{\mu(1 - \nr + \mu)}
      }{
      \nr
    }
    \frac{
      \pf{0} \pf{\nr - 1}
      }{
      \pf{\nr - 1} \pf{\nr - 1}
    }
    =
    \frac{
      \omega^{\mu(1 - \nr + \mu)}
      }{
      \nr
    }
    (1 - \omega) \cdots (1 - \omega^{\nr -1})
    =
    \omega^{\mu(1 - \nr + \mu)}
    .
    \qedhere
  \]
\end{proof}

The computation for the left-hand twist is essentially identical:
now the variable assignment is
\begin{center}
  \begin{tikzpicture}[line width = 1, scale = 1, xscale = 1.5, baseline={(current bounding box.center)}]
    \draw[-] (-0.5,0) to [out = 00, in = -60] (1,1);
    \draw[-] (1,1) to [out = 120, in = 60] (0,1);
    \draw[white, line width=10] (0,1) to [out = 240, in = 180] (1.5,0);
    \draw[->] (0,1) to [out = 240, in = 180] (1.5,0);
    \node[left] at (-0.5,0) {\(\beta\)};
    \node[right] at (1.5,0) {\(\beta\)};
    \node[left] at (0,0.5) {\(\beta + \mu + n \)};
    \node[right] at (1,0.5) {\(\beta + \mu + n + \nr - 1\)};
  \end{tikzpicture}
\end{center}
and one can similarly show that only the \(n = 0\) term contributes \( \omega^{\mu(1 - \nr + \mu)} \).

\appendix

\section{Dilogarithms}%
\label{sec:quantum dilogarithms}
Here we set conventions on classical and quantum dilogarithms and state some known properties used in the main text.
Recall that \(\nr \ge 2\) is an integer and
\[
  \omega \defeq \exp(2\pi \ii /\nr) \text{ and } \omega^{x} \defeq \exp(2\pi \ii x/\nr) \text{ for } x \in \mathbb{C}.
\]
To define the complex logarithm we take a branch cut along the negative real axis and choose \(\Im \log \in (-\pi, \pi]\).

The \defemph{dilogarithm} is the analytic function \(\mathbb{C} \setminus [1, \infty) \to \mathbb{C}\) defined by
\[
  \operatorname{Li}_2(z)
  =
  - \int_{0}^{z} \frac{\log(1 - t)}{t} \, d t
\]
\textcite{Zagier2007} gives an overview of its properties.
To deal with the branch cut, define
\begin{equation}
  \label{eq:ldil-def}
  \ldil{\zeta}
  \defeq
  \frac{
    \operatorname{Li}_2(e^{\tu \zeta})
  }{\tu}
  +
  \zeta \log(1 - e^{\tu \zeta})
\end{equation}
It satisfies the recurrence
\begin{align}
  \label{eq:ldil-0}
  \exp \ldil{\zeta + k}
  &=
  (1 - e^{\tu \zeta})^{k}
  \exp \ldil{\zeta}
\end{align}

\begin{lemma}
  \label{thm:dilog is meromorphic}
  \(\exp \ldil{}\) is a meromorphic function on \(\mathbb{C}\) with zeros at positive integers and poles at negative integers.
\end{lemma}

\begin{proof}
  \(\operatorname{Li}_{2}(z)\) and \(\log(1-z)\) both have branch cuts along \([1, \infty)\), which give potential discontinuities of \(e^{\ldil{\zeta}}\) along the sets \( k - i [0, \infty), k \in \mathbb{Z} \).
  Near one of these cuts
  \[
    e^{\ldil{\zeta}} = \exp \leftfun(
    \frac{
      \operatorname{Li}_{2}(z) + [\log(z) + \tu k ] \log(1-z) 
      }{
        \tu
      } 
      \rightfun)
  \]
  where \(z = e^{\tu \zeta}\).
  Viewing things in the \(z\)-plane, when \(z\) crosses the branch cut from above to below \(\operatorname{Li}_{2}(z)/\tu\) picks up a factor of \(-\log z \) and \([\log z + \tu z] \log(1 - z)\) picks up a factor of \( \log z+ \tu k\).
  In total \(\ldil{\zeta}\) jumps by \(\tu k\) so \(e^{\ldil{\zeta}}\) is continuous across the branch cuts.
  In particular it is continuous in a neighborhood of \(0\), with finite value
  \[
    e^{\ldil{0}} = e^{\operatorname{Li}_{2}(1)/\tu} = e^{\pi \ii /12}.
  \]
  so it is holomorphic near \(0\).
  The claim about poles and zeros follows from \eqref{eq:ldil-0}.
\end{proof}

This lift of \(\operatorname{Li}_{2}\) is slightly different than the one used by \citeauthor{Neumann2004} \cite[Proposition 2.5]{Neumann2004}.
We discuss the relationship further in \cref{sec:recovering Neumann}

Just as \(\csfunc{}\) is defined using \(\ldil{}\) the quantum invariant \(\qfunc{}\) is built with a closely related function called a quantum dilogarithm.
For \(z \in \mathbb{C}\) with \(|\Re z - \tfrac{\nr-1}{2 \nr}| < 1 + \tfrac{1}{\nr}\), define
\[
  \phi_{\nr}(z)
  \defeq
  \int_{\mathbb{R} + \ii \epsilon}
  \frac{
    \exp w[2z -  1 + 1/\nr]
  }{
    4 \sinh(w ) \sinh(w/ \nr)
  }
  \frac{dw}{w}
\]
and set
\[
  \pf{\zeta}
  =
  \pf[\nr]{\zeta}
  \defeq
  e^{\phi_{\nr}(\zeta/\nr)}
\]
It is a version of \defemph{Faddeev's noncompact quantum dilogarithm} \cite[Section 6]{Faddeev2001}
\begin{equation}
  \Phi_{\mathsf b}(z) \defeq
  \exp
  \int_{\mathbb{R} + \ii \epsilon}
  \frac{
    \exp(-2 \ii zw)
  }{
    4 \sinh(w \mathsf b ) \sinh(w/\mathsf b)
  }
  \frac{dw}{w}
\end{equation}
with
\begin{equation}
  \pf[\nr]{\zeta}
  =
  \Phi_{\sqrt \nr}\left(\ii \frac{\zeta}{\sqrt \nr} + \frac{\ii}{2\sqrt \nr} - \frac{\ii \sqrt \nr}{2} \right)
\end{equation}
In particular \(\pf{}\) satisfies the recurrence relation
\begin{equation}
  \label{eq:pf-recurrence}
  \pf{\zeta + k}
  = \pf{\zeta} (1 - \omega^{\zeta + 1})^{-1} (1 - \omega^{\zeta + 2})^{-1} \cdots (1 - \omega^{\zeta + k})^{-1}
  =
  \pf{\zeta}
  \qlog{\zeta}{k}
\end{equation}
where
\[
  \qlog{\zeta}{k}
  \defeq
  \qp{\omega^{\zeta+1}}{k}^{-1}
  =
  (1 - \omega^{\zeta + 1})^{-1} (1 - \omega^{\zeta + 2})^{-1} \cdots (1 - \omega^{\zeta + k})^{-1}
\]
is a version of the \defemph{\(q\)-Pochammer symbol}
\begin{equation}
  \label{eq:qp-def}
  \qp{a}[q]{k}
  \defeq
  \begin{cases}
    (1 - a)(1-aq) \cdots (1 - a q^{k-1})
    &
    k > 0
    \\
    1
    &
    k = 0
    \\
    \left[
      (1-aq^{-1}) \cdots (1 - a q^{k})
    \right]^{-1}
    &
    k < 0
  \end{cases}
\end{equation}

\Cref{eq:pf-recurrence} shows that \(\pf{}\) can be analytically continued to a meromorphic function on \(\mathbb{C}\) with an essential singularity at infinity. 
\(\pf{}\) has zeros at negative integers and poles at positive integers \(n \ge \nr\).
A special case of \eqref{eq:pf-recurrence} gives the essential quasi-periodicity relation
\begin{equation}
  \label{eq:pf-quasiperiodic}
  \pf{\zeta + \nr}
  =
  \pf{\zeta}
  (1 - \omega^{\nr \zeta})^{-1}
  .
\end{equation}

The function
\begin{equation}
  \df{\zeta}
  = \prod_{k=1}^{\nr -1} (1 - \omega^{\zeta + k})^{k/\nr}
  \defeq
  \exp \leftfun( \frac{1}{\nr} \sum_{k=1}^{\nr-1} k \log(1 - \omega^{\zeta + k}) \rightfun)
\end{equation}
is also called a cyclic quantum dilogarithm, and is used by \textcite{Baseilhac2004} in their construction of geometric quantum invariants.
It is closely related to \(\pf{}\) but instead satisfies the relation
\[
  \df{\zeta + k}
  =
  \df{\zeta} (1 - \omega^{\zeta}) \cdots (1 - \omega^{\zeta + k -1}) (1 - \omega^{\nr \zeta})^{-k/\nr}
\]
where the \(\nr\)th root is computed using \(\log\).
This introduces an overall \(\nr\)th root indeterminacy to the quantum invariants defined in \cite{Baseilhac2004}.
By a result \cite[Theorem 1.9]{Garoufalidis2014} of \citeauthor{Garoufalidis2014}, \(\operatorname{D}\) is related to \(\pf{}\) by
\begin{equation}
  \label{eq:exact-qlf-value}
  \pf{\zeta}
  =
  \frac{1 - \omega^{\nr \zeta}}{1 - \omega^{\zeta}}
    e^{-\ldil{\zeta}/\nr}
    \df{\zeta}^{-1}
  =
  \exp \frac{1}{\nr}\leftfun[
  - \ldil{\zeta}
  + \sum_{k=1}^{\nr -1} (\nr - k) \log(1 - \omega^{\zeta + k})
  \rightfun ]
\end{equation}
where \(\ldil{}\) is the lifted classical dilogarithm \eqref{eq:ldil-def}.
One interpretation of \cref{eq:exact-qlf-value} is that \(\qfunc{}\) is the ratio of an invariant defined in terms of  \(\operatorname{D}\) and the \(\nr\)th root of \(\csfunc{}\).
Both of these factors have a power of \(\omega\) indeterminacy which in some sense cancels in \(\qfunc{}\).

It is convenient to abbreviate
\begin{equation}
  \label{eq:pft def}
  \pft{\zeta}
  \defeq
  \pf{\zeta + \nr - 1} \omega^{(\nr-1)\zeta}.
\end{equation}
For example, the inversion relation \cite[equation (16)]{Faddeev2001} can be written as
\begin{equation}
  \label{eq:pf inversion}
  \pf{\zeta} \pft{-\zeta}
  =
  \omega^{-\zeta(\zeta + \nr - 1)/2}
  \exp
  \frac{
    \pi \ii
    }{
    6
  }
  \left(
    3
    -
    \nr
    -  1/\nr
  \right)
\end{equation}
and we can write a key summation identity in a simple form:

\begin{lemma}
  \label{thm:f delta lemma}
  For \(n \in \mathbb{Z}\),
  \begin{equation}
    \label{eq:pf cancellation}
    \sum_{k = 0}^{\nr -1} 
    \frac{
      \pf{\zeta + k}
    }{
      \pft{\zeta + n + k}
    }
    =
    \delta_{\nr}(n)
    \nr
    (1 - \omega^{\nr \zeta})^{-n/\nr}
  \end{equation}
  where \(\delta_{\nr}(k) = 1\) if \(k \in \nr \mathbb{Z}\) and \(0\) otherwise.
\end{lemma}

\begin{proof}
  Writing
   \begin{align*}
    \frac{
      \pf{\zeta + k}
      }{
      \pft{\zeta + n + k}
    }
    &=
    \frac{
      1 - \omega^{\nr \zeta}
      }{
      \omega^{(\nr - 1)(\zeta + n )}
    }
    \frac{
      \pf{\zeta + k}
      }{
      \pf{\zeta + n + k - 1}
    }
    \omega^{k}
    \\
    &=
    \frac{
      1 - \omega^{\nr \zeta}
      }{
      \omega^{(\nr - 1)(\zeta + n)}
    }
    \frac{
      \pf{\zeta }
      }{
      \pf{\zeta + n - 1}
    }
    \frac{
      \qlog{\zeta}{k}
      }{
      \qlog{\zeta + n - 1}{k}
    }
    \\
    &=
    \omega^{-(\nr - 1)(\zeta + n)}
    \frac{
      1 - \omega^{\nr \zeta}
      }{
      1 - \omega^{\zeta + n}
    }
    \frac{
      \pf{\zeta }
      }{
      \pf{\zeta + n}
    }
    \frac{
      \qlog{\zeta}{k}
      }{
      \qlog{\zeta + n - 1}{k}
    }
  \end{align*}
  we see that
  \begin{align*}
    \sum_{k = 0}^{\nr -1} 
    \frac{
      \pf{\zeta + k}
    }{
      \pft{\zeta + n + k}
    }
    =
    \omega^{-(\nr - 1)(\zeta + n)}
    \frac{
      1 - \omega^{\nr \zeta}
      }{
      1 - \omega^{\zeta + n}
    }
    \frac{
      \pf{\zeta }
      }{
      \pf{\zeta + n}
    }
    \fstack{\zeta }{\zeta + n - 1}{1}
  \end{align*}
  in terms of the function \(f\) of \cite[Appendix A.2]{McPhailSnyderAlgebra} studied by \citeauthor{Kashaev1993} \cite{Kashaev1993}.
  By \cite[eq. A.23]{McPhailSnyderAlgebra} (see also \cite[eq. 40]{Garoufalidis2021descendant})
  \begin{align*}
    \fstack{\zeta }{\zeta + n - 1}{1}
    &=
    \nr
    \frac{
      1 - \omega^{\zeta + n}
      }{
      1 - \omega^{\nr \zeta}
    }
    \frac{
      \omega^{(\nr -1)(\zeta + n)}
      }{
      \qlog{\zeta + n}{ \modb{- n}}
    }
    \frac{
      (\omega)_{\modb{-n}  + \nr - 1}
      }{
      (\omega)_{\modb{ n}}
      (\omega)_{\nr - 1}
    }
  \end{align*}
  where \(\modb{n}\) is the element of \(\{0, \dots, \nr -1\}\) with \(n \equiv \modb{n} \mod{\nr} \) and
  \[
    (\omega)_{k} \defeq \qp{\omega}{k} = (1 - \omega) \cdots (1- \omega^{k}).
  \]
  Because \((\omega)_{k} = 0\) for \(k \ge \nr\) this vanishes unless \(n \in \nr \mathbb{Z}\).
  The identity now follows from
  \[
    \pf{\zeta + \ell \nr}
    =
    (1 - \omega^{\nr \zeta})^{-\ell}
    .
    \qedhere
  \]
\end{proof}

\printbibliography

\end{document}

%% file: macros-quantum.tex
\newcommand{\nr}{N}
\newcommand{\set}[1]{\left\{ \def\given{\ \middle| \ }  #1 \right\}  }
\newcommand{\defeq}{\mathrel{:=}}
\NewDocumentCommand{\iso}{}{\cong}
\DeclareMathOperator{\tr}{tr}
\DeclareMathOperator{\id}{id}
\DeclareMathOperator{\End}{End}
\renewcommand{\hom}{\operatorname{Hom}}

\renewcommand{\Re}{\operatorname{Re}}
\renewcommand{\Im}{\operatorname{Im}}

\NewDocumentCommand{\du}{}{\amalg}

\ExplSyntaxOn
\DeclareExpandableDocumentCommand{\IfEmptyTF}{mmm}
 {
    \tl_if_empty:nTF {#1} {#2} {#3}
 }
\ExplSyntaxOff

\def\leftfun#1{\mathopen{}\left#1}
\def\rightfun#1{\right#1}

\NewDocumentCommand{\ParenIfNonempty}{m}{%
  \IfEmptyTF{#1}{%
  }{%
    \leftfun(#1\rightfun)%
  }
}

\NewDocumentCommand{\ii}{}{\mathsf{i}}
\NewDocumentCommand{\tu}{}{2 \pi \ii}

\NewDocumentCommand{\evec}{m}{\mathbf{e}_{#1}}

\NewDocumentCommand{\extpower}{O{}}{\mathop{\bigwedge\nolimits^{\!#1}}}

\NewDocumentCommand{\slg}{}{\operatorname{SL}_2(\mathbb{C})}
\NewDocumentCommand{\sla}{}{\mathfrak{sl}_2}
\NewDocumentCommand{\pslg}{}{\operatorname{PSL}_2(\mathbb{C})}
\NewDocumentCommand{\pjsp}{O{\mathbb{C}}}{\operatorname{#1 P}^{1}}

\NewDocumentCommand{\qn}{}{\triangleleft}

\NewDocumentCommand{\qgrplong}{O{\omega}}{\mathcal{U}_{#1}(\mathfrak{sl}_2)}
\NewDocumentCommand{\qgrp}{O{\omega}}{\mathcal{U}_{#1}}
\NewDocumentCommand{\weyl}{O{}}{\mathcal{W}_{#1}}

\NewDocumentCommand{\mer}{}{\mathbf{m}}
\NewDocumentCommand{\lon}{}{\mathbf{l}}
\NewDocumentCommand{\extr}{m}{M_{#1}}
\NewDocumentCommand{\comp}{m}{S^3 \setminus #1}

\NewDocumentCommand{\cohomol}{m m}{\operatorname{H}^{#1}\leftfun(#2\rightfun)}

\ExplSyntaxOn
\NewDocumentCommand{\reid}{m O{}}{
  \bool_case:nTF
  {
    {\str_if_eq_p:NN {#1} {1} } { \operatorname{R}\sb{1} }
    {\str_if_eq_p:NN {#1} {1f} } { \operatorname{R}\sb{1}^{\operatorname{fr}} }
    {\str_if_eq_p:NN {#1} {2} } { \operatorname{R}\sb{2}^{#2} }
    {\str_if_eq_p:NN {#1} {2a} } { \operatorname{R}\sb{2}^{\operatorname{a}} }
    {\str_if_eq_p:NN {#1} {2b} } { \operatorname{R}\sb{2}^{\operatorname{b}} }
    {\str_if_eq_p:NN {#1} {3} } { \operatorname{R}\sb{3} }
  }
  { } %
  { } %
}
\ExplSyntaxOff

\NewDocumentCommand{\modcat}{m}{#1\operatorname{-mod}}
\NewDocumentCommand{\tsunit}{}{\mathbf{1}}
\NewDocumentCommand{\delooping}{}{\mathsf{B}}

\NewDocumentCommand{\sclbrak}{m}{\left\langle #1 \right \rangle}

\NewDocumentCommand{\point}{}{\cdot}

\NewDocumentCommand{\qcat}{O{\nr}}{\mathsf{C}_{#1}}
\NewDocumentCommand{\qcatb}{O{\nr}}{\overline{\mathsf{C}}_{#1}}
\NewDocumentCommand{\tcat}{O{}}{\mathsf{T}_{#1}}
\NewDocumentCommand{\vect}{}{\mathsf{Vect}}

\NewDocumentCommand{\qpf}{O{\nr} m}{\Gamma_{#1}\ParenIfNonempty{#2}}

\NewDocumentCommand{\bra}{m}{\left\langle #1 \right|}
\NewDocumentCommand{\ket}{m}{\left| #1 \right\rangle}
\NewDocumentCommand{\braket}{m m}{\left\langle #1 \middle| #2 \right\rangle}

\NewDocumentCommand{\wb}{m}{v\ParenIfNonempty{#1}}

\NewDocumentCommand{\repvar}{m}{\mathsf{R}\leftfun(#1\rightfun)}
\NewDocumentCommand{\repvarlog}{m}{\widetilde{\mathsf{R}}\leftfun(#1\rightfun)}
\NewDocumentCommand{\admvar}{m}{\mathsf{A}\leftfun(#1\rightfun)}

\NewDocumentCommand{\qfunc}{O{\nr} m}{%
  \mathcal{Z}_{#1}^{\psi}\ParenIfNonempty{#2}
}

\NewDocumentCommand{\qfuncb}{O{\nr} m}{%
  \overline{\mathcal{Z}}^{\psi}_{#1}\ParenIfNonempty{#2}
}

\NewDocumentCommand{\qinv}{O{\nr} m}{%
  \mathrm{Z}_{#1}^{\psi}\ParenIfNonempty{#2}
}

\NewDocumentCommand{\qinvb}{O{\nr} m}{%
  \overline{\mathrm{Z}}_{#1}^{\psi}\ParenIfNonempty{#2}
}

\NewDocumentCommand{\kado}{O{\nr} m}{%
  \mathrm{F}_{#1}\ParenIfNonempty{#2}
}

\NewDocumentCommand{\csfunc}{O{} m}{%
  \mathcal{I}_{#1}^{\psi}\ParenIfNonempty{#2}
}

\NewDocumentCommand{\brkern}{m O{\nr} m m}{%
  \Omega^{#1}_{#2}%
  \IfEmptyTF{#3}{%
    \IfEmptyTF{#4}{}{\leftfun(#4\rightfun)}%
  }{%
    \leftfun(#3%
    \IfEmptyTF{#4}{%
      \rightfun)%
    }{%
      \middle|#4\rightfun)%
    }%
  }
}

\NewDocumentCommand{\octfunc}{m m m}{\brkern{#1}{#2}{#3}}

\NewDocumentCommand{\shadset}{}{\mathsf{U}}
\NewDocumentCommand{\matset}{}{\mathsf{Q}}

\NewDocumentCommand{\pf}{O{} m }{\varphi_{#1}\ParenIfNonempty{#2}}
\NewDocumentCommand{\pft}{O{} m }{\widetilde{\varphi}_{#1}\ParenIfNonempty{#2}}
\NewDocumentCommand{\nrdirac}{m}{\delta_{\nr}\leftfun(#1\rightfun)}

\NewDocumentCommand{\modb}{m O{\nr}}{\left\{#1\right\}_{#2}}
\NewDocumentCommand{\cutoff}{m O{\nr}}{\theta_{#2}\leftfun(#1\rightfun)}

\NewDocumentCommand{\Qfunc}{m}{Q_{\nr}\ParenIfNonempty{#1}}

\NewDocumentCommand{\df}{O{} m }{\operatorname{D}_{#1}\ParenIfNonempty{#2}}

\NewDocumentCommand{\pfl}{O{\nr} m}{\phi_{#1}\ParenIfNonempty{#2}}

\NewDocumentCommand{\qlfl}{O{\nr} m m}{%
  \lambda_{#1}\leftfun(#2
  \IfEmptyTF{#3}{%
    \rightfun)
  }{%
    \middle|#3\rightfun)%
  }
}

\NewDocumentCommand{\qlog}{m m}{w\leftfun(#1 \middle | #2 \rightfun)}
\NewDocumentCommand{\qp}{m O{\omega} m}{\left(#1 ; #2\right)_{#3}}

\NewDocumentCommand{\fstack}{m m m}{f\leftfun( \begin{matrix} #1 \\ #2 \end{matrix}\middle| #3 \rightfun)}

\NewDocumentCommand{\ldil}{m}{\lambda\ParenIfNonempty{#1}}

\newcommand{\evup}[1]{\operatorname{ev}^\uparrow_{#1}}
\newcommand{\coevup}[1]{\operatorname{coev}^\uparrow_{#1}}
\newcommand{\evdown}[1]{\operatorname{ev}^\downarrow_{#1}}
\newcommand{\coevdown}[1]{\operatorname{coev}^\downarrow_{#1}}

\NewDocumentCommand{\vb}{ m}{v_{#1}}

\NewDocumentCommand{\rmat}{o m m m m }{
  \IfNoValueTF{#1}{%
    \widehat{R}_{#2 #3}^{#4 #5}
  }{%
    \widehat{R}(#1)_{#2 #3}^{#4 #5}
  }
}

\NewDocumentCommand{\rmatm}{o m m m m }{
  \IfNoValueTF{#1}{%
    \widehat{\overline{R}}_{#2 #3}^{#4 #5}
  }{%
    \widehat{\overline{R}}(#1)_{#2 #3}^{#4 #5}
  }
}

\NewDocumentCommand{\lN}{}{\mathrm{N}}
\NewDocumentCommand{\lW}{}{\mathrm{W}}
\NewDocumentCommand{\lS}{}{\mathrm{S}}
\NewDocumentCommand{\lE}{}{\mathrm{E}}
\NewDocumentCommand{\up}{}{\uparrow}
\NewDocumentCommand{\dn}{}{\downarrow}

\newcommand{\br}{to [out=00,in=180]}

\ExplSyntaxOn
\NewDocumentCommand{\tikzBraidingEast}{m m m m m}{
  \bool_case:nTF
  {
    {\str_if_eq_p:NN {#1} {+} } { 
      \draw[->] #3 \br #5;
      \draw[white, line~width=10, line~cap=butt] #2 \br #4;
      \draw[->] #2 \br #4;
    }
    {\str_if_eq_p:NN {#1} {-} } {
      \draw[->] #2 \br #4;
      \draw[white, line~width=10, line~cap=butt] #3 \br #5;
      \draw[->] #3 \br #5;
    }
  }
  { } %
  { } %
}
\NewDocumentCommand{\tikzBraiding}{m m m m m}{
  \bool_case:nTF
  {
    {\str_if_eq_p:NN {#1} {+} } { 
      \draw[-] #3 \br #5;
      \draw[white, line~width=10, line~cap=butt] #2 \br #4;
      \draw[-] #2 \br #4;
    }
    {\str_if_eq_p:NN {#1} {-} } {
      \draw[-] #2 \br #4;
      \draw[white, line~width=10, line~cap=butt] #3 \br #5;
      \draw[-] #3 \br #5;
    }
  }
  { } %
  { } %
}
\ExplSyntaxOff

%% file: fig/octahedron-side.pdf_tex
\begingroup%
  \makeatletter%
  \providecommand\color[2][]{%
    \errmessage{(Inkscape) Color is used for the text in Inkscape, but the package 'color.sty' is not loaded}%
    \renewcommand\color[2][]{}%
  }%
  \providecommand\transparent[1]{%
    \errmessage{(Inkscape) Transparency is used (non-zero) for the text in Inkscape, but the package 'transparent.sty' is not loaded}%
    \renewcommand\transparent[1]{}%
  }%
  \providecommand\rotatebox[2]{#2}%
  \newcommand*\fsize{\dimexpr\f@size pt\relax}%
  \newcommand*\lineheight[1]{\fontsize{\fsize}{#1\fsize}\selectfont}%
  \ifx\svgwidth\undefined%
    \setlength{\unitlength}{144bp}%
    \ifx\svgscale\undefined%
      \relax%
    \else%
      \setlength{\unitlength}{\unitlength * \real{\svgscale}}%
    \fi%
  \else%
    \setlength{\unitlength}{\svgwidth}%
  \fi%
  \global\let\svgwidth\undefined%
  \global\let\svgscale\undefined%
  \makeatother%
  \begin{picture}(1,1)%
    \lineheight{1}%
    \setlength\tabcolsep{0pt}%
    \put(0,0){\includegraphics[width=\unitlength,page=1]{octahedron-side.pdf}}%
    \put(0.05208333,0.91666667){\makebox(0,0)[lt]{\lineheight{1.25}\smash{\begin{tabular}[t]{l}$1$\end{tabular}}}}%
    \put(0.89461762,0.70059374){\makebox(0,0)[lt]{\lineheight{1.25}\smash{\begin{tabular}[t]{l}$1'$\end{tabular}}}}%
    \put(0.88020835,0.27083331){\makebox(0,0)[lt]{\lineheight{1.25}\smash{\begin{tabular}[t]{l}$2'$\end{tabular}}}}%
    \put(0.02604167,0.04166665){\makebox(0,0)[lt]{\lineheight{1.25}\smash{\begin{tabular}[t]{l}$2$\end{tabular}}}}%
    \put(0.05185756,0.66282773){\makebox(0,0)[lt]{\lineheight{1.25}\smash{\begin{tabular}[t]{l}$P_-$\end{tabular}}}}%
    \put(0.78125,0.375){\makebox(0,0)[lt]{\lineheight{1.25}\smash{\begin{tabular}[t]{l}$P_-'$\end{tabular}}}}%
    \put(0.05100026,0.37277058){\makebox(0,0)[lt]{\lineheight{1.25}\smash{\begin{tabular}[t]{l}$P_+$\end{tabular}}}}%
    \put(0.78125458,0.65625012){\makebox(0,0)[lt]{\lineheight{1.25}\smash{\begin{tabular}[t]{l}$P_+'$\end{tabular}}}}%
    \put(0,0){\includegraphics[width=\unitlength,page=2]{octahedron-side.pdf}}%
    \put(0.46875,0.86979167){\makebox(0,0)[lt]{\lineheight{1.25}\smash{\begin{tabular}[t]{l}$P_1$\end{tabular}}}}%
    \put(0.46875,0.10416663){\makebox(0,0)[lt]{\lineheight{1.25}\smash{\begin{tabular}[t]{l}$P_2$\end{tabular}}}}%
    \put(0,0){\includegraphics[width=\unitlength,page=3]{octahedron-side.pdf}}%
  \end{picture}%
\endgroup%